\documentclass[11pt,a4paper]{article}

\usepackage[english]{babel}
\usepackage{amsmath,amssymb,amsthm}
\usepackage{graphicx}
\usepackage[margin=2cm]{geometry}
\usepackage[nooneline,font=small,labelsep=quad]{caption}
\usepackage{lmodern}
\usepackage{subfigure}
\usepackage{hyperref}

\newtheorem{definition}{Definition}[section]
\newtheorem{lemma}[definition]{Lemma}

\newtheorem{theorem}[definition]{Theorem}

\newtheorem{corollary}[definition]{Corollary}
\newtheorem*{remark}{Remark}
\newtheorem*{notation}{Notation}

\def\picinput#1{\includegraphics{#1.pdf}}

\def\sc{v}
\def\ssc{u}
\def\scw{w}

\def\wedge{w}
\def\vpi{{\underline{\pi}}}
\def\vsigma{{\underline{\sigma}}}
\def\vlambda{{\underline{\lambda}}}
\def\vtau{{\underline{\tau}}}
\def\vwedges{\underline{w}}

\def\word{W}
\def\letter{A}
\def\lperm{\pi_\ell}
\def\rperm{\pi_r}
\def\Sing{\Sigma}

\def\Quad{\Gamma}
\newcommand*{\xQuad}[1][]{\Gamma_{#1}}
\newcommand*{\lQuad}[1][]{\Gamma_{#1}^\ell}
\newcommand*{\rQuad}[1][]{\Gamma_{#1}^r}

\def\NN{\mathbb{N}}
\def\ZZ{\mathbb{Z}}

\def\RR{\mathbb{R}}
\def\CC{\mathbb{C}}
\def\PP{\mathbb{P}}

\def\TT{\mathbb{T}}

\def\CCC{\mathcal{C}}
\def\HHH{\mathcal{H}}
\def\QQQ{\mathcal{Q}}
\def\GGG{\mathcal{G}}
\def\LLL{\mathcal{L}}
\def\MMM{\mathcal{M}}

\def\AAA{\mathcal{A}}
\def\LLL{\mathcal{L}}
\def\III{\mathcal{I}}

\def\Re{\operatorname{Re}}
\def\Im{\operatorname{Im}}
\def\SL{\operatorname{SL}}

\def\hol{\operatorname{hol}}
\def\dev{\operatorname{dev}}

\def\GL{\operatorname{GL}}
\def\area{\operatorname{Area}}
\def\sys{\operatorname{sys}}

\def\phi{\varphi}
\def\epsilon{\varepsilon}

\title{Diagonal changes for surfaces in hyperelliptic components\\{\normalsize A geometric natural extension of Ferenczi-Zamboni moves}}
\author{Vincent Delecroix and Corinna Ulcigrai}
\date{2013}

\begin{document}

\maketitle

\abstract{We describe geometric algorithms that generalize the classical continued fraction algorithm for the torus to all translation surfaces in hyperelliptic components of translation surfaces. We show that these algorithms produce all saddle connections which are best approximations in a geometric sense, which generalizes the notion of best approximation for the classical continued fraction. In addition, they allow to list all systoles along a Teichmueller geodesic and all bispecial words which appear in the symbolic coding of linear flows. The elementary moves of the described algorithms provide a  geometric invertible extension of the renormalization moves introduced by S.~Ferenczi and L.~Zamboni for the corresponding interval exchange transformations.}

\tableofcontents

\section{Introduction}
We begin this introduction by describing in \S\ref{subsec:torus_induction} a geometric version of the standard (additive) continued fraction algorithm, in terms of changes of bases for lattices. One of the key properties of the continued fraction algorithm is that it generates all rational best approximations of an irrational number. This property has a geometric interpretation: the continued fraction algorithm produces all saddle connections which are geometric best approximations (see Definition~\ref{def:BAtorus}).

In this paper we define diagonal changes algorithms which provide geometric generalizations of the continued fraction algorithm for linear flows on translation surfaces of higher genera (tori are translation surfaces of genus 1). Basic definitions appear in~\S\ref{subsubsec:intro_quadrangulations} and the algorithm is described in~\S\ref{subsubsec:intro_algorithm}.
 
The diagonal changes algorithms have several nice properties which are described in \S\ref{subsec:properties} of this introduction: they produce all geometric best approximations (see~\S\ref{subsubsec:intro_Diophantine}), allow to construct all bispecial words in the symbolic coding of linear flows (see~\S\ref{subsubsec:intro_language}) and detect all systoles along a Teichm\"uller geodesic (see~\S\ref{subsubsec:intro_Teich}).

\subsection{Geometric continued fraction algorithm for the torus} \label{subsec:torus_induction}
Let $\Lambda \subset \CC$ be a lattice. The standard continued fraction algorithm provides a way to construct a sequence of vectors in $\Lambda$ that are good approximation of the vertical direction. Let us present a geometric version of this algorithm. We choose a basis $(\scw_\ell, \scw_r)$ of $\Lambda$ such that:
\begin{itemize}
\item $\Re(\scw_\ell) < 0$ and $\Re(\scw_r) > 0$,
\item $\Im(\scw_\ell) > 0$ and $\Im(\scw_r) > 0$.
\end{itemize}
It is clear that such a basis exists if $\Lambda$ does not contain vertical or horizontal non-zero vectors. The basis $(\scw_\ell, \scw_r)$ forms a \emph{wedge} that contains the vertical direction; in other words, the vertical is contained in the positive cone generated by this basis. The parallelogram $Q = Q(\scw_\ell, \scw_r)$ formed from these two vectors is a fundamental domain for the action of $\Lambda$ on $\CC$. 
We say that the parallelogram $Q$ is \emph{left-slanted} (respectively \emph{right-slanted}) if the vertical half-axis $\{ z;\, \Re(z)=0\ \text{and}\ \Im(z)>0\}$ crosses the left (resp.~right) top side, that is the side parallel to $\scw_r$ (resp.~$\scw_\ell$). An example is shown in figure~\ref{fig:slanted_torii}.
\begin{figure}[!ht]
\begin{center}
\subfigure[left slanted\label{fig:lslanted_par}]{\picinput{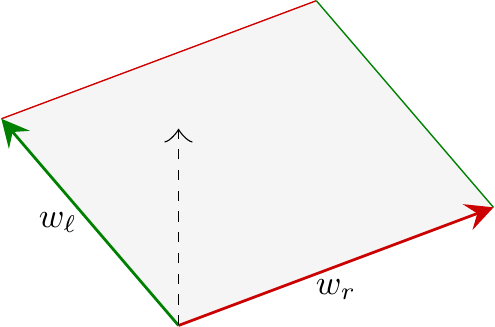}}  \hspace{3cm}
\subfigure[right slanted\label{fig:rslanted_par}]{\picinput{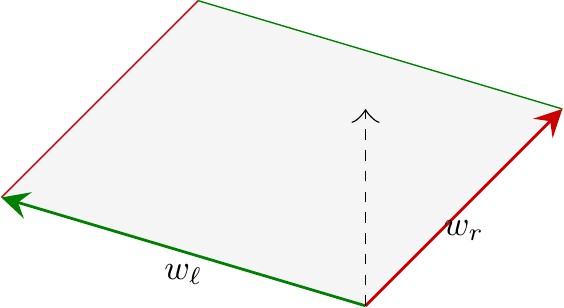}} 
\end{center}
\caption{examples of left and right slanted parallelograms}
\label{fig:slanted_torii}
\end{figure}

One step of the algorithm is as follows. If the parallelogram $Q$ defined by the basis $(\scw_\ell, \scw_r)$ is left slanted, consider the new basis $\scw'_\ell = \scw_\ell$ and $\scw'_r = \scw_d = \scw_r + \scw_\ell$. Geometrically, the new parallelogram $Q'$ with sides $(\scw'_\ell,\scw'_r)$ is obtained by cutting the old one along a diagonal and pasting the lower triangle as in Figure~\ref{fig:left_cut_paste}. Remark that, after this operation, the vertical axis is contained in the parallelogram $Q'$. We call such move a \emph{left move}. If the parallelogram is right slanted, then we made a \emph{right move} in a symmetric way (see Figure~Figure~\ref{fig:right_cut_paste}).

\begin{figure}[!ht]
\begin{center}
\subfigure[a left move\label{fig:left_cut_paste}]{\picinput{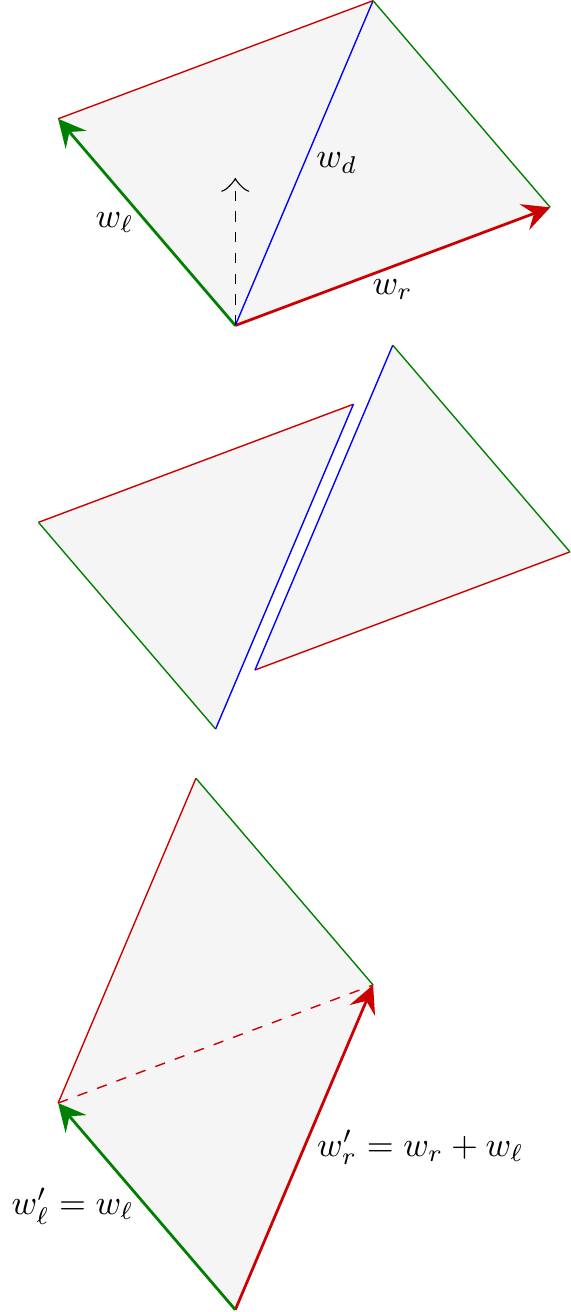}} \hspace{2cm} 
\subfigure[a right move\label{fig:right_cut_paste}]{\picinput{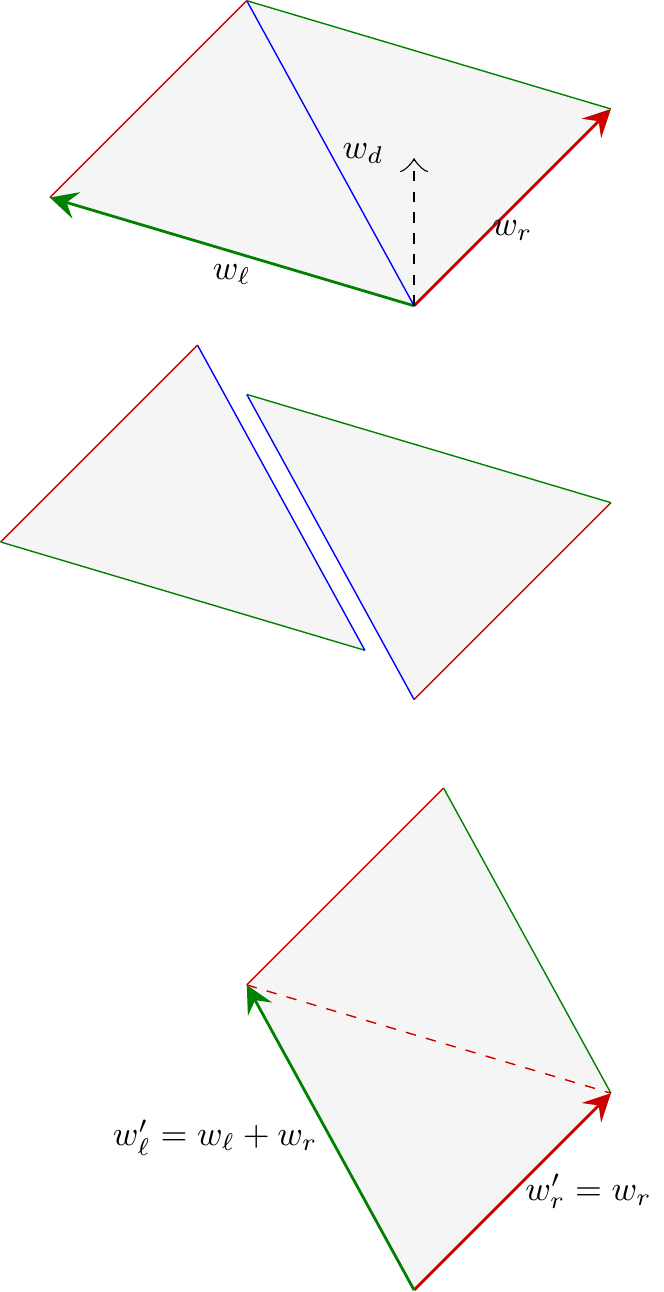}}   
\caption{a left move and a right move for the two examples of Figure~\ref{fig:slanted_torii}\label{fig:torus_cut_and_paste}}
\end{center}
\end{figure}

Let us set $\scw^{(0)}_\ell = \scw_\ell$ and $\scw^{(0)}_r = \scw_r$. Applying successively the above step we get a sequence of bases $(\scw^{(n)}_\ell, \scw^{(n)}_r)$ of $\Lambda$ for which the imaginary parts of both vectors in the base tend to infinity. Notice that the algorithm may stop after a finite number of steps, but this is the case if and only if  the lattice $\Lambda$ contains a vertical vector.
Let us also remark that one can also define a cut and paste operation which is the inverse operation to the diagonal change defined above. Thus, the algorithm can also be defined in backward time. The backward orbit is infinite if and only if $\Lambda$ does not contain horizontal vectors. 
In the sequel, we assume that $\Lambda$ does neither contain vertical nor horizontal vectors.

Let us recall some well known Diophantine approximation properties of this sequence of bases. Let $\Quad$ be the set of primitive vectors of $\Lambda$ with positive imaginary part. One can decompose $\Quad$ as  union of $\Quad_\ell$ and $\Quad_r$ which denote respectively the primitive vectors with positive and negative real part. Remark that, for any $n \in \NN$, $\scw^{(n)}_\ell$ belongs to $\Quad_\ell$ and $\scw^{(n)}_r$ belongs to $\Quad_r$.
\begin{definition}\label{def:BAtorus}
A vector $\sc \in \Quad_r$ is a \emph{(right) geometric best approximation} if 
\[
\forall \ssc \in \Quad_r, \quad \Im(\ssc) < \Im(\sc) \Rightarrow  |\Re(\ssc)| > |\Re(\sc)| .
\]
The definition  of  \emph{left geometric best approximations} is obtained by replacing $\Quad_r$ by $\Quad_\ell$.
\end{definition}

\begin{remark}
In geometric terms, $\sc$ is a right best approximation if and only if the rectangle $R(\sc):=\left[0,\Re(\sc)\right] \times \left[0, \Im(\sc)\right]$ does not contains any vector of $\Lambda$ in its interior.
\end{remark}
The geometric continued fraction algorithm \emph{constructs} all geometric best approximations in the following sense:
\begin{theorem}\label{thm:torus_best_approx}
Let $\Lambda$ be a lattice in $\CC$ that does not contain neither horizontal nor vertical vectors. Then the sequence of bases $(\scw^{(n)}_\ell, \scw^{(n)}_r)$ built from the algorithm is uniquely defined up to a shift in the numbering. Moreover, the vectors $\scw^{(n)}_\ell$ and $\scw^{(n)}_r$ are exactly the geometric best approximations.
\end{theorem}
The above theorem can be interpreted and proved in terms of Diophantine approximation: intermediate convergents of a real number $\alpha$ are exactly the approximation of the first kind (see~\cite[thm~15 p.~22]{Khinchin64}). We will prove this statement in much more generality in Theorem~\ref{thm:wedges_are_best_approx}.

\vspace{0.5cm}

The quotient $\TT_\Lambda = \CC / \Lambda$ is a flat torus on which the origin is marked. On $\TT_\Lambda$ there is a family of linear flows, which are the quotients of the straight line flows $\phi^\theta_t: z \mapsto z + t e^{\sqrt{-1}\, \theta}$ where $\theta$ is a fixed element in the circle $S^1 = \RR / (2 \pi) \ZZ$. A \emph{saddle connection} is a trajectory of a linear flow from the marked point to itself. There is a one to one correspondence between saddle connections and primitive vectors of $\Lambda$. The algorithm hence produces saddle connections which give better and better approximation of the vertical linear flow.

In \S\ref{subsubsec:intro_Teich} we recall the well-known connection of the continued fraction algorithm with the geodesic flow on the modular surface and explain that the geometric continued fraction algorithm also detects \emph{systoles} for the geodesic flow. 

\subsection{Diagonal changes algorithms for translation surfaces} \label{subsec:intro_def}
We start this section by defining translation surfaces, which are generalizations of flat tori. We then introduce the notion of wedges and their associated quadrangulations. Using them, we define algorithms which consist of diagonal changes and provide a generalization for translation surfaces of the continued fraction.
 
\subsubsection{Translation surfaces, wedges and quadrangulations.}\label{subsubsec:intro_quadrangulations}
A translation surface can be defined by gluing polygons in the following way. Let $(P_i)_i$ be a finite collection of polygons in the plane $\CC$, with a pairing of edges such that for each edge $e$ of a polygon $P_i$ there is an edge $\sigma(e)$ of a polygon $P_j$ such that $e$ and $\sigma(e)$ are parallel, of the same length and have opposite outgoing normal vector (with respect to their polygon). Let us identify each edge $e$ with the corresponding edge $\sigma(e)$ by the unique translation that sends $e$ to $\sigma(e)$. The quotient $X$ of $\sqcup P_i$ under those identification is called a \emph{translation surface}. We will always assume that a translation surface is connected. Flat tori (see~\S\ref{subsec:torus_induction}) are examples of translation surfaces built from one parallelogram and see Figure~\ref{fig:g2example} for a translation surface built from 3 quadrilaterals.

Let $\Sing = \Sing(X)$ be the finite subset of points of $X$ which are images of vertices of the polygons $P_i$ in $X$. Such points are called \emph{singularities} of $X$. The surface $X$ carries a flat (Euclidean) metric on $X \backslash \Sing$ induced by the Euclidean metric on the plane, with conical singularities at the points in $\Sing$ with cone angles of the form  $2\pi k$ with $k \in \NN$. A \emph{cone-point} with cone angle $2\pi k$  has a neighborhood which is isometric to a finite $k$-sheeted cover of the plane branched at the origin, which can be parametrized by polar coordinates $(\rho, \theta)$ where $\rho \in \RR^+$ and $\theta \in \RR/(2 \pi k) \ZZ$. 

The surface $X$ also inherits a \emph{translation structure} from $\CC$, which is an atlas on $X \backslash \Sing$ whose transition maps are translations. On $X \backslash \Sigma$ there is a well defined notion of (oriented) directions and hence one can define \emph{linear flows} which correspond to moving along lines in a given direction in $S^1$. The flow  $\phi^\theta_t$ in direction $\theta \in S^1$ is explicitly given in local charts by $\phi^\theta_t: z \mapsto z + t e^{\sqrt{-1}\, \theta}$.  Note that the flow is not well defined at $x \in (X \backslash \Sing)$ if its orbit $\phi^\theta_t(x)$ goes into a singularity.

A translation surface $X$ is in particular a Riemann surface endowed with a non-zero Abelian differential.
The complex structure is obtained from the translation charts and the differential form, generally denoted $\omega$, is obtained by lifting $dz$.
Conversly, a compact Riemann surface with a non-zero Abelian differential $\omega$ determines a translation surface (by finding local coordinates $z$ such that $\omega$ is locally $dz$). If $x \in X$ is a conical singularity of angle $2 \pi k$ then we can write locally $\omega$ around $x$ as $z^{k-1} dz$. For more details on the various definitions of translation surfaces, we refer to~\cite{Masur} and~\cite{Zorich-survey}.

\smallskip

We consider the following notion of isomorphism between translation surfaces. If the surface $S$ is defined from some polygons and identifications of their edges then we allow the two following operations. The \emph{cut operation} consists in cutting a polygon along a segment that joins two of its vertices and, in the new set of polygons, identify the two newly created sides. The \emph{paste operation} consists in gluing two polygons that were identified. 
Two surfaces $X$ and $X'$ defined respectively from $((P_i)_i,\sigma)$ and $((P'_j)_j,\sigma')$ are \emph{isomorphic} if there exists a sequence of cut and paste operations that goes from $((P_i)_i,\sigma)$ to $((P'_j)_k,\sigma')$ (where we consider that two polygonal representation are equal if we can pass from one to the other by translating the polygons).
The \emph{stratum} $\HHH(k_1-1,\ldots,k_n-1)$ of translation surfaces is the set of isomorphisms classes of translation surfaces with conical singularities with angles $2\pi k_1$, \ldots, $2\pi k_n$, or, equivalently, of non-zero Abelian differentials with zeros of order $k_1-1$, \dots, $k_n-1$. If there are $m_i$ singularities with total angle $\pi k_i$ we use the notation $\HHH((k_1-1)^{m_1},\ldots,(k_n-1)^{m_n} )$. 

An \emph{affine diffeomorphism} $\Psi: X \to X'$ between two translation surfaces, is an homeomorphism which maps $\Sing(X)$ to $\Sing(X')$ and is affine in the coordinate charts. Because of connectedness of $X \backslash \Sing(X)$, the linear part of the affine diffeomorphism is constant and may be identified to a matrix in $\GL(2,\RR)$. We call this matrix the \emph{derivative} of $\Psi$. Two translation surfaces $X$ and $Y$ are \emph{translation equivalent} if there exists an affine diffeomorphism $\Psi: X \to Y$ whose derivative is the identity matrix. It is easy to see that two translation surfaces $X$ and $Y$ are translation equivalent if and only if $Y$ is obtained from $X$ by a sequence of cut and paste operations.

\paragraph{Bundles of saddle connections.} 
Let $X$ be a translation surface with singularities $\Sigma$ and let  $\phi_t^\theta$, $\theta \in S^1$, be the family of linear flows on $X$. 
A \emph{saddle connection} in $X$ is the orbit of some linear flow that joins two singularities of $X$. Note that if $X$ is built from a union of polygons, any side $v$ of a polygon gives a saddle connection on $X$.

If in direction $\theta$ there is no saddle connection, then the flow $\phi_t^\theta$ is minimal (meaning that any infinite trajectory is dense in $X$). This result was first proven by M.~Keane in the context of interval exchange transformations~\cite{Keane75} and the corresponding condition for interval exchange transformations (orbits of discontinuity points are infinite and  distinct) is often called \emph{Keane's condition}. On the flat torus $\CC / (\ZZ \oplus \ZZ \sqrt{-1})$ the directions of saddle connections are exactly the rational ones (ie the angles $\theta \in S^1$ for which the slope $\tan(\theta)$ is a rational number). For a general translation surface the set of directions for which there exists a saddle connection is countable but has no particular algebraic structure.

The \emph{displacement vector} (sometimes called \emph{holonomy vector}) associated to an oriented saddle connection is the vector in $\CC$ which gives the displacement between the initial and final point seen as an element of $\CC$. More precisely, a saddle connection is a set of points $(\phi^\theta_t(x))_{t \in I}$ for some point $x \in X \backslash \Sigma$ and some interval $I = [a,b] \subset \RR$, its displacement vector is $(b-a) e^{\sqrt{-1}\theta}$. Given a side of a polygon $P_i$ that defines the surface, its sides are saddle connections and their displacement are simply the sides seen as complex vectors. The displacement can also be seen as the integral of the Abelian form $\omega$ along the saddle connection. It is well known that for any translation surface $X$ the set of displacement vectors of saddle connections on $X$ is a discrete subset of $\CC$, see for example~\cite{Vorobets} or~\cite{Masur}. 

We call \emph{natural orientation} of a saddle connection $\gamma$ the unique orientation of $\gamma$ such that its displacement vector has non-negative imaginary part. We say that a saddle connection \emph{starts} (respectively \emph{ends}) at a singularity if that singularity is the first endpoint (respectively last endpoint) of the saddle connection according to its natural orientation. A saddle connection is \emph{left slanted} (respectively \emph{right slanted}) if with its natural orientation its real part is negative (resp.~positive), as shown in  Figure~\ref{fig:lslanted_sc} (resp.~Figure~\ref{fig:rslanted_sc}).
 
\begin{figure}[!ht]
\begin{center}
\subfigure[left slanted \label{fig:lslanted_sc}]{\picinput{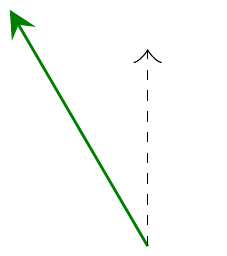}}  \hspace{.5cm} 
\subfigure[right slanted  \label{fig:rslanted_sc} ]{\picinput{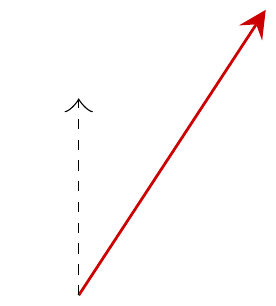}}  \hspace{2cm} 
\subfigure[a wedge  \label{fig:wedge}]{\picinput{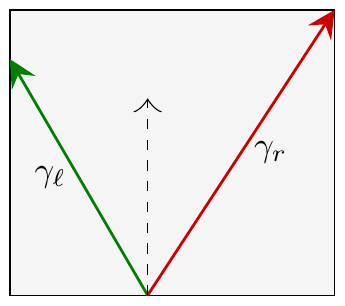}}  
\caption{left and right slanted saddle connections and a wedge}
\label{fig:wedges}
\end{center}
\end{figure}

Let $\xQuad = \xQuad(X)$ denote the set of all saddle connections on a given translation surface $X$ and let $\lQuad$ (respectively  $\rQuad$) the subset of all left-slanted (respectively right-slanted) saddle connections.
Saddle connections in $\xQuad$ can be subdivided as follows into subsets, which (following the notation introduced by L.~Marchese in \cite{Marchese}) we will call \emph{bundles of saddle connections}.
Assume that the singularity set $\Sigma$ consist of $n$ singularities of cone-angles  $2 \pi k_1, \dots 2\pi k_n$. Remark that, if the conical angle at $p_i \in \Sigma$ is $2\pi k_i$, from $p_i$ there are $k_i$ outgoing trajectories of the vertical linear flow and  $k_i$ outgoing trajectories of the horizontal linear flow (since $p_i$ has a neighborhood isomorphic to $k_i$ planes).  
For each $p_i \in \Sigma$, choose a reference horizontal ray $v_i$ starting from $p_i$.
For any two linear trajectories $\gamma, \gamma '$ starting at $p_i$ we denote by $\angle(\gamma,\gamma') \in [0, 2\pi k_i)$ the angle between them. Each saddle connection $\gamma$ starting at $p_i$ belongs to one of the $k_i$ \emph{outgoing half planes}, that is 
the angle $\angle(\gamma, v_i)$ with respect to the chosen horizontal $v_i$ from $p_i$ satisfies 
\begin{equation*}
  2 \pi j \leq \angle(\gamma, v_i) <  2 \pi j + \pi, \quad \text{for a unique}\ 0 \leq j < k_i.
\end{equation*}
Two saddle connections belong to the same \emph{bundle} if and only if they start from the same singularity $p_i$ and \emph{belong to the same half-plane}.
Remark that there are $k$ bundles of saddle connections on $X$, where $k = k_1 + \dots + k_n$ is the total angle. We will label them with the integers $1$, \dots, $k$ and denote them by $\Quad_1$, \dots, $\Quad_k$.

\paragraph{Wedges.} 
In the case of the torus, the diagonal changes algorithm produces a sequence of bases of saddle connections which form a wedge and provide better and better approximations of the vertical. On a translation surface, the algorithms we consider will produce a sequence of collections of $k$ \emph{wedges} (defined below), one for each of the $k$ vertical rays in $X$ emanating from the singularities.

\begin{definition}[wedge]
A \emph{wedge} $\scw$ on a translation surface $X$ is a pair of saddle connections $\scw=(\scw_\ell,\scw_r)$ such that:
\begin{itemize}
\item[(i)] $\scw_\ell$ and $\scw_r$ start from the same conical singularity of $X$,
\item[(ii)] $\scw_\ell$ is left-slanted and $\scw_r$ is right-slanted,
\item[(iii)] $(\scw_\ell,\scw_r)$ consist of two edges of an embedded triangle in $S$.
\end{itemize}
\end{definition}
A picture of a wedge is shown in Figure~\ref{fig:wedge}. Remark that $(i)$ and $(iii)$ are equivalent to asking that the saddle connections $\scw_\ell$ and $\scw_r$ forming the wedge belong to the same bundle.  
Remark also that a wedge has the property that it contains a unique vertical trajectory, that is there is exactly one trajectory of the vertical flow which starts from the conical singularity shared by $\scw_\ell$ and $\scw_r$ and intersects the interior of the triangle with edges  $\scw_\ell$ and $\scw_r$. 

\paragraph{Quadrangulations.}
Let us now define special decompositions of $X$ into polygons that are quadrilaterals. A \emph{quadrilateral} $q$ in a flat surface $X$ is the image of an isometrically embedded quadrilateral in $\CC$ so that the vertices of $q$ are singularities of $X$ and there is no other singularities of $X$ in $q$.
 
\begin{definition}[admissible quadrilateral] \label{def:admissible}
A quadrilateral $q$ in $X$ is \emph{admissible} if there is exactly one trajectory of the vertical linear flow of $X$ starting from one of its vertices and exactly one ending in a vertex. Equivalently, it is admissible if left-slanted and right-slanted saddle connections alternate while we turn around the quadrilaterals.
\end{definition}
Examples of admissible and non-admissible quadrilaterals are given in Figure~\ref{fig:admissibles}.

\begin{figure}[!ht]
\begin{center}
\subfigure[admissible \label{fig:admissible} ]
{\picinput{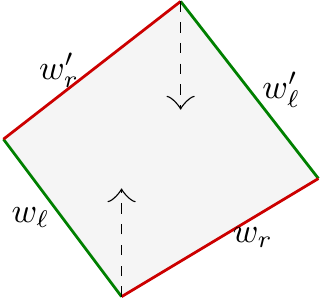}}  \hspace{16mm}
\subfigure[non admissible \label{fig:nonadmissible1} ]
{\picinput{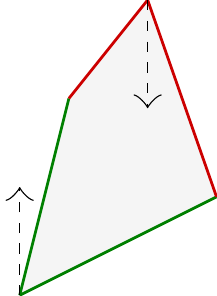}}  \hspace{16mm}
\subfigure[non admissible \label{fig:nonadmissible2}]
{\picinput{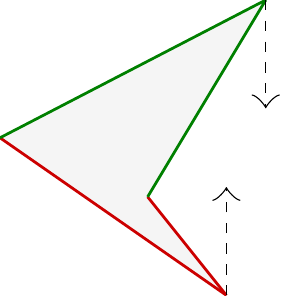}} 
\caption{examples of admissible and non-admissible quadrilaterals}
\label{fig:admissibles}
\end{center}
\end{figure} 

Let $q$ be an adimssible quadrilateral. We will refer to the saddle connections starting from the same singularity as the \emph{bottom sides} of the quadrilateral $q$ and to the ones ending in the same singularity as the \emph{top sides} of $q$. Furthermore, we will call~\emph{bottom right side} (resp.~\emph{bottom left side}) the right-slanted (resp.~left-slanted) bottom side of $q$ and~\emph{top right side} (resp.~\emph{top left side}) the left-slanted (resp.~right-slanted) top side of $q$. Remark that from the definition it follows that the \emph{bottom} sides of an admissible quadrilateral $q$ form a wedge. We will refer to it as \emph{the wedge of the quadrilateral} $q$.

\begin{definition}[quadrangulation]
 A \emph{quadrangulation} $Q$ 
 of $X$ is a decomposition of $X$ into a union of {admissible} quadrilaterals. 
\end{definition}
Given a quadrangulation $Q$, we write  $q \in Q$ if $q$ is a quadrilateral in the decomposition  and we call \emph{wedges of the quadrangulation} $Q$  the collection of wedges of all quadrilaterals in $Q$. An example of a quadrangulation is given in Figure~\ref{fig:g2example}: the quadrilaterals $q_1, q_2, q_3$ give a quadrangulation of a surface in genus $2$ with one $6\pi$ conical singularity.
 \begin{figure}[!ht]
\begin{center}
\picinput{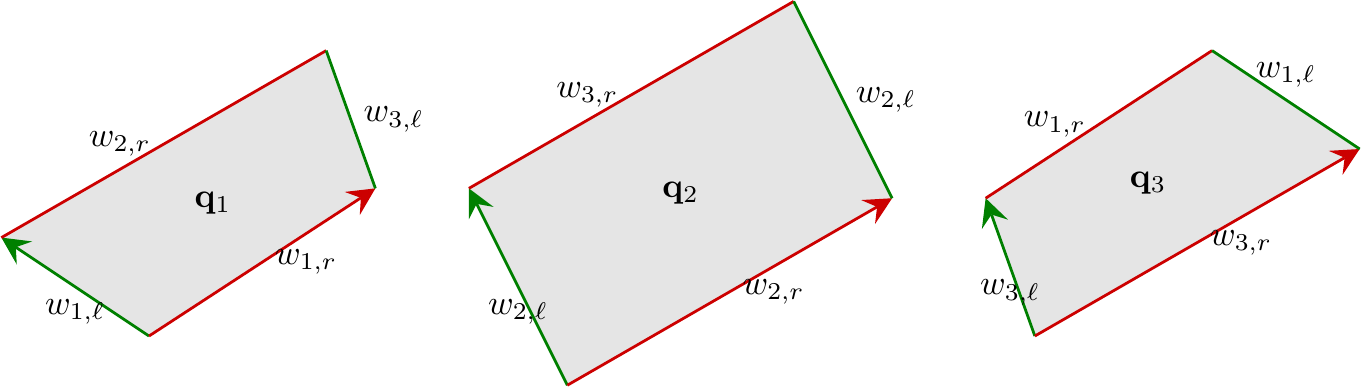}   
\caption{a quadrangulation of a surface in $\mathcal{H}(2)=\CCC^{hyp}(3)$}
\label{fig:g2example}
\end{center}
\end{figure}
  Let us stress that quadrilaterals in a quadrangulation are by definition admissible. 
 As each quadrilateral is glued to some other, each top side of a quadrilateral is also the bottom side of another quadrilateral, thus it belongs to a wedge. Hence, the  wedges of $Q$ on the surface $X$  completely determine the quadrangulation. In~\S\ref{sec:staircasedata} we will introduce a combinatorial datum given by a pair of permutations that describes how quadrilaterals are glued to each other. 

\subsubsection{Diagonal changes via staircase moves} \label{subsubsec:intro_algorithm}
Let $Q$ be a quadrangulation of a translation surface. A \emph{diagonal change} consists in replacing the left or right part of the wedge of a quadrilateral $q \in Q$ by the diagonal of the quadrilateral $q$. We consider elementary moves on the set of wedges (the \emph{staircase moves}) which, by performing simultaneous diagonal changes, produce a new set of wedges which correspond to a new quadrangulation $Q'$ of $X$. The moves of the geometric continued fraction algorithm in \S\ref{subsec:torus_induction} are a special case of staircase moves.

\paragraph{Staircases and staircase moves.}
 Let $Q$ be a quadrangulation of a translation surface $X$ and let $\scw = (\scw_{\ell},\scw_{r})$ be the wedge of a  quadrilateral $q \in Q$. We denote by $\scw_{d}$ the \emph{diagonal} saddle connection of $q$ which starts at the singularity of $\scw$ and ends at the top singularity of $q$.   

As in the case of the torus, we say that a quadrilateral $q$ is \emph{left-slanted} if the vertical issued from the bottom singularity crosses the top left side of $q$ and \emph{right-slanted} if it crosses the top right side (see Figure~\ref{fig:slanted_torii} for an illustration). Remark that the diagonal $\scw_{d}$ of $q$ form a wedge with $\scw_{\ell}$ (respectively with $\scw_{r}$) if and only if $q$ is left-slanted (respectively right-slanted) (see Figure \ref{fig:slanted_torii}). Therefore, for each quadrilateral we have the following alternatives:
\begin{itemize}
  \item if the quadrilateral $q$ is \emph{left-slanted},  either we keep the wedge $(\scw_{\ell},\scw_{r})$ or we do a \emph{left-diagonal change}, that is we replace it by $(\scw_{\ell},\scw_{d})$ (which in this case is again a wedge);
  \item if the quadrilateral $q$ is \emph{right-slanted}, either we keep the wedge $(\scw_{\ell},\scw_{r})$ or we do a \emph{right-diagonal change}, that is we replace it by the $(\scw_{d},\scw_{r})$ (which in this case also is a wedge); 
\end{itemize}

The key geometrical object which allow to perform diagonal changes consistently and hence define elementary moves are
\emph{staircases}:    
\begin{definition}[staircase] \label{def:cstaircase} 
Given a quadrangulation $Q $ of $X$, a \emph{left staircase} $S$ \emph{for}  $Q$ (respectively a \emph{right staircase} $S$ \emph{for} $Q$) is a subset $S \subset X$ which is the union of quadrilaterals  
$q_1, \dots, q_n$ of $Q$ that are cyclically glued so that the top left (resp.~top right) side of $q_i$ is identified with the bottom right (resp.~bottom left) side of $q_{i+1}$ for $1\leq i < k$ and of $q_1$ for $i=n$. 

A left (respectively right) staircase $S$ is \emph{well slanted} if all its quadrilaterals are left (resp.~right) slanted.
\end{definition}
An example of a right-staircase (which explain the choice of the name \emph{staircase}) is given in Figure~\ref{fig:rslanted_staircase1}: remark that the two sides labeled by $\scw_{1,\ell}$ are identified, so that the staircase is the union of $3$ quadrilaterals. An example of a well slanted staircase is the right staircase in Figure~\ref{fig:rslanted_staircase1} (all three quadrilaterals all right slanted), while the staircase in Figure~\ref{fig:rslanted_staircase2} is not well slanted (it is a right-staircase in which $q_1$ and $q_3$ are right slanted but $q_2$ is left slanted).

We remark that a left staircase (respectively right staircase) $S$ in $X$ is a topological cylinder whose boundary consists of a union of saddle connections which are all left slanted (resp.~all right slanted). 
Remark also that a staircase $S$ for $Q$ has a natural decomposition as union of admissible quadrilaterals induced by the quadrangulation $Q$ of $S$. 

\begin{figure}[!ht]
\begin{center}
\subfigure[a well slanted right staircase \label{fig:rslanted_staircase1}]{\picinput{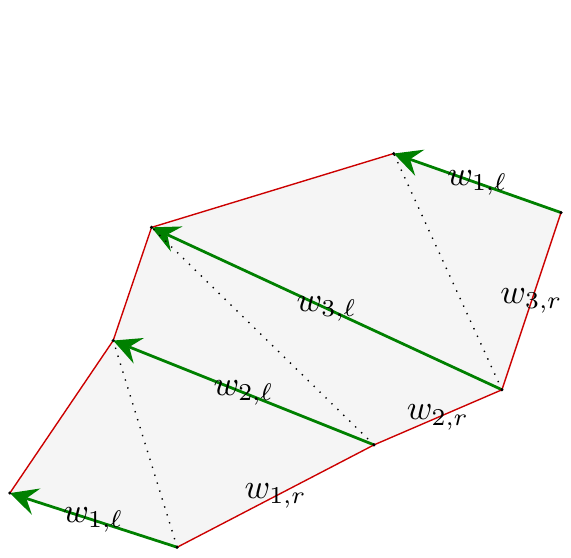}}  \hspace{1.6cm} 
\subfigure[diagonal changes in the same staircase \label{fig:rslanted_staircase2}]{\picinput{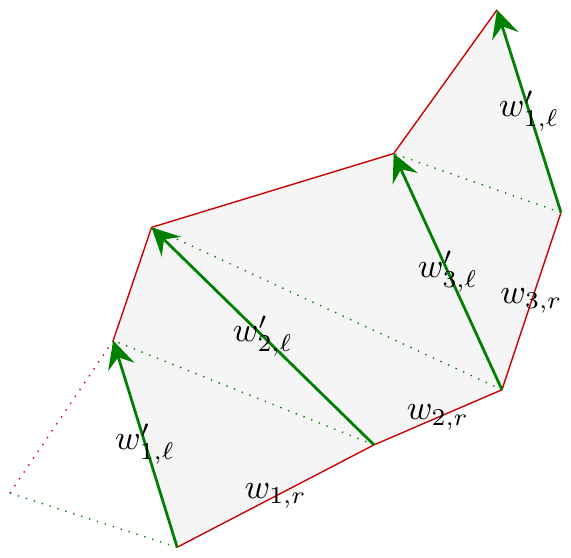}}    
\label{fig:rslanted_staircase}
\caption{diagonal changes in a right staircase}
\end{center}
\end{figure}

\begin{definition}[staircase move]\label{def:staircasemove}
  Given a quadrangulation $Q$ and a well slanted left-staircase $S$ (respectively a well slanted right staircase $S$), the \emph{staircase move} in $S$ is the operation which consists in doing simultaneously left (resp.~right) diagonal changes in all the quadrilaterals of $S$.
\end{definition}
Remark that given a quadrangulation there may be none or several well slanted staircases. In the first case no staircase  move is possible while in the latter there is a choice of staircase moves.

The importance of staircases lies in the following elementary result (see Lemma \ref{lem:staircase_move_data}): 
if $Q$ is a quadrangulation of a surface $X$ and $S$ be a well slanted staircase in $Q$, the staircase move in $S$ produces a new quadrangulation $Q'$ of $X$.  Furthermore, one can show that staircase moves are the minimal possible ways to combine individual diagonal changes consistently in order to keep a quadrangulation (see Lemma~\ref{lem:stable}).

\paragraph{Diagonal changes algorithms for surfaces in hyperelliptic strata.}
We prove the existence of quadrangulations and diagonal changes given by staircase moves for a class of translation surfaces which belong to the so called \emph{hyperelliptic components of strata}. Here below we provide an introduction to hyperelliptic components, but we refer to~\S\ref{sec:hyperellipticdef} for more details. 

An affine automorphism $s: X \to X$ of a translation surface $X$ is an \emph{hyperelliptic involution} if it is an involution, that is $s^2$ is the identity, and the quotient of $X \backslash \Sigma(X)$ by  $s$ is a (punctured) sphere. An example of a surface which admits an hyperelliptic involution is given in Figure~\ref{fig:g2example}. The surface is obtained from three quadrilaterals, one which is fixed by the involution (the quadrilateral $q_2$) and the  other two which are  exchanged ($q_1$ and $q_3$). On the picture, the hyperelliptic involution can be seen as a rotation by $180$ degrees.
One can show that if a translation surface admits an hyperelliptic involution, then this involution is unique.

Strata of translation surfaces are generally not connected and their connected components were classified by M.~Kontsevich and A.~Zorich~\cite{KontsevichZorich03}. Hyperelliptic components are the connected components of strata with the property that each surface in them admits an hyperelliptic involution. From the Kontsevich-Zorich classification, it follows that in each stratum $\HHH(k_1,\ldots,k_n)$  there are either one, two or three connected components, some of which are hyperelliptic. For each integer $k \geq 1$ there is exactly one hyperelliptic component which contains surfaces with total conical angle $2\pi k$. We denote this component by $\CCC^{hyp}(k)$. If $k$ is odd, then $\CCC^{hyp}(k) \subset \HHH(k-1)$ while if $k$ is even $\CCC^{hyp}(k) \subset \HHH(k/2-1, k/2-1)$ (see also Theorem~\ref{thm:KZhyperelliptic})
. For $1 \leq k \leq 4$ (that correspond to genus $1$ or $2$), the strata are connected and we have the following equalities: $\HHH(0) = \CCC^{hyp}(1)$ (this is the torus case), $\HHH(0,0) = \CCC^{hyp}(2)$, $\HHH(2) = \CCC^{hyp}(3)$ and $\HHH(1,1) = \CCC^{hyp}(4)$.

As in the genus $2$ example in Figure~\ref{fig:g2example} above, if $X$ belongs to a hyperelliptic component $\CCC^{hyp}(k)$ it turns out that all quadrilaterals in the quadrangulation are either parallelograms $q$, in which case $s(q) = q$, or come into pairs $q_i, q_j$  such that $q_i \neq q_j$ and $s(q_i)=q_j$ in which case $q_i$ and $q_j$ have parallel diagonals. This will be proved in Lemma~\ref{lem:hyperelliptic_staircase}.
 
\smallskip

Our main results for translation surfaces in hyperelliptic components are the following two theorems.
\begin{theorem} \label{thm:existence_of_quadrangulation}
Let $X$ be a surface in a hyperelliptic component $\CCC^{hyp}(k)$ that admits no horizontal and no vertical saddle connections.
Then $X$ admits a quadrangulation.
\end{theorem}
 
\begin{theorem} \label{thm:existence_of_staircase_move}
Let $Q$ be a quadrangulation of a surface $X$ in $\CCC^{hyp}(k)$ and assume that no quadrilateral in $Q$ has a vertical diagonal. Then, there exists at least one well slanted staircase in $Q$.
\end{theorem}
These two results allow us to define diagonal changes algorithms given by staircase moves in hyperelliptic components. 
Start from a quadrangulation $Q$ of $X \in \CCC^{hyp}(k)$, which exists by Theorem~\ref{thm:existence_of_quadrangulation}. 
 Theorem~\ref{thm:existence_of_staircase_move} implies that  there exists a staircase move for $Q$. Remark that there can be more than one well slanted staircase and hence several possible moves. Diagonal changes algorithms correspond to a systematic way of choosing which staircase moves to perform. In the torus case, where quadrangulations consist of only one quadrilateral (a parallelogram), there is no choice. In~\S\ref{subsec:algorithms} we give some examples of various diagonal changes algorithms. Nevertheless, we will show that the actual choice of an algorithm in some sense does not matter, since the sequence of wedges and well slanted staircases produced by \emph{any} sequence of staircase moves is the same (see Theorem~\ref{thm:wedges_are_best_approx} below).

In various works S.~Ferenczi and L.~Zamboni (see for example~\cite{FerencziZamboni-struct, FerencziZamboni-eig}) defined and studied an induction algorithm for interval exchange transformations with symmetric permutations, namely the permutations in $S_n$ defined by $i \mapsto n-i+1$ for $1 \leq i \leq n$. These interval exchange transformations may be obtained as first return maps of linear flows on sufaces in $\CCC^{hyp}(n-1)$. We call their induction the \emph{Ferenczi-Zamboni induction} (see also~\S\ref{subsec:comparisons}). Staircase moves provide a geometric invertible extension of the elementary moves in the Ferenczi-Zamboni induction, in a sense that is made precise in Section~\ref{sec:suspensions}. We note that Theorem~\ref{thm:existence_of_staircase_move} is originally proved in~\cite{FerencziZamboni-struct} in the context of interval exchange transformations. 

In view of these two results, a natural question would be to investigate other components of strata of translation surfaces.  We do not know if in general any translation surface admit a quadrangulation. Nevertheless, in~\S\ref{subsec:nonhyp_quadrangulation} we provide examples of quadrangulations of translation surface in which no staircase move is possible.
 
\subsection{Applications of diagonal changes algorithms}\label{subsec:properties}
In this section we summarize properties of diagonal changes algorithms and highlight some of its applications: they detect geometric best approximations (see~\S\ref{subsubsec:intro_Diophantine}), allow to produce bispecial factors for symbolic codings of linear flows (see~\S\ref{subsubsec:intro_language}) and may be used to construct the sequence of systoles along a Teichm\"uller geodesic (see~\S\ref{subsubsec:intro_Teich}).

\subsubsection{Geometric best approximations}\label{subsubsec:intro_Diophantine}
The notion of \emph{geometric best approximation} is a generalization for saddle connections on translation surfaces of the one for the torus (see Definition~\ref{def:BAtorus}). To define geometric best approximations for higher genera surfaces it is natural to compare only saddle connections which belong to the same bundle. Recall that if $X$ has conical singularities with cone angles $2\pi k_1, \dots, 2\pi k_n$, there are $k = k_1 + \dots + k_n$ bundles of saddle connections (see the beginning of~\S\ref{subsubsec:intro_quadrangulations} for the definition).  Let us label them and denote them by $\xQuad[1], \ldots, \xQuad[k]$.
  Each $\xQuad[i]$ can be decomposed as $\lQuad[i] \cup \rQuad[i]$ where $\lQuad[i]$ (respectively $\rQuad[i]$) consists of left-slanted (respectively right-slanted) saddle connections in $\xQuad[i]$.

We will adopt the following convention. Remark that 
given a saddle connection on $X$ we can associate to it a pair $(i,\sc)$ where $\sc \in \CC$ is its displacement (or holonomy) vector and $0\leq i < k $ is such that the saddle connection belongs to the bundle $\Quad_i$. Conversely, knowing the bundle to which the saddle connection belong and its displacement vector $\sc \in \CC$ completely determines the saddle connection. 
Thus, 
we can abuse the notation by identifying saddle connections with their displacement vector as long as the bundle is clear from the context.
\begin{notation}\label{complex_sc}
For a saddle connection $\sc$ on a translation surface, let $\Re(\sc)$, $\Im(\sc)$ and $|\sc|$ denote respectively the real part, the imaginary part and the absolute value of the displacement vector of $\sc$.

Given a bundle $\Quad_i$ of saddle connections,  we will denote by the complex number $\sc \in \CC$ the saddle connection in $\Quad_i$ that has $\sc$ as its displacement vector and we will hence write $\sc \in \Quad_i$. 
\end{notation}

\begin{definition}\label{def:BA}
A saddle connection $\sc \in \rQuad[i]$  is a \emph{right (geometric) best approximation} if 
\[
\forall \ssc \in \rQuad[i], \quad \Im  \ssc < \Im \sc  \Rightarrow |\Re \ssc | > |\Re \sc|.
\]
A similar definition for \emph{left (geometric) best approximation} is obtained by replacing $\rQuad[i]$ by $\lQuad[i]$.
\end{definition}
As for the torus, we can rephrase the definition in terms of singularity-free rectangles. 
Let us call an \emph{immersed rectangle} $R \subset X$ a subset without singularities in its interior which is obtained  by isometrically immersing in $X$ an Euclidean rectangle with horizontal and vertical sides in $\CC$ (recall that immersed means \emph{locally} injective opposed to embedded which means \emph{globally} injective). We remark that an immersed rectangle does not have to be embedded in $X$ and can have self-intersections. The following equivalent geometric characterization is proved at the beginning of  section \ref{sec:BAsystoles}. 
\begin{lemma}\label{lem:equivalentBA}
A saddle connection $\sc$ on $X$ is a geometric best approximation if and only if there exists an immersed rectangle $R(\sc)$ in $X$ which has $\sc$ as its diagonal.
\end{lemma}

One of the important properties of diagonal changes is that any sequence of staircase moves produces all geometric best approximations (see Theorem~\ref{thm:torus_best_approx} for the torus case). Let us recall that if $X$ is a surface in hyperelliptic component with neither horizontal nor vertical saddle connections, then by Theorem~\ref{thm:existence_of_quadrangulation} it admits quadrangulations and for each of them, in virtue of  Theorem \ref{thm:existence_of_staircase_move}, there is at least one staircase move. Furthermore, we will see that, starting from any quadrangulation $Q^{(0)}$, by self-duality of the algorithm one can define \emph{backwards moves} (see~\S\ref{sec:Markov}) and hence produce  a bi-infinite sequence $(Q^{(n)})_{ n \in \ZZ}$ of quadrangulations of $X$ obtained by a sequence of staircase moves. In Theorem~\ref{thm:wedges_are_best_approx2} we state and prove a more precise version of the following result.
\begin{theorem} \label{thm:wedges_are_best_approx}
Let $X$ be a surface in $\CCC^{hyp}(k)$ that has neither horizontal nor vertical saddle connections.
Let $( Q^{(n)})_{ n \in \ZZ}$ be \emph{any} sequence of quadrangulations of $X$ where $Q^{(n+1)}$ is obtained from $Q^{(n)}$ by a staircase move. Then the saddle connections belonging to the wedges of the quadrangulations $Q^{(n)}, n \in \ZZ$,  are exactly all geometric best approximations of $X$.
\end{theorem}

\subsubsection{Bispecial words in the language of cutting sequences}\label{subsubsec:intro_language}
Let $X$ be a translation surface such that the vertical flow on $X$ is minimal (for example without vertical saddle connections) and let $Q $ be be a quadrangulation of $X$. Let us denote by $q_1, \dots, q_k$ its quadrilaterals and let us label the saddle connections in $Q$ as follows. To the saddle connections $\scw_{i,\ell}$ and $\scw_{i,r}$ which form the wedge $\scw_i$ of the quadrilateral $q_i \in Q$ let us associate respectively the labels $(i,\ell)$ and $(i,r)$. Given an infinite orbit of the vertical flow in $X$, its \emph{cutting sequence} with respect to $Q$ is the infinite word on the alphabet $\AAA = \{1,\ldots,d\} \times \{\ell,r\}$ that corresponds to the sequence of names of saddle connections of $Q$ crossed by that orbit.    

It follows from minimality of the vertical flow on $X$ that each cutting sequence of an infinite orbit is made of the same pieces, in the sense that the set of finite words in $\AAA^*$ that appear in a cutting sequence does not depend on the cutting sequence but only on $X$. The set of finite words that appear in a cutting sequence (or all cutting sequences) is the \emph{language of $Q$} and is denoted $\LLL_Q$. Note that $\LLL_Q$ can also be defined in terms of symbolic coding of bipartite interval exchanges (see~\S\ref{sec:suspensions}). In the torus case, or equivalently interval exchanges of two intervals which are rotations of the circle, the coding is on a two letter alphabet $\{\ell,r\}$. The sequences that are obtain for the torus are called \emph{Sturmian words} and have several characterization (for example in terms of balance or complexity, see \cite{PF}). For higher genera cases, there is a characterization of such sequences in~\cite{BelovChernyatev} and~\cite{FerencziZamboni-language} based on bifurcations.

A word in $\LLL_Q$ is called \emph{left special} (resp.~\emph{right special}) if it may be extended in two ways on the left (resp.~on the right). It is \emph{bispecial} if it is left and right special. An important questions in symbolic dynamics is to describe the set of bispecial words in a language. The diagonal changes algorithm provides a full answer to this question.
\begin{theorem}\label{thm:best_approx_and_bispecials}
Let $X$ be a surface in $\CCC^{hyp}(k)$ without vertical saddle connections and let $Q$ be a quadrangulation of $X$.
Let $(Q^{(n)})_{n \in \NN}$ be \emph{any} sequence of quadrangulations obtained by a sequence of staircase moves starting from $Q$. Then, the set of bispecial words of the language $\LLL_Q$ is exactly the set of cutting sequences of diagonals of all quadrangulations in $(Q^{(n)})_{ n \in \NN}$.
\end{theorem}
Furthermore, cutting sequences of diagonals can be constructed recursively from moves of the algorithm in terms of substitutions, as explained in~\S\ref{sec:language} (see Theorem~\ref{thm:seqsubstitutions}). Thus, diagonal changes algorithms can be used to construct a list of bispecial words. We derive Theorem~\ref{thm:best_approx_and_bispecials} from Theorem~\ref{thm:wedges_are_best_approx}, since we show in~\S\ref{sec:induction} that in our context bispecial words correspond to geometric best approximations. A combinatorial proof of Theorem~\ref{thm:best_approx_and_bispecials} was first given  in~\cite{FerencziZamboni-struct} in the context of interval exchange transformations.

\subsubsection{Applications to Teichm\"uller dynamics}\label{subsubsec:intro_Teich}
In this section we mention other applications of diagonal changes algorithms in Teichm\"uller dynamics. Let us first recall the well-known connection between classical continued fractions and the geodesic flow on the modular surface (see for example~\cite{Series} and also~\cite{Arnoux94} for a more geometric approach in the same spirit as~\S\ref{subsec:torus_induction}).

\paragraph{Tori and the modular surface.}
The \emph{modular surface} is the quotient $\MMM_1= \mathbb{H}/\SL(2,\mathbb{Z})$ of the upper half plane $\mathbb{H} = \{ z;\ \Im z > 0\}$ by the action of $\SL(2,\ZZ)$ by Moebius transformations. Its unit tangent bundle $T^1 \mathcal{M}_1$ is isomorphic to $\SL(2,\RR)/\SL(2,\ZZ)$ (see for example \cite{BekkaMayer}). It is well-known that the \emph{space of unimodular lattices} is isomorphic to $\MMM_1$ and the space $\HHH^1(0)$ of tori of unit area is isomorphic to $\SL(2,\mathbb{R})/\SL(2,\ZZ)$. The correspondence is obtained by mapping the lattice $\Lambda \subset \mathbb{C}$, or equivalently the flat torus $\mathbb{T}^2_\Lambda$, to the point $\scw_2 / \scw_1 \in \mathbb{H}$, where $\scw_1$ and $\scw_r$ form a direct base of the lattice $\Lambda$ and are respectively the shortest and the second shortest saddle connections on $\mathbb{T}^2_\Lambda$.

The \emph{geodesic flow} $g_t$ on the unit tangent bundle of the modular surface $T^1\mathcal{M}_1 \cong \SL(2,\RR)/\SL(2,\ZZ)$ is given by the action of the $1$-parameter group of diagonal matrices
\begin{equation}\label{eq:gt}
\left\{ g_t= \left( \begin{array}{cc} e^t & 0 \\ 0& e^{-t} \end{array}\right);\ t \in \mathbb{R}\right \}
\end{equation}
by left multiplication on $\SL(2,\RR)$. 
Orbits of $g_t$ project to geodesics on $\MMM_1$ with respect to the hyperbolic metric. The continued fraction algorithm can be used to describe the \emph{Poincar{\'e} first return map} of the geodesic flow on a suitably chosen section of $T^1 \MMM_1 $ (this classical connection, known since Hedlund and Morse, was nicely pinpointed by Series in \cite{Series}).

Furthermore, the geometric continued fraction algorithm can be used to describe the set of vectors in $\Lambda$, or equivalently the set of saddle connections on $\TT^2_\Lambda$, which become short along a geodesic in the following sense. The \emph{systole function} is
$\sys (\Lambda) = \{ \min \vert v \vert;\ v \in \Lambda \backslash \{0\} \}$. Recall that compact sets in $\mathcal{M}_1$ can be characterized as sets on which the systole function is bounded (by Mahler's compactness criterion). 
Given a flat torus $\mathbb{T}^2_\Lambda \in T^1\mathcal{M}_1$, consider the systole function evaluated along the 
\emph{geodesic} passing though it, that is the map $t \mapsto \sys (g_t \Lambda)$.
 We say that a  vector $v \in {\Lambda}$ (or equivalently the corresponding saddle connection on $\mathbb{T}^2_\Lambda$)  \emph{realizes the systole at time $t$} if $|g_t v| = \sys(g_t \Lambda)$.  
Then the  vectors in $v \in {\Lambda}$ that realizes systoles for some $t\in \RR$ are exactly the vectors $\scw^{(n)}_\ell$ and $\scw^{(n)}_r$ in the sequence of bases $\left((\scw^{(n)}_\ell, \scw^{(n)}_r)\right)_{n \in \ZZ}$ built from the geometric continued fraction algorithm.

\paragraph{Systoles along Teichm\"uller geodesics.}
Diagonal changes algorithms play a role in describing short saddle connections along \emph{Teichm\"uller geodesics} analogous to the role played by the standard continued fraction for the torus.

Let $\HHH(k_1-1,\ldots,k_n-1)$ be a stratum of translation surfaces (as defined in~\S\ref{subsubsec:intro_quadrangulations}) and let $\HHH^1(k_1-1,\ldots,k_n-1) \subset \HHH(k_1-1,\ldots,k_n-1)$ consist of translation surfaces of area one. Seen as a topological space, $\HHH^1(k_1-1,\ldots, k_n-1)$ is never compact. Nevertheless, as in the case of tori, compact sets can be defined using the systole function
$$
\sys(X) = \min \{|\sc|; \ \sc \in \Gamma(X)\}
$$ 
where $X$ is a translation surface and as before $\Gamma(X)$ is the set of saddle connections on $X$.

The linear action of $\SL(2,\RR)$ on $\CC$ identified to $\RR^2$ induces an action of $\SL(2,\RR)$ on translation surfaces: given a translation surface $X$ obtained gluing polygons $P_i \subset \CC$ and $A \in \SL(2,\RR)$, the surface $A \cdot X$ is obtained gluing the polygons $A P_i$ using the same identifications. This is well defined since the linear action  preserves pairs of parallel congruent sides. The restriction of the $\SL(2,\RR)$-action on $\HHH^1(k_1-1,\ldots,k_n-1)$ to the diagonal subgroup $g_t$ in \eqref{eq:gt} is known as the \emph{Teichm\"uller geodesic flow}. 

Let $X$ be a translation surface and, as in the case of the torus, consider the systole function $t \mapsto \sys(g_t X)$. We say that a saddle connection $\sc$ on $X$ \emph{realizes the systole at time $t$} if $\sys(g_t X) = | g_t \sc |$. 
\begin{theorem}\label{thm:flat_systoles}
  Let $X$ be a surface in a hyperelliptic component of a stratum $\CCC^{hyp}(k)$ with no horizontal nor vertical saddle connections. Let $(Q^{(n)})_{n \in \ZZ}$ be a sequence of quadrangulations of the surface $X$ where $Q^{(n+1)}$ is obtained from $Q^{(n)}$ by a staircase move. Then, the set of saddle connections on $X$ which realize the systoles along the Teichm\"uller geodesic passing through $X$ is a subset of the sides the quadrangulations $Q^{(n)}_n$, $n \in \ZZ$. 
\end{theorem}
Theorem~\ref{thm:flat_systoles} is proved as a consequence of Theorem~\ref{thm:wedges_are_best_approx} in~\S\ref{sec:systolesproof}.

\paragraph{Pseudo-Anosov diffeomorphisms.}
We mention another application of diagonal changes algorithms which we prove in~\cite{DU:II}. An important problem in dynamics is to study the set of closed orbits of a flow. We show in~\cite{DU:II} that diagonal changes algorithms can be used to effectively produce a list, ordered by length, of all closed orbits of the Teichm\"uller flow in each hyperelliptic component  $\CCC^{hyp}(k)$. Since closed Teichm\"uller geodesics are in one to one correspondence with conjugacy classes of   \emph{pseudo-Anosov diffeomorphisms}, one can equivalently list pseudo-Anosov conjugacy classes, ordered by dilation. Furthermore, diagonal changes are much better suited for this problem than other algorithms such as Rauzy-Veech induction or train-track splittings, as explained in~\S\ref{subsec:comparisons}. 

\paragraph{Lagrange spectra.}
Recently, Lagrange spectra for translation surfaces, which are a generalization of the classical Lagrange spectrum in Diophantine approximation, were defined and studied by P.~Hubert, L.~Marchese and C.~Ulcigrai in~\cite{HubertMarcheseUlcigrai}.    
If $\CCC$ is a connected component of a stratum of translation surfaces, its Lagrange spectrum $\LLL(\CCC)$ is the set of values $\LLL(\CCC):=\{1/{a(X)};\ X \in \CCC \} \subset \RR \cup \{+\infty\}$, where 
\begin{equation}\label{def:aX}
a(X):=\liminf_{|\Im(v)|\to\infty}
\frac{\{ |\Im(v)||\Re(v)|; \ \text{v saddle connection on $X$} \}}{\area(X)},
\end{equation}
where $\area(X)$ is the area of the surface $X$ with respect to its flat metric. Equivalently, one has that $a(X)=s^{2}(X)/2$, where  $s(X):=\liminf_{t \to \infty} \sys(g_t X)/\area(X)$, see~\cite{Vorobets} and~\cite{HubertMarcheseUlcigrai}.

If $X$ belongs to a hyperelliptic component $\CCC^{hyp}(k)$, we show in Theorem~\ref{thm:lagrange} that $a(X)$ can be computed using diagonal changes algorithms. Furthermore, in~\cite{HubertMarcheseUlcigrai} it is shown that $\LLL(\CCC)$ is the closure of the values $1/a(X)$, $X\in \CCC$ for which the Teichm\"uller geodesic through $X$ is closed.  Thus, diagonal changes algorithms can be used to get finer and finer approximations of the Lagrange spectrum $\LLL(\CCC^{hyp}(k))$, by first listing closed Teichm\"uller geodesics in $\CCC^{hyp}(k)$ and then computing the corresponding Lagrange values. 


\subsection{Comparison with other algorithms in the literature} \label{subsec:comparisons}
In this section we compare diagonal changes algorithms with other induction algorithms in the literature: Ferenczi-Zamboni induction, Rauzy-Veech induction, da Rocha induction and train-track splittings. From a Diophantine point of view, we mention the analogy between Y.~Cheung's $Z$-convergents and our best approximations. From a combinatorial point of view we mention a link between the combinatorics of diagonal changes and cluster algebras combinatorics.

The Ferenczi-Zamboni induction (FZ induction for short), which is called by the authors
\emph{self-dual induction}, is an induction algorithm for interval exchange transformations (IETs), first introduced for IETs of $3$ intervals  in a series of papers jointly with Holton \cite{3IET1, 3IET2, 3IET3}, then in  \cite{FerencziZamboni-struct}
for all symmetric IETs (namely those with combinatorics given by the permutation $\pi(k) = n-k+1$, $1\leq k \leq n$). Very recently Ferenczi in~\cite{Ferenczi-all} developed a new induction for any IETs. The algorithm was defined and developed with the main aim of giving a combinatorial description of the IETs language and in particular to produce the list of bispecial words, see \cite{3IET2, FerencziZamboni-language}. The FZ algorithm was also used by Ferenczi and Zamboni to produce examples of IETs with special ergodic and spectral properties (see \cite{3IET3, 3IET4, FerencziZamboni-eig}).

The diagonal changes that we describe are a geometric version for translation surfaces in hyperelliptic components of FZ induction for symmetric IETs.
While the proofs of the existence of FZ-moves and that FZ induction sees all bispecial words given in~\cite{FerencziZamboni-struct} are purely  combinatorial, the proofs of the analogous results for staircase moves  in this paper are  very geometric. 
Some of the definitions and proofs for FZ induction are combinatorially quite heavy and we
believe that one of the advantages of our geometric approach is to
make the induction easier to understand and  proofs simpler and more
transparent. Recently Ferenczi extended the FZ induction for any IET~\cite{Ferenczi-all}. Following our paper, he also gave in \cite{Ferenczi-diagonalchanges} a geometric counterpart in the language of diagonal changes. Many of the geometric properties of staircase moves seem to
extend also for these general algorithms, which we stress are not
given by quadrangulations and staircase moves.

Another very well known induction algorithm for translation surfaces and IETs
is \emph{Rauzy-Veech induction}. The  Rauzy-Veech
induction for translation surfaces is a geometric invertible extension of the Rauzy induction for interval
exchanges in the same way the staircase moves for translation surfaces are an extension of FZ-moves for IETs. Rauzy-Veech induction has been a key tool to prove conjectures on the ergodic properties of IETs and linear flows on translation surfaces. The dynamics of the induction itself has been studied in detail (see for example~\cite{Veech} or most recently~\cite{AvilaBufetov} and~\cite{AvilaGouezelYoccoz}).

Despite the many applications of the Rauzy-Veech induction to ergodic problems, diagonal changes algorithms are a much better suited tool to attack some dynamical questions, in particular to list geometric best approximations in each bundle and to enumerate conjugacy classes of pseudo-Anosovs or equivalently Teichm\"uller closed geodesics. The heuristic explanation for this is that the Rauzy-Veech algorithm involves a choice of a conical singularity and of a separatrix which gives a transveral for the IET. Therefore, the domain on which the induction is defined,  namely the space of zippered rectangles introduced by Veech in~\cite{Veech}, is a \emph{finite-to-one} cover of connected components of strata of translation surfaces. In \cite{DU:II}, on the other hand, we explain that the space of quadrangulations, 
which is the analogous for staircase moves of the  space of zippered rectangles, yield  a faithful representation of hyperelliptic components.

The idea of an induction algorithm which, contrary to Rauzy-Veech induction, did not require the choice of a separatrix was long advocated, in particular by P.~Arnoux. In the setting of IETs a similar idea is also at the base of the induction invented by L. da Rocha (see~\cite{daRocha}) and of the induction  described by Cruz and da Rocha in~\cite{CdR} for rotational IETs. We remark that the latter, similarly to FZ induction for symmetric IETs, also uses  a bipartite IET structure. Recently Inoue and Nakada in~\cite{InoueNakada} defined a geometric extension of the Cruz-da Rocha induction by using zippered rectangles of a bipartite form  and showed that this extension is dual to Rauzy-Veech induction on zippered rectangles. 

Rauzy-Veech induction and diagonal changes algorithms may be seen as train-tracks algorithms. Train-tracks are combinatorial objects embedded in surfaces, that allow to describe measured foliations (such as the vertical foliation in a translation surface). Train-tracks splittings have been used in particular to provide a way to describe and enumerate conjugacy classes of pseudo-Anosov diffeomorphisms, see for examples~\cite{PapadopoulosPenner90}, \cite{Takarajima} and~\cite{LanneauThieffaut}. In the context of translation surfaces and Teichm\"uller dynamics, several results which exploit train-tracks splittings and a related symbolic coding of the Teichm\"uller flow were obtained by U.~Hamenst\"adt, see for example~\cite{Hamenstaedt}. Our diagonal changes algorithms use train-tracks of a very special form (which correspond to the bipartite nature of the IETs arising from quadrangulations, see~\S\ref{sec:correspondence_Q_IETs}). The train-tracks splittings allowed in our induction are the one which preserve this bipartite structure. Train-tracks algorithms often have the drawback that there is a large choice of possible moves and the graphs which describe combinatorial data are very large. In the case of our algorithms, the combinatorial graph associated to the moves (defined in~\S\ref{sec:Markov}) has a much more manageable size (see the comparison table in~\cite{DU:II}). Furthermore, as explained in~\cite{DU:II}, one can produce pseudo-Anosov diffeomorphisms from certain paths in the graph without having to check a rather subtle irreducibility condition which is needed when considering loops in the graph of train-tracks (see for example~\cite[Proposition 3.7]{Takarajima}).

In~\cite{SU1}, J.~Smillie and C.~Ulcigrai characterized the language of cutting sequences for linear trajectories on translation surfaces obtained from regular $2n$-gons (this characterization could also be proven for double regular $n$-gons with $n$ odd, see D.~Davis~\cite{Diana}). The characterization is based on an induction algorithm which uses affine diffeomorphisms in the Veech group, see also~\cite{SU2}. One can show that this algorithm turns out to be a diagonal changes algorithm. Ferenczi in~\cite{Ferenczi2ngon} considered the interval exchanges which arise as Poincar{\'e} maps of linear flows in regular $2n$-gons and described the FZ-moves which arise when performing FZ-induction starting from them. One can also see that the diagonal changes algorithm by Smillie and Ulcigrai is an acceleration of the geometric extension of the moves in~\cite{Ferenczi2ngon}. 

\smallskip

Let us now mention the connections with  $Z$-convergents and then cluster algebras. The notion of geometric best approximation for translation surfaces that we define in this paper is very close to the notion of $Z$-convergents for translation surfaces introduced by Y.~Cheung (see his joint paper~\cite{CheungHubertMasur} with P.~Hubert and H.~Masur for the definition). The definition is parallel to the notion of best approximation in the space of higher dimensional lattices that was used by Y. Cheung in~\cite{CheungHD}. The $Z$-convergents were further used by P.~Hubert and T.~Schmidt~\cite{HubertSchmidt} to provide transcendence criterion in the context of translation surfaces. In all these works on translation surfaces, the sequence of $Z$-convergents are considered from a theoretical point of view: no actual description of these sets were given. Diagonal changes algorithms provide an explicit construction of best approximations.

Finally, we remark that it turns out that the combinatorics which appear in diagonal changes (in particular the graph $\GGG$) is related to cluster algebras. Recently R.~Marsh and S.~Schroll in~\cite{MarshSchroll} explained this connection. In the case of FZ induction, they explain how one can put in one-to-one correspondence the trees of relations introduced in~\cite{FerencziZamboni-struct} with triangulations on the sphere and diagonal changes for these triangulations with the FZ-moves on the trees of relations defined by~\cite{CassaigneFerencziZamboni}. The combinatorics of these moves are exactly our staircase moves seen on the sphere (recall that in hyperelliptic components, each surface is a double cover of the sphere). 

\subsection{Structure of the paper}
In~\S\ref{sec:suspensions} we give a formal definition of staircase moves on the space of parameters which describe quadrangulations. We also explain the link between quadrangulations and bipartite interval exchanges and hence between staircase moves and FZ moves. Finally, we prove that staircase moves  display a form of self-duality and Markov structure. 

In~\S\ref{sec:induction} we first give the definition of hyperelliptic components. We then prove that translation surfaces in hyperelliptic components always admit a quadrangulation (Theorem~\ref{thm:existence_of_quadrangulation}) and that each of these quadrangulations has a well slanted staircase (Theorem~\ref{thm:existence_of_staircase_move}).

The applications of diagonal changes algorithms given by staircase moves mentioned above are considered in~\S\ref{sec:BA_and_bisp}. We first prove that staircase moves produce exactly all geometric best approximations (Theorem~\ref{thm:wedges_are_best_approx}). We then show how this result can be used to study the systole function along Teichm\"uller geodesics. Finally, we prove that bispecial words are exactly cutting sequences of best-approximations (Theorem~\ref{thm:seqsubstitutions}).

\subsection{Acknowledgements}
We would like to thank S.~Ferenczi, E.~Lanneau and S.~Schroll for useful discussions. 
We are grateful to P.~Hubert, who invited the second author to Marseille for a scientific visit during which diagonal changes were initially conceived  and who also immediately pointed out the connection with FZ induction. V.~Delecroix is supported by the ERC Starting Grant ``Quasiperiodics'' and C.~Ulcigrai is partially supported by the EPSRC Grant EP/I019030/1, which we thankfully acknowledge for making the authors collaboration possible. 

\section{Diagonal changes on the space of quadrangulations}\label{sec:suspensions}
We begin this section by describing in~\S\ref{sec:staircasedata} the combinatorial and length data which define a quadrangulation.
We then describe the link between quadrangulations and bipartite interval exchange maps (see~\S\ref{sec:correspondence_Q_IETs}). The induction developed by Ferenczi and Zamboni operates on bipartite interval exchanges.
In~\S\ref{subsec:staircase_moves} and~\S\ref{subsec:algorithms}  we give a more formal definition of staircase moves and the associated diagonal changes algorithms and explain the relation with FZ induction.   
Finally, in~\S\ref{sec:Markov} we show that the staircase moves are invertible and provide a Markov structure to the parameter space of quadrangulations. In particular, we show that our staircase moves provide a geometric realization of the natural extension of elementary FZ moves.
We also show that the inverse of a staircase move is again a staircase move.
In this sense these types of inductions are sometimes described as self-dual inductions.

\subsection{Parameters on quadrangulations}\label{sec:staircasedata}
Let $Q$ be a quadrangulation of a translation surface $X$. We saw in the introduction that $Q$ is determined by the collection  of wedges of quadrilaterals in $Q$. 
In addition to wedges, the quadrangulation  $Q$ 
 also determines a pair of permutations   which  describe how the quadrilaterals of the quadrangulation $Q $ are glued to each other as follows (refer to Figure \ref{fig:permutation_geometrically}).
 
\begin{figure}[!ht]
\begin{center}
\picinput{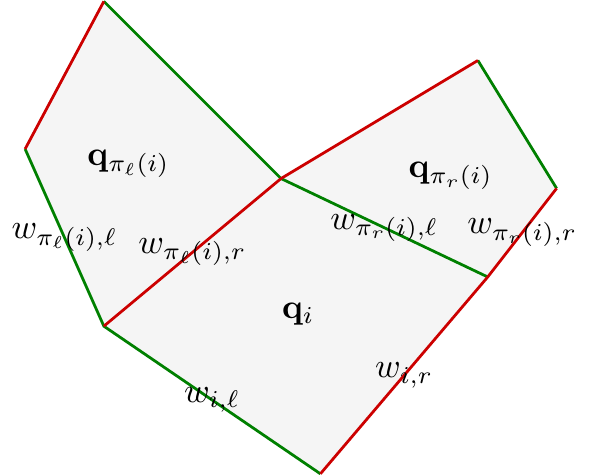}
\end{center}
\caption{a quadrilateral $q_i$ glued with $q_{\pi_r(i)}$ and $q_{\pi_\ell(i)}$\label{fig:permutation_geometrically}}
\end{figure} 

\begin{definition} Let $Q$ be a quadrangulation with $k$ quadrilaterals and let us label the quadrilaterals by the integers $\{1,\ldots,k\}$.
Let $q_i$ denote the quadrilateral labelled $i$.
The \emph{combinatorial datum} $\vpi = \vpi_Q$ of the labelled quadrangulation $Q$ is a pair $(\pi_\ell, \pi_r)$ of permutations of $\{1,\ldots,k \}$ such that
\begin{itemize}
\item[(i)] for each $1\leq i \leq k$, the top left side of $q_i$ is glued to the bottom right side of $q_{\lperm(i)}$.
\item[(ii)] For each $1\leq i \leq k$, the top right side of $q_i$ is glued to the bottom left side of $q_{\rperm(i)}$.
\end{itemize} 
\end{definition}
Thus $\lperm(i), \rperm(i)$  describe to which wedges the top sides of the quadrilateral $q_i$ belong, as illustrated by Figure \ref{fig:permutation_geometrically}. 
 
\begin{figure}[!ht]
\begin{center}
  \subfigure[a quadrangulation $Q$ with $\pi_\ell=(1,2,3)$ and $\pi_r=(1)(2,3)$]{\picinput{exampleg2}} 
  \subfigure[the graph $G_Q$ associated to $Q$]{\picinput{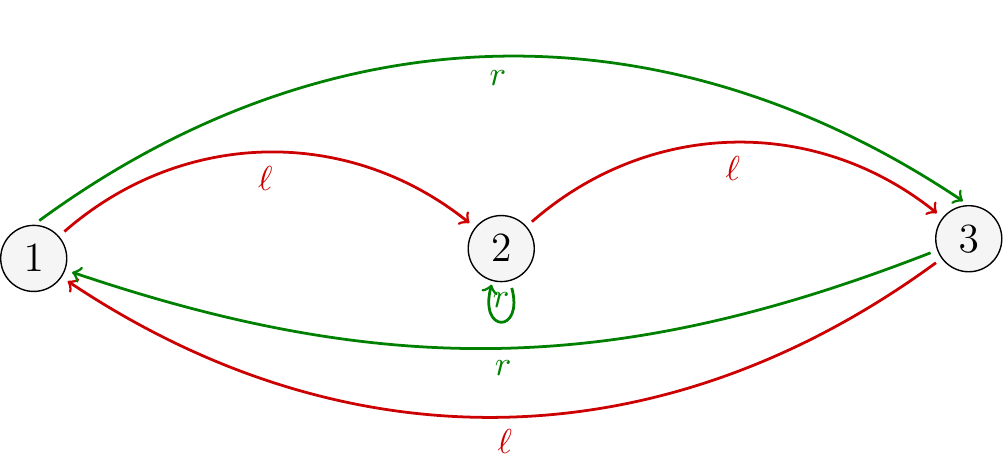}}
\end{center}
\caption{the graph $G_Q$ associated to a quadrangulation $Q$ of a surface in $\mathcal{H}(2)=\CCC^{hyp}(3)$}
\label{fig:graphExg2}
\end{figure} 

We mention that the combinatorial datum $\vpi_Q=(\lperm, \rperm)$ of a labelled quadrangulation $Q$ can also be described by a graph $G_Q$, whose vertices are in one-to-one correspondence with the quadrilaterals  $q_1, \dots, q_k$ and will be denoted by the corresponding index $1 \leq i \leq k$. The edges of $G_Q$ are labelled by $r$ or $l$ and are such that for each $1 \leq i \leq k$ there is an $\ell$-edge from $i$ to $\pi_\ell(i)$ and $r$-edge from $i$ to $\pi_r(i)$. An example is given in Figure~\ref{fig:graphExg2}. These graphs are used by Ferenczi and Zamboni in~\cite{FerencziZamboni-eig}. 

Let $Q$ be a labelled quadrangulation and let $\scw_1, \dots, \scw_k$ be the wedges corresponding to $q_1, \dots, q_k$.
Remark that quadrilaterals  in a quadrangulation (or equivalently, wedges) are in one to one correspondence with bundles of saddle connections. Thus, labelling the quadrilaterals in $Q$ by  $q_1, \dots, q_k$  automatically induces also a labelling of bundles by $\Quad_1, \dots, \Quad_k$ so that each $\scw_{i,\ell}$ (resp. $\scw_{i,r}$) belong to the bundle $\Quad_{i,\ell}$ (resp. $\Quad_{i,r}$). 

Since for each $\scw_{i,\ell}$ and $\scw_{i,r}$ the bundle to which they belong (resp. $\Quad_{i,\ell}$ or $\Quad_{i,r}$) is clear from the context, we will  without confusion  identify the  saddle connections in the wedges with the complex numbers which give their displacement vectors.
Using this notation and remarking that by construction $\scw_{i,\ell}$ and $\scw_{\lperm(i),r}$ are the left sides of the quadrilateral $q_i$ while $\scw_{i,r} $ and $\scw_{\rperm(i),\ell}$ are its right sides, we have
\begin{equation} \label{eq:train_track_wedges}
\scw_{i,\ell} + \scw_{\pi_\ell(i),r} = \scw_{i,r} + \scw_{\pi_r(i),\ell},\qquad 1\leq i \leq k.
\end{equation}
The equations in \eqref{eq:train_track_wedges} are called \emph{train-track relations}.

\smallskip

Conversely, we can construct a surface with a quadrangulation by starting from a combinatorial datum $\pi = (\pi_\ell,\pi_r)$ in $S_k \times S_k$ and a length datum 
$$\vwedges = ( (\scw_{1,\ell},\scw_{1,r}), \ldots, (\scw_{k,\ell},\scw_{k,r})) \in \left((\RR_- \times \RR_+) \times (\RR_+ \times \RR_+)\right)^k,$$ 
where $\RR_+ = \{ t \in \RR; \ t > 0\}  $ and $\RR_- = \{ t \in \RR; \  t < 0\}$. If $\vwedges$ satisfies the train-track relations~\eqref{eq:train_track_wedges} we can build a labelled quadrangulation $Q$ that we denote $(\vpi,\vwedges)$.
When we write $Q =(\vpi, \vwedges)$ we assume implicitely that $\vwedges$ satisfies the train-track relations.

\subsection{Bi-partite interval exchanges and quadrangulations}\label{sec:correspondence_Q_IETs}
Let us define bipartite interval exchange transformations and show that they arise as Poincar\'e first return maps of the vertical linear flow in a quadrangulation. Given $Q =(\vpi, \vwedges)$, the union of the wedges of $Q$ provide a convenient section for the vertical flow on the associated surface. The first return map on this section has a bipartite structure: for each $1\leq i \leq k $ the points on the wedge $\scw_i$ are divided in two sets depending on their future (the left part go to $q_{\pi_\ell(i)}$ and the right part to $q_{\pi_r(i)}$) and there is another partition with respect to their past (the left part comes from $q_{\pi_r^{-1}(i)}$ and the right part comes from $q_{\pi_\ell^{-1}(i)}$).

\begin{figure}[!ht]
  \begin{center}
	\subfigure[A bipartite IET\label{subfig:bipartite_IET}]{\picinput{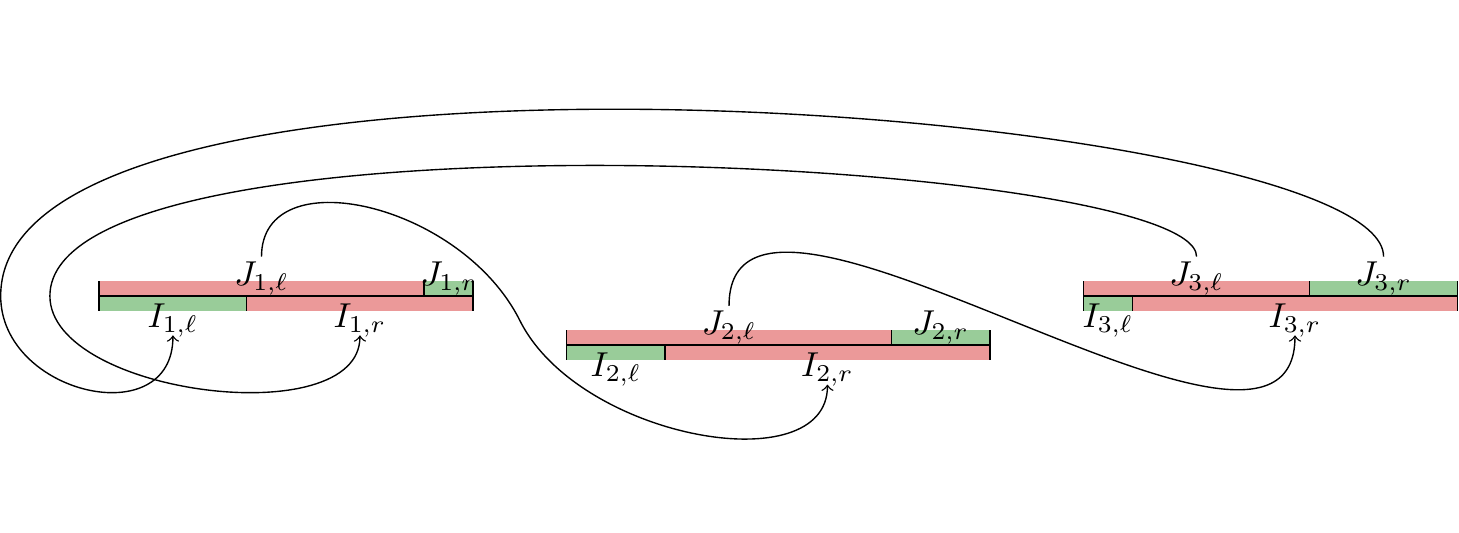}}
	\\
	\subfigure[The suspension of a the bipartite IET in Figure~\ref{subfig:bipartite_IET}\label{subfig:bipartite_IET_suspension}]{\picinput{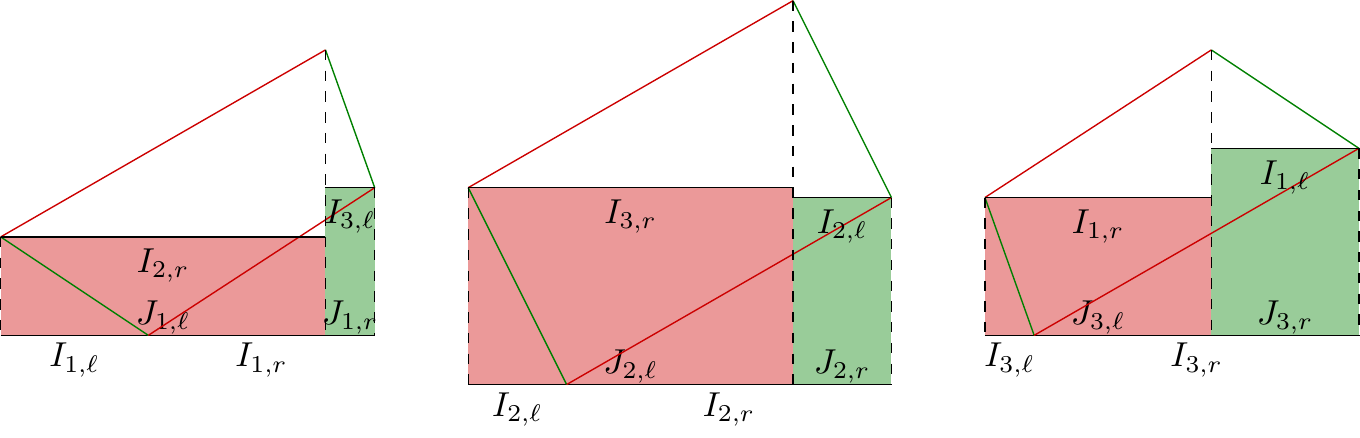}}
  \end{center}
  \caption{a bipartite interval exchange transformations with 3 intervals and one of its suspension. The resulting translation surface belongs to $\HHH(2) = \CCC^{hyp}(3)$}
\end{figure}

A \emph{bipartite interval exchange map} is a piecewise isometry $T: I \to I$ defined on the \emph{disjoint union} $I=\bigsqcup_{i=1}^k I_i$ of $k$ open bounded intervals $I_1,\ldots,I_k$. Each interval $I_i$ is  partitioned in two different ways as union of two intervals and $T$ maps isometrically the intervals in the first partition to the intervals in the second partition, so that the image of a right interval (resp.~a left interval) is a left (resp.~right) interval (see Figure~\ref{subfig:bipartite_IET}).

More formally, let $\vpi = (\pi_\ell, \pi_r)$ where $\pi_\ell$ and $\pi_r$ are two permutations of $\{1,\ldots,k\}$.
Let $\vlambda = ((\lambda_{1,\ell}, \lambda_{1,r}), \dots, (\lambda_{k,\ell}, \lambda_{k,r})) \in (\RR_- \times \RR_+)^k$ be such that 
\begin{equation}\label{eq:train_track_lengths}
  \lambda_{i,\ell} + \lambda_{\pi_\ell(i),r} = \lambda_{i,r} + \lambda_{\pi_r(i),\ell}, \qquad \forall 1\leq i \leq k.
\end{equation}
The relations given by the second formula in~\eqref{eq:train_track_lengths} are the \emph{train-track relations} for the lengths,  analogous to the ones for the wedges~\eqref{eq:train_track_wedges}.

For $i \leq i \leq k$, set $I_i = \left(\lambda_{i,\ell},\lambda_{i,r}\right) \subset \RR$ and let 
\[\begin{array}{cc}
I_{i,\ell}  = \left(\lambda_{i,\ell},0\right), & I_{i,r} = \left(0,\lambda_{i,r}\right), \\
J_{i,\ell} = \left(\lambda_{i,\ell}, \lambda_{i,\ell} + \lambda_{\pi_\ell(i),r}\right), & J_{i,r} = \left(\lambda_{i,r} + \lambda_{\pi_r(i),\ell}, \lambda_{i,r}\right).
\end{array}
\]
Remark that $\{I_{i,\ell}, I_{i,r}\}$ is obviously a partition of $I_i \backslash \{0\}$ and the train track relations~\eqref{eq:train_track_lengths} imply that $\{J_{i,\ell}, J_{i,r}\}$ is a partition of $I_i \backslash \{\lambda_{i,d}\}$ where $\lambda_{i,d} = \lambda_{i,\ell} + \lambda_{\pi_\ell(i),r} = \lambda_{i,r} + \lambda_{\pi_r(i),\ell}$.
 
\begin{definition} \label{def:bipartite_IET}
The \emph{bipartite interval exchange map} with data $(\vpi, \vlambda)$ is the map from $I = I_1 \sqcup \ldots \sqcup I_k$ that maps by translation $J_{i,l}$ to $I_{\pi_\ell(i),r}$ and $J_{i,r}$ to $I_{\pi_r(i),\ell}$. The map is not defined at the points $\lambda_{i,d} \in I_i$, $1\leq i \leq k$.
\end{definition}

We introduced bipartite IETs so that the following holds. Let us call \emph{interior} of a wedge $\scw= (\scw_\ell, \scw_r)$ the union of the interiors of the saddle connections $\scw_\ell$ and $\scw_r$ together with their common singularity point.
\begin{lemma}[cross sections of quadrangulations] \label{lem:bipartite_IET_section}
Given a quadrangulation $Q = (\vpi, \vwedges )$, the Poincar{\'e} first return map $F$ of the vertical flow on the union of the interiors of the  wedges of $Q$ is conjugate to the bipartite IET $T = (\vpi,\vlambda)$, where the vector $\vlambda$ is given by the real parts of the wedges. More precisely, if $p$ is the projection $p$ that maps a point $z$ of the wedge $\scw_i$ to the point $\Re(z) \in I_i$, we have $p F = T p$. 
\end{lemma}
Remark that for each $1\leq i \leq k$ the IET $T$ is defined at all points of $I_i$ except at the point $\lambda_{i,d} \in I_i$, which   corresponds to the unique point of the wedge $w_i$ whose trajectory hits  an endpoint of a wedge (and hence $F$ is not defined there). 
Clearly the Lebesgue measure on $I$ is invariant under $T$. The pull back of the Lebesgue measure  $p$ is the absolutely continuous \emph{transverse measure} invariant under the Poincar{\'e} map.  

\smallskip

Conversely, starting from a given bipartite IET $T$ we can construct as follows a family of quadrangulations on a surface $X$ for which $T$ is the Poincar{\'e} first return map on the union of the wedges, see Figure~\ref{subfig:bipartite_IET_suspension}.
\begin{definition}
A \emph{suspension data} $\vtau$ for the bipartite IET $(\vpi, \vlambda)$ is a vector $\vtau = \left((\tau_{1,\ell},\tau_{1,r}), \dots, (\tau_{k,\ell},\tau_{k,r})\right)$ in $(\RR_+ \times \RR_+)^k$ that satisfies the train-track relations
\[
\tau_{i,\ell} + \tau_{i,r} = \tau_{\pi_r(i),\ell} + \tau_{\pi_\ell(i),r}, \quad \text{for $i=1,\ldots,k$.}
\]
\end{definition}
To the interval exchange data $(\vpi, \vlambda)$ and the suspension datum $\vtau$ we associate a quadrangulation $Q = (\vpi , \vlambda, \vtau) = (\vpi, \vwedges)$ where the wedges of $Q$ are $\scw_{i,\ell} = \lambda_{i,\ell} + \sqrt{-1}\, \tau_{i,\ell}$ and $\scw_{i,r} = \lambda_{i,r} + \sqrt{-1}\, \tau_{i,r}$. 
The following result can be seen as a converse of Lemma~\ref{lem:bipartite_IET_section}. 
\begin{lemma}[suspensions of bipartite IETs] \label{lem:bipartite_IET_suspension}
Given a bipartite IET $T = (\vpi, \vlambda)$ and a suspension datum $\vtau$ for $T$, let $Q = (\vpi,\vlambda,\vtau)$ be the associated quadrangulation. Then the Poincar{\'e} map of the vertical flow on the associated surface  on the union of the interior of the wedges of $Q$ is conjugated to $T$. 
\end{lemma}

We remark that the vertical flow on the translation surface given by $Q = (\vpi,\vwedges)$ can also be represented as a special flow over the corresponding bipartite IET $T = (\vpi, \vlambda)$. The components of the vector $\vtau$ give the heights of the corresponding Rohlin towers. One can embedd geometrically these towers into the surface as shown in Figure~\ref{subfig:bipartite_IET_suspension}. Note that with this representation by Rohlin towers, the section is naturally given by horizontal intervals in the surface.

\subsection{Staircase moves and Ferenczi-Zamboni moves}\label{subsec:staircase_moves}
In the introduction we already gave a geometric definition of a staircase move (Definition~\ref{def:staircasemove}).
Let us now describe the corresponding operation on quadrangulation data.

Given a quadrangulation $Q = (\vpi, \vwedges)$,   
let us recall that the top right  side of the quadrilateral $q_i$ is glued to the quadrilateral $q_{\pi_{r(i)}}$. 
Thus, if $\{i, \pi_r (i), \dots \pi_r^n(i)\}$ is a \emph{cycle} of $\pi_r$, that is $\pi_r^j(i)\neq i$ for $1 \leq j \leq n$ but $\pi_r^{n+1}(i)=i$, the corresponding quadrilaterals  $\{q_i, q_{\pi_r (i)}, \dots q_{\pi_r^n(i)}\}$  are glued to each other through top right/bottom left  sides. Similarly, since the top left  side of $q_i$ is glued to $q_{\pi_\ell(i)}$, the  quadrilaters $\{q_i, q_{\pi_\ell (i)}, \dots q_{\pi_\ell^n(i)}\}$ indexed by a cycle of $\pi_\ell$ are glued to each other through top left/bottom right sides.  
\begin{notation} \label{not:cstaircase} 
Given a cycle $c \in \pi_\ell$ (respectively a cycle $c \in \pi_r$) we denote by $S_c$ the left (respectively right) staircase for  $Q$ which is obtained as union of the quadrilaterals in $Q$ indexed by the cycle $c$. Abusing the notation, we will denote  by $S=S_c$ both the collection of quadrilaterals and  their union as a subset of $X$, so we will both write $S \subset X$ and $q \in S$ where $q $ is one of the quadrilaterals of $Q$ contained in $S$.
\end{notation}

Let  $Q =(\vpi, \vwedges)$ be a quadrangulation.
For each wedge $\scw_i = (\scw_{i,\ell},\scw_{i,r})$ of a quadrilateral $q_i \in Q$, we denote by $\scw_{i,d}$ (or by  $\scw_{i,d^+}$) the (forward)  diagonal of the quadrilateral, which is given by
\begin{equation}\label{eq:def_diagonal}
\scw_{i,d}= \scw_{i,d^+}: = \scw_{i,\ell} + \scw_{\pi_\ell(i),r} = \scw_{i,r} + \scw_{\pi_r(i),\ell},
\end{equation} 
where the above equality holds by the train-track relations~\eqref{eq:train_track_wedges} for $\vwedges$. 
Remark that a right (resp.~left) staircase $S_c$ associated to a cycle $c $ of $\pi_r$ (resp.~$\pi_\ell$) is well slanted (see Definition~\ref{def:cstaircase}) if and only if $\Re(\scw_{i,d}) < 0$ ($\Re(\scw_{i,d}) > 0$) for all $i \in c$.

Let $c$ be a cycle of $\pi_r$ and  assume  that the corresponding staircase $S_c$ is  well slanted. Let us show that the staircase move in $S_c$ produces a new quadrangulation and describe its data (refer to Figure~\ref{fig:staircase_move_data} and see also Lemma~\ref{lem:staircase_move_data} below). 
Since in a  diagonal change, we replace a side of a wedge with its diagonal it is clear that after the staircase move we obtained a new length data  $\vwedges'$ given by 
\begin{equation}\label{eq:def_w1}
\scw_i' = \left\{ \begin{array}{ll}
  (\scw_{i,d}, \scw_{i,r} ) & \text{if $i \in c$,} \\
  \scw_i & \text{otherwise.}
\end{array} \right.
\end{equation}
From the well slantedness of the staircase $S_c$, it follows that also $\scw_i'$, $1\leq i \leq k$ are wedges, that is $w'_{i,\ell} \in \RR_- \times \RR_+$ and $w'_{i,r} \in \RR_+ \times \RR_+$. 
Furthermore, the wedges  $\vwedges'$ determine a new quadrangulation $Q'$ since, as shown in  Figure~\ref{fig:staircase_move_data},  $\scw_i'$ for $i \in c$ is the wedge of the quadrilateral $q_i'$ which has $\scw_{\pi_r(i), d}$  as right top edge and $\scw_{\pi_l \pi_r(i),\ell}$  as left top edge . This also shows that the quadrilateral glued to the top right  side of $q_i'$ is $q_{\pi_r(i)}'$ while the quadrilateral glued to the top left  side of $q_i'$ is $q_{\pi_{\ell}(\pi_r(i))}$, as shown in Figure~\ref{fig:staircase_move_data}. 

Thus, the combinatorics $\vpi' = (\pi_\ell', \pi_r')$ of the new quadrangulation $Q'$  is given by
\begin{equation} \label{eq:def_c_pi1}
\pi'_\ell(i) = 
\left\{ \begin{array}{ll}
  \pi_\ell \circ \pi_r(i) & \text{if $i \in c$,} \\
  \pi_\ell(i) & \text{otherwise.}
\end{array}\right.
\qquad \text{and} \qquad
\pi_r'=\pi_r.
\end{equation}
We will denote by $c \cdot \vpi$ the new combinatorial datum $\vpi'$ given by the above formulas.  
It follows from the formula for $\vpi'$ that the train-track relations for $\vpi'=c\cdot \vpi$ are satisfied by $\vwedges'$.

\begin{figure}[!ht]
\begin{center}
\picinput{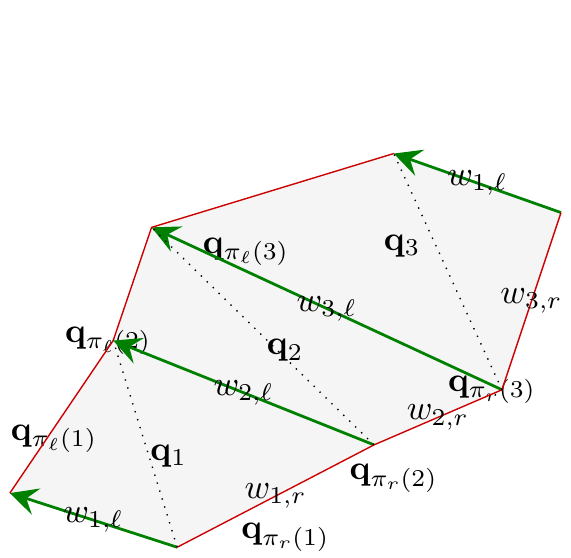} \vspace{1cm} \picinput{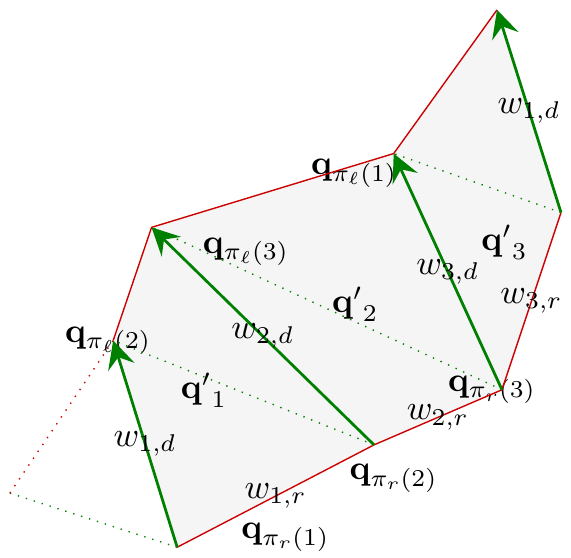}
\end{center}
\caption{right staircase move on the parameters $(\vpi,\vwedges)$ of a quadrangulation}
\label{fig:staircase_move_data}
\end{figure}

Similary, if $c$ is a cycle of $\pi_\ell$ and $S_c$ is well slanted, the staircase move in $S_c$ produces a new quadrangulation $Q'=(c\cdot \vpi, \vwedges')$ where  $\vwedges'$ and $c \cdot \vpi  = (\pi_\ell', \pi'_r)$ is given by
\begin{equation} \label{eq:def_w2}
\scw'_i =
\left\{ \begin{array}{ll}
  (\scw_{i,\ell}, \scw_{i,d} ) & \text{if $i \in c$,} \\
  \scw_i & \text{otherwise,}
\end{array} \right.
\end{equation}
\begin{equation} \label{eq:def_c_pi2}
\pi_\ell'=\pi_\ell
\qquad \text{and} \qquad
\pi'_r(i) = 
\left\{ \begin{array}{ll}
  \pi_r \circ \pi_\ell(i) & \text{if $i \in c$,} \\
  \pi_r(i) & \text{otherwise.}
\end{array}\right.
\end{equation}

\smallskip
We remark that the operation on the permutation $\vpi$ does not depend on the length datum and the operation on the wedges $\vwedges$ is linear. Thus, to describe the new length datum $\vwedges'$, we introduce the $2k \times 2k$ matrix $A_{\vpi,c}$ as follows.
We index the rows and columns of $A_{\vpi,c}$ by the $2k$ indices $(1,\ell), (1,r)$, $(2,\ell), (2,r)$, \dots, $(k,\ell), (k,r)$. 
Let $I_{2k}$ the be $2k \times 2k$ identity matrix and for $1\leq i, j\leq k$ and  $\epsilon, \nu \in \{l,r\}$ let $E_{(i,\epsilon), (j,\nu)  }$ be the $2k \times 2k$ matrix whose entry in row $(i,\epsilon)$ and column $(j,\nu)$ is $1$ and all the other entries are zero.
We set
\begin{equation} \label{eq:def_A_pi_c}
A_{\vpi,c} =
\left\{ \begin{array}{ll}
  I_{2k} + \sum_{i \in c} E_{ (i,\ell),(\pi_\ell(i),r)} & \quad \text{if $c$ is a cycle of $\pi_r$}, \\
  I_{2k} + \sum_{i \in c} E_{ (i,r),(\pi_r(i),\ell)} & \quad \text{if $c$ is a cycle of $\pi_\ell$}.
\end{array} \right.
\end{equation}
Thus, with the convention that $\vwedges$ and $\vwedges'$ denote column vectors, one can verify from  equations~\eqref{eq:def_diagonal}, \eqref{eq:def_w1} and \eqref{eq:def_w2} that we can write  $\vwedges' = A_{\vpi,c}\ \vwedges$.  
Thus, we proved the following: 
\begin{lemma}[staircase move on data]\label{lem:staircase_move_data}
Given a labelled quadrangulation  $Q = (\vpi, \vwedges)$ and a cycle $c $ of $\vpi$, if the staircase $S_c$ is well slanted, when performing on $Q$ the staircase move in $S_c$ one obtains a new labelled quadrangulation $Q' = (\vpi',\vwedges')$ with
\[
\vpi' = c\cdot \vpi,  \qquad \vwedges' = A_{\vpi,c}\ \vwedges,
\]
where $c \cdot \vpi$ and $A_{\vpi,c}$ are given by formulas~\eqref{eq:def_c_pi1},~\eqref{eq:def_c_pi2}~and~\eqref{eq:def_A_pi_c} above. 
\end{lemma}

One can moreover show that staircases are the smallest unions of quadrilaterals in which one can simultaneously perform diagonal changes to obtain a new quadrangulation, in the following sense.
\begin{lemma}\label{lem:stable}
Let $Q = (\vpi,\vwedges)$ be a quadrangulation and let $\mathcal{I}_\ell, \mathcal{I}_r  \subset \{1,\ldots,d\}$ be such that the quadrilaterals $q_i$ with  $i \in \mathcal{I}_\ell$ are left slanted  and the quadrilaterals $q_i$ with $i \in \mathcal{I}_r$ are right slanted. 
 
The new set of wedges obtain after individual diagonal changes in the quadrilaterals $Q_i$ for $i \in \mathcal{I}= \mathcal{I}_\ell \cup \mathcal{I}_r$ is associated to a quadrangulation if and only if the set of indices $\III_\ell$ (respectively $\III_r$) is a union of cycles of $\pi_\ell$ (resp. $\pi_r$).
\end{lemma}
\noindent We leave the proof to the reader.

\smallskip
One can verify that staircase moves provide a geometric extension of the elementary moves on bipartite IETs which appear in the FZ induction~\cite{FerencziZamboni-struct}, in the following sense.
\begin{remark} 
Let $Q=(\vpi,\vlambda,\vtau)$ be a quadrangulation of a surface in $\CCC^{hyp}(k)$ and $T=(\vpi,\vlambda)$ be the corresponding bipartite IET.  Let $Q'= (\vpi', \vlambda', \vtau')$ be the quadrangulation obtained from $Q$ by performing a staircase move in $c$ and let $T'$ be corresponding bipartite IET. Then $T'$ is the bipartite IET obtained from $T$ by one elementary step of a FZ move.
\end{remark}

\smallskip 
An alternative description of the geometric extension can be given in terms of Rohlin towers.  The action of a staircase move at the level of Rohlin towers associated to a quadrangulation is the stacking operation shown in Figure~\ref{fig:diagonal_change_suspension}. 

\begin{figure}[!ht]
\begin{center}
\picinput{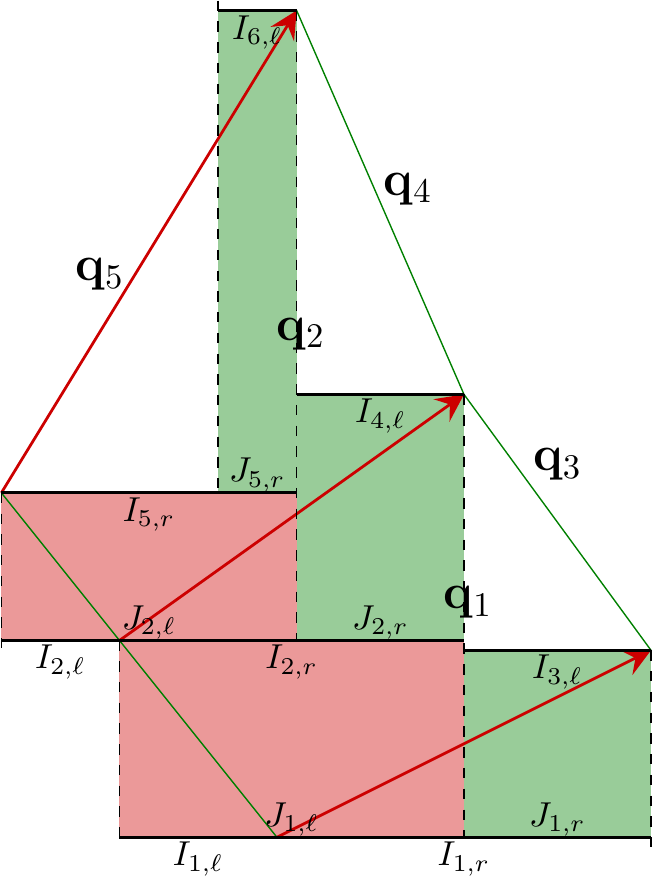}
\hspace{1cm}
\picinput{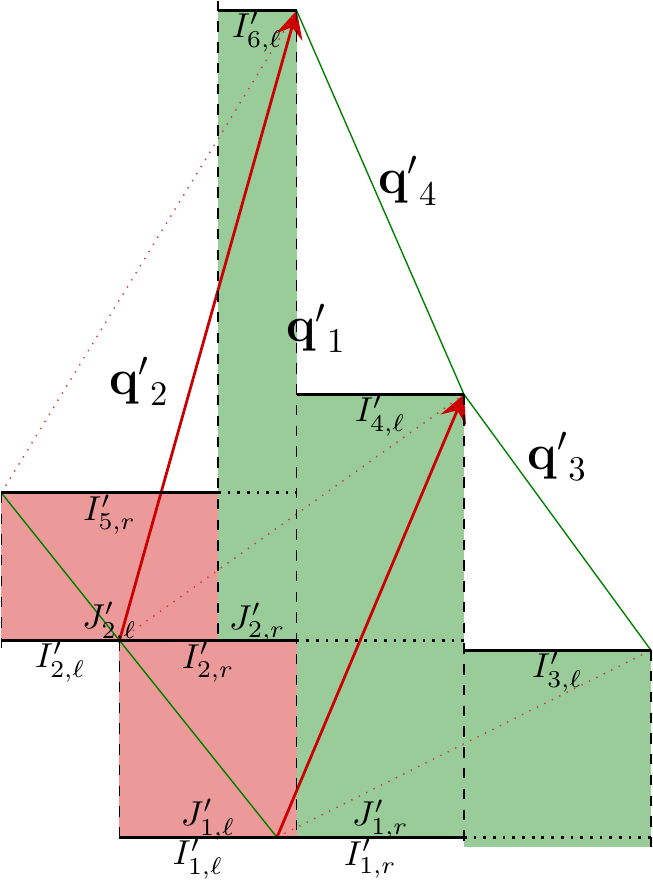}
\end{center}
\caption{diagonal changes seen on suspension}
\label{fig:diagonal_change_suspension}
\end{figure}

\subsection{Diagonal changes algorithms given by staircase moves.} \label{subsec:algorithms}
Let $Q = Q^{(0)}$  be a given starting quadrangulation. An \emph{algorithm} produces a sequence of quadrangulations $Q^{(1)}$, $Q^{(2)}$, \ldots in such way that $Q^{(n+1)}$ is obtained from $Q^{(n)}$ by a sequence of staircase moves. 
As we already mentioned there might be several possible staircase  moves. 
Remark that if $S_1$ and $S_2$ are (disjoint) well slanted staircases in $Q$, the staircase moves in $S_1$ and $S_2$ commute, so that the order in which they are performed does not matter and the two moves can be performed simultaneously. If $Q'$ is obtained from $Q$ by performing staircase moves in a subset of the well slanted staircases of $Q$, we will say that  $Q'$ is obtained from $Q$ by  \emph{simultaneous staircase moves}. 

Let us first define the greedy algorithm, which corresponds to the algorithm introduced in~\cite{FerencziZamboni-struct} for bipartite IETs.
\begin{definition}[greedy algorithm]
The \emph{greedy diagonal changes algorithm} starting from $Q= Q^{(0)}$ produces the sequence $(Q^{(n)})_{ n \in \mathbb{N}}$ of quadrangulations where  $Q^{(n+1)}$ is obtained  from $Q^{(n)}$ by performing simultaneous staircase moves in all well slanted staircases for $Q^{(n)}$. 
\end{definition}

Let us remark that a left (resp.~right) staircase move does not modify $\pi_\ell$ (resp.~$\pi_r$). Thus, even if the  quadrilaterals in a staircase change,  the left (resp.~right) staircases (each seen as union of the corresponding quadrilaterals) do not change during a   left (resp.~right) staircase move. 
 Thus, it makes sense to define the 
 \emph{multiplicity} of a left (resp.~right) staircase $S_c$ as  the maximum $n$ such that we can perform $n$ consecutive left (right)  staircase moves in $S_c$. The following algorithm may be thought as a generalization of the multiplicative continued fraction algorithm (associated to the Gauss map) that is an acceleration of the additive one (associated to the Farey map).
\begin{definition}[left/right algorithm]
The \emph{left/right algorithm} starting at $Q^{(0)}=Q$ produces a sequence $(Q^{(n)})_{ n \in \mathbb{N}}$ where, if $n$ is even, $Q^{(n+1)}$ is obtained from $Q^{(n)}$  by performing in each left slanted staircase as many left  staricase moves as the multiplicity of the staircase, while if $n$ is odd $Q^{(n+1)}$ is obtained by  doing the same for all right slanted staicases for $Q^{(n)}$. 
\end{definition}
Remark that in the left/right algorithm $Q^{(n+1)}$ is in general obtained from $Q^{(n)}$ by several staircase moves that are not simultaneous.
A version of this algorithm was already used in~\cite{Delecroix} for the stratum $\HHH(2) = \CCC^{hyp}(3)$.

In~\cite{DU:II} we describe a third diagonal change algorithm, which we call \emph{geodesic algorithm}, which is determined by the return map to a Poincar\'e section of the Teichm\"uller flow. 
Another different version of a diagonal change algorithm at the level of interval exchanges was used in~\cite{Ferenczi2ngon} to describe interval exchanges that comes from flat surfaces built from $2n$-gons. One can check that their algorithm is actually the ``additive'' version at the level of IETs of the algorithm described by Smillie and Ulcigrai in~\cite{SU1,SU2}.

Let us say that a diagonal changes algorithm given by staircase moves is a \emph{slow algorithm} if each of its moves is given by simultaneous staircase moves. The greedy algorithm is an example of a slow algorithm, while the left/right algorithm, the geodesic algorithm and the one described by  Smillie and Ulcigrai in~\cite{SU1,SU2} are not. Theorem~\ref{thm:same_objects} in~\S\ref{sec:sameobjects} shows that \emph{any} slow algorithm actually produce the same geometric objects and therefore the choice of an actual algorithm is not so important.

Further information on the relation between different (not necessarily slow) algorithms can be deduced from~\cite{DU:II}, where we give a detailed description of the structure of the set of quadrangulations on a given surface $X$. In particular, we show that given a surface $X$ in a hyperelliptic component, for any two quadrangulations $Q_1$ and $Q_2$ of $X$ there exists a sequence of backward and forward staircase moves from $Q_1$ to $Q_2$.

\subsection{Invertibility, self-duality and Markov structure on parameter space} \label{sec:Markov}  \label{sec:move_parameters}
In this section we introduce the space of (labelled) quadrangulations of surfaces in a component $\CCC^{hyp}(k)$. We prove that staircase moves are invertible and self-dual on this set of quadrangulations (see Theorem~\ref{thm:Markov_self_dual}).

Let us fix $k$ and build the space of all labelled quadrangulations of surfaces in $\CCC^{hyp}(k)$.
Start from a fixed combinatorial datum $\vpi$ of such a surface and consider the oriented graph $\GGG = \GGG(\vpi)$ defined as follows.
The vertices are the set of combinatorial data that may be obtained from $\pi$ by a sequence of staircase moves.
There is an edge from $\pi$ to $\pi'$ labelled by $c$ if and only if $c \cdot \pi = \pi'$. In Figure~\ref{fig:graph_of_graphs} we show the graph associated to $\pi_\ell = (1,3)$ and $\pi_r=(1,2)$. The notation for cycles used in the figure, which makes clear whether a cycle belong to $\pi_\ell$ or $\pi_r$, is the following: if $c$ is a cycle of $\pi_\ell$ then we write it as a word of length $k$ on the alphabet $\{\cdot,\ell\}$, where the $i^{th}$ letter of the word is $\ell$ if and only if $i \in c$. For example, the cycle $c=\{1,3\}$ is denoted $\ell \cdot \ell$. Cycles of $\pi_r$ are denoted in the same way using words on the alphabet  $\{\cdot,r\}$.

\begin{figure}[!ht]
\begin{center} \picinput{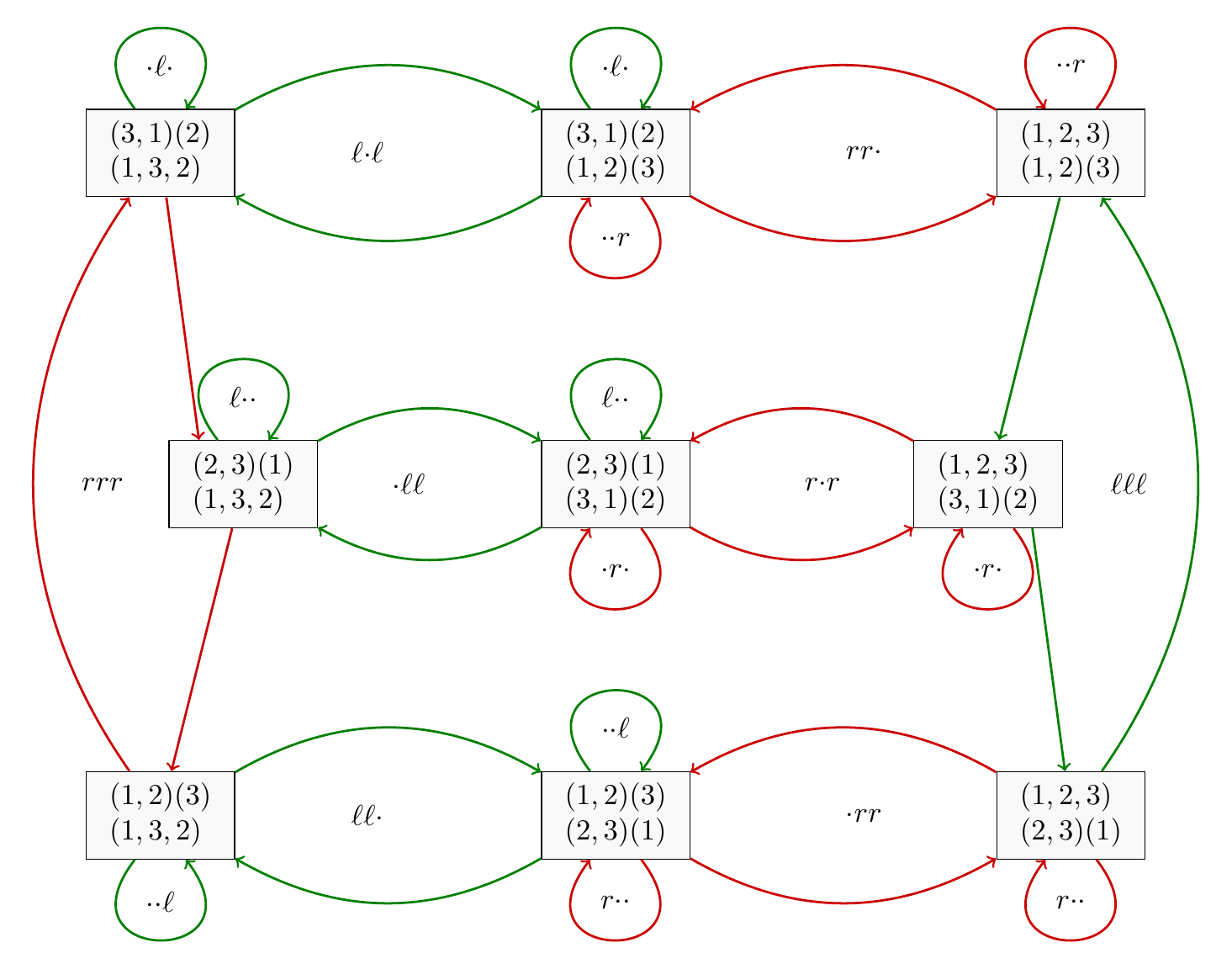} \end{center}
\caption{the graph $\GGG$ of combinatorial data for quadrangulations in $\CCC^{hyp}(3) \simeq \HHH(2)$}
\label{fig:graph_of_graphs}
\end{figure}

As we will see in~\S\ref{subsec:hyperelliptic}, if $\vpi = (\pi_\ell,\pi_r)$ is a combinatorial datum of a quadrangulation of a surface in $\CCC^{hyp}(k)$, there exists an involution $\iota$ of $\{1,\ldots,k\}$, that corresponds to the action of the hyperelliptic involution on the quadrilaterals. Moreover $\pi_\ell\, \pi_r\, \iota$ is a $k$-cycle and is invariant under the operation $c \cdot \pi$ associated to a staircase move, i.e. the $k$-cycles associated to the vertices of $\GGG$ are the same. It is proven in~\cite{CassaigneFerencziZamboni} that this invariant is complete, i.e. that two pairs $\vpi$ and $\vpi'$ belongs to the same graph if and only if $\pi_\ell\, \pi_r\, \iota = \pi'_\ell\, \pi'_r\, \iota'$. The same result is proved in~\cite{DU:II} using the ergodicity of the Teichmueller flow on~$\CCC^{hyp}(k)$.
In particular, starting from different combinatorial data $\vpi$ and $\vpi'$ that correspond to two quadrangulations of surfaces in the same component $\CCC^{hyp}(k)$, then the graph $\GGG(\vpi)$ and $\GGG(\vpi')$ are isomorphic.  More precisely, there exists a permutation $\sigma$ in $S_k$ such that  the isomorphism is given by $(\pi_\ell,\pi_r) \mapsto (\sigma \pi_\ell \sigma^{-1}, \sigma \pi_r \sigma^{-1})$.

For each combinatorial datum $\vpi = (\pi_\ell,\pi_r)$ in $\GGG$, let us introduce the cones $\Delta_{\vpi} \subset (\RR^{2})^k$ and $\Theta_\vpi \subset (\RR^{2})^k$ that parametrize all possible lengths and heights of wedges with combinatorial datum $\vpi$, that is the lengths and heights which satisfy the train-track relations given by $\vpi$. Formally
\begin{align*}
\Delta_{\vpi} = \{& \left( (\lambda_{1,\ell},\lambda_{1,r}), \dots, (\lambda_{k,\ell},\lambda_{k,r}) \right) \in (\RR_- \times \RR_+)^k; \\
   &\lambda_{i,\ell} + \lambda_{\pi_\ell(i),r} = \lambda_{i,r} + \lambda_{\pi_r(i),\ell}\quad \text{for $1\leq i \leq k$} \} \nonumber \\ \nonumber
\Theta_{\vpi} = \{&\left( (\tau_{1,\ell},\tau_{1,r}), \dots, (\tau_{k,\ell},\tau_{k,r}) \right) \in (\RR_+ \times \RR_+)^k; \\
 & \tau_{i,\ell} + \tau_{\pi_\ell(i),r} = \tau_{i,r} + \tau_{\pi_r(i),\ell}\quad \text{for $1\leq i \leq k$} \}. 
\end{align*}
Then the \emph{space of labelled quadrangulations} of surfaces in $\CCC^{hyp}(k)$ is
\[
\QQQ_k = \{(\vpi,\vlambda,\vtau); \  \vpi \in \GGG,\ \vlambda \in \Delta_\vpi, \ \vtau \in \Theta_\vpi\}.
\]
In~\cite{DU:II} we show that each hyperelliptic component $\CCC^{hyp}(k)$ is essentially the same as $\QQQ_k / \sim$ where $\sim$ is the equivalence relation generated by staircase moves.

\smallskip

Given $(\vpi,\vlambda,\vtau) \in \QQQ_k$ and a cycle $c$ of $\pi_r$ or $\pi_\ell$,  remark that the heights $\vtau  \in \Theta_\vpi$ play no role in  determining whether $S_c$ is well-slanted. Thus, let $\Delta_{\vpi,c} \subset \Delta_\vpi$ be the subset of lengths data for which the staircase $S_c$ is well slanted.   
Recall that the (forward) diagonal $\scw_{i,d} = w_{i,\ell} + w_{\pi_\ell(i),r} = w_{i,r} + w_{\pi_r(i),\ell}$ of $q_i$ is left (resp. right) slanted if and only if its real part $\lambda_{i,d} = \Re \scw_{i,d}$ is greater than $0$ (resp. less than $0$).
Thus, formally, we have
\begin{equation}\label{eq:deltac_def}
\Delta_{\vpi,c} := \left\{ \begin{array}{ll}
  \{\vlambda \in \Delta_{\vpi} \ | \quad \lambda_{i,d} < 0 \ \forall i \in c\}, & \text{if $c$ is a cycle of $\pi_r$,} \\
  \{\vlambda \in \Delta_{\vpi} \ |\quad \lambda_{i,d} > 0 \ \forall i \in c\},  & \text{if $c$ is a cycle of $\pi_\ell$.} \\
\end{array}
  \right.
\end{equation}

Then one can perform a staircase move in $S_c$ if and only if $\vlambda \in \Delta_c$. Using the Definitions~\eqref{eq:def_c_pi1} and~\eqref{eq:def_c_pi2} of $c \cdot \vpi$ and the definition~\eqref{eq:def_A_pi_c} of $A_{\vpi,c}$ and remarking that $A_{\vpi,c}$ acts linearly both on the real and imaginary part of each saddle connection in $\vwedges$, we can formally define a staircase move on the parameter space as follows:
\begin{definition}\label{def:mchat}
Let $\vpi =(\pi_\ell,\pi_r)\in \GGG$ and let $c$  be a cycle of $\pi_r$ or $\pi_\ell$.  The staircase move $\widehat{m}_{\vpi,c}$ on $\QQQ_k$ is map defined on $\{ \vpi \} \times \Delta_{\vpi,c} \times \Theta_\vpi \subset \QQQ_k$ which sends $(\vpi, \vlambda, \vtau)$ to $\widehat{m}_{\vpi,c} (\vpi,\vlambda, \vtau) = (c \cdot \vpi,\, A_{\vpi,c}\ \vlambda, A_{\vpi,c}\ \vtau )$. 
\end{definition}
 
\smallskip
 
Geometrically, the inverse of a staircase move in $X$ is simply a staircase move in the surface obtained from $X$ by counterclockwise rotation by 90 degrees. To formalize the action by rotation, we introduce the operator $R$ on the parameter space of quadrangulations $\QQQ_k$. Remark that if $q \subset \CC$ is an admissible quadrilateral, multiplying by the imaginary unit $\sqrt{-1}$ we get the rotated quadrilateral $\sqrt{-1} q$ which is  still admissible. Thus, if $Q$ is a labelled quadrangulation for $X$, then the  collection of quadrilaterals $q' = \sqrt{-1}q$  also determine a quadrangulation of $X$, which we denote by $\sqrt{-1}Q$. We denote by  $Q'$ the quadrangulation $\sqrt{-1} Q$  labelled so that the wedge $\sc_i'$ of the quadrilateral $q_i'$ contains the same vertical ray which was contained in the wedge  $\sc_i$ of $q_i$, as shown in Figure~\ref{fig:rotation}). 
 As we prove below, this convention for the labelling (but not for example the more naive convention of calling $q_i'$ the quadrilateral $\sqrt{-1} q$) guarantees that the operator $R$ that sends  $Q$ to $Q'$ is a well defined operation on the space $\QQQ_k$ of labelled quadrangulations.  
The explicit formulas for the wedges and combinatorial datum of $q'\in Q'$ can be easily obtained from $Q=(\vpi, \vwedges)$ by looking at Figure~\ref{fig:rotation} and lead to the following formal definition:

\begin{figure}[!ht]
  \begin{center}
\subfigure[$Q$]{\picinput{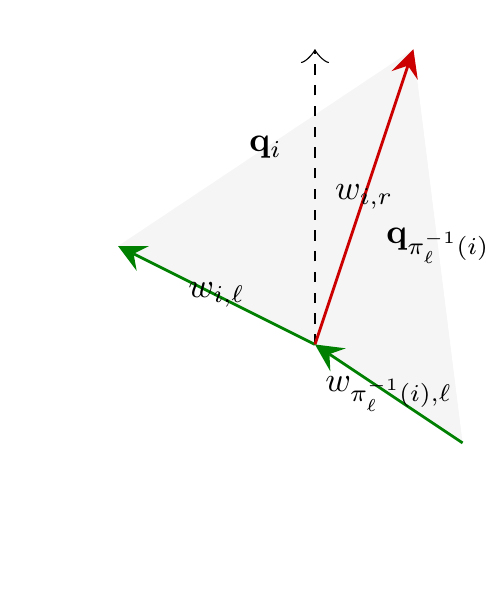}} \hspace{.2cm}
\subfigure[$\sqrt{-1}\, Q$]{\picinput{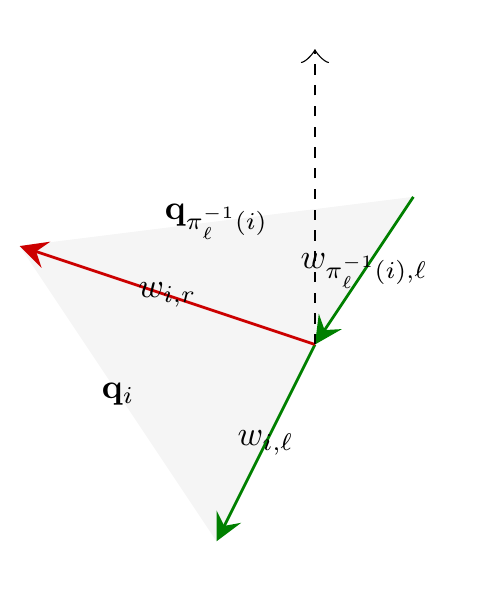}} \hspace{.2cm}
\subfigure[$Q'$]{\picinput{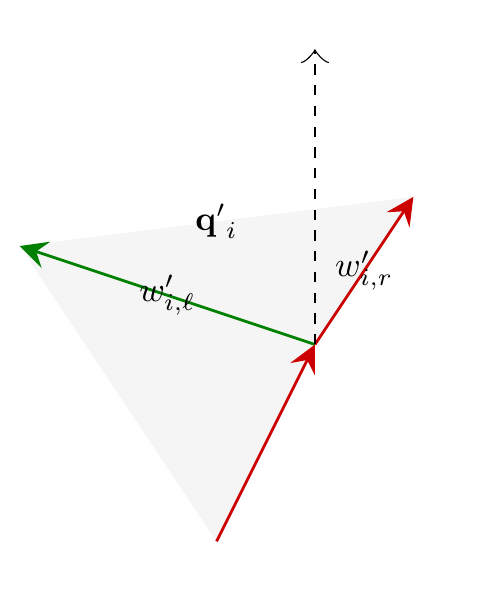}}
 \end{center}
  \caption{a quadrangulation seen from the vertical labelled $i$, its rotation by $\pi/2$ and its new labels}
  \label{fig:rotation}
\end{figure}

\begin{definition}\label{def:R}
The \emph{rotation operator} $R$ sends $Q=(\vpi, \vwedges) \in \QQQ_k$ to  $RQ=(\vpi', \vwedges')$ given by the following formulas:
\[
\pi'_\ell = \pi_\ell\, \pi_r\, \pi_\ell^{-1}, 
\quad
\pi'_r = \pi_\ell^{-1}
\]
and
\[
q'_i = \sqrt{-1}\, q_{\pi_\ell^{-1}(i)}
\quad
w'_{i,\ell} = \sqrt{-1}\, w_{i,r}
\quad
w'_{i,r} = -\sqrt{-1}\, w_{\pi_\ell^{-1}(i),\ell}.
\]
\end{definition}
Let us show that is a well defined operator from $\QQQ_k$ to $\QQQ_k$. It is clear from the geometric description and admissibility of quadrilaterals that $\vwedges'$ is also a vector of wedges and that they satisfy the train-track relations for $\vpi'$.  Hence, if  $\vwedges'= \vlambda' + \sqrt{-1}\vtau$, we have that $\vlambda' \in \Delta_{\vpi'}$ and  $\vtau' \in \Theta_{\vpi'}$. Thus, since $\QQQ_k = \GGG \times \Delta_{\vpi'} \times \Theta_{\vpi'}$,   one only needs to verify that $\vpi'= (\pi_\ell\, \pi_r\, \pi_\ell^{-1}, \pi_\ell^{-1})$ belong to the same graph $\GGG=\GGG(\pi)$. This is proved in~\S~\ref{subsubsec:triangulations} (see Corollary~\ref{cor:same_graph}) and can be shown either from the combinatorial description in~\cite{CassaigneFerencziZamboni} or from the connectedness of $\CCC^{hyp}(k)$ proved in~\cite{DU:II}. 

\smallskip

The operator $R$ is invertible and one can check that the inverse rotation $R^{-1}: \QQQ_k \to \QQQ_k$ is given by $(\vpi',\vwedges') = R^{-1}(\vpi,\vwedges)$ where
\begin{equation} \label{eq:def_Rinverse}
\pi'_\ell = \pi_r^{-1}, \quad \pi'_r = \pi_r \, \pi_\ell \, \pi_r^{-1}
\quad \text{and} \quad
\scw'_{i,l} = \sqrt{-1}\, \scw_{\pi_r^{-1}(i),r}, \quad \scw'_{i,r} =- \sqrt{-1}\, \scw_{i,\ell}.
\end{equation}

Let us remark that $R$ exchanges the role of $\vlambda$ and $\vtau$, more precisely if $(\vpi',\scw') = R (\vpi,\scw)$ then
\begin{equation} \label{eq:rotations_exchange}
  \scw'_{i,\ell} = - \tau_{i,r} + \sqrt{-1}\, \lambda_{i,r}
\quad \text{and} \quad
\scw'_{i,r} = \tau_{\pi_{\ell}^{-1}(i),\ell} - \sqrt{-1}\, \lambda_{\pi_{\ell}^{-1}(i),\ell}.
\end{equation}

So far, for a given admissible quadrilateral $q_i$ in a quadrangulation $Q = (\pi,\vwedges)$ we only considered the forward diagonal $\scw_{i,d} = \scw_{i,d^+} = \scw_{i,l} + \scw_{\pi_\ell(i),r}$ connecting the bottom vertex to the top one.

\begin{definition}
  Let $q_i$ be a quadrilateral in a quadrangulation $Q=(\vpi,\vwedges)$. The \emph{backward diagonal} $\scw_{i,d^-}$ of $q$ is the diagonal joining the left vertex to the right vertex of $q_i$.
\end{definition}

The definition is given so that the forward diagonal $\scw_{i,d^+}' $ of  the quadrilateral $q'_i$ in $Q'=RQ$ is obtained by rotating the backward diagonal of $q_{\pi_{\ell}^{-1}(i)}$, that is
\begin{equation} \label{diagonal:eq}
  \scw'_{i,d^+} = \sqrt{-1}\, \scw_{\pi_{\ell}^{-1}(i),d^-} = \sqrt{-1}\, (\scw_{\pi_{\ell}^{-1}(i),r} - \scw_{\pi_{\ell}^{-1}(i),\ell}). 
\end{equation}

It is clear geometrically that left (right) staircases becomes right (left) staircases after rotation. More precisely, if $c$ is a left cycle of $\vpi = (\pi_\ell,\pi_r)$, then it is also a  right cycle of $ \vpi' = (\pi_\ell \, \pi_r \, \pi_\ell^{-1}, \pi_\ell^{-1} )$. On the other hand, if $c = \{ i_1, \dots, i_n \} $ is a right cycle of $\vpi = (\pi_\ell,\pi_r)$,  then $ \pi_\ell \,c := \{ \pi_\ell (i_1), \dots, \pi_\ell(i_n) \}$ is a  left cycle of $ \vpi' = (\pi_\ell \, \pi_r\, \pi_\ell^{-1}, \pi_\ell^{-1})$. Thus, let us define 
\begin{equation}\label{def:cprime}
c' := \left\{ \begin{array}{ll}
 c & \text{if $c$ is a cycle of $\pi_\ell$,}  \\
  \pi_\ell \, c   & \text{if $c$ is a cycle of $\pi_r$.}  
\end{array}\right.
\end{equation}
Then, if $S_c$ is a 
 right (resp.~left) staircase  for $Q$, it corresponds to the left (resp.~right) staircase $S_c'$ for $Q'=RQ$ under the action of $R$,  that is, 
$S_{c'}$ is the union of the rotated quadrilaterals $\sqrt{-1} q$, $q \in Q$. 


Recall that $\Delta_{\vpi,c}$ is defined  so that we can perform a staircase move in $S_c$ exactly when $\vlambda \in \Delta_{\vpi,c}$, i.e. $S_c$ is well slanted (see~\eqref{eq:deltac_def}). 
Similarly, we define the set of parameters such that the rotated staircase $ S_{c'}$ for $Q'=R Q$ (where $c'$ is given by~\eqref{def:cprime}) is well slanted   so that we can perform a move in $Q'$. 
if $c' $  is a cycle in $\vpi'$, 
 It is clear that this set depends only on $\vtau$ since  the forward diagonal $\scw'_{i,d^+}$ of the quadrilateral $q'_i$  is obtained by rotating the backward diagonal of the  quadrilateral $q_{\pi_\ell^{-1}(i)}$ of $Q$ and this exchanges the role of lengths and suspension datas (see Equations~\eqref{diagonal:eq} and \eqref{eq:rotations_exchange}). 
 Thus this set of parameters is $\{\vpi\} \times \Delta_\vpi \times \Theta_{\vpi,c}$ where
\[
\Theta_{\vpi,c} = \left\{
\begin{array}{ll}
  \{\vtau \in \Theta_\vpi;\ \tau_{i,d^-} = \tau_{i,r} - \tau_{i,\ell} < 0, \quad i \in c \}  & \text{if $c$ is a left cycle,} \\
  \{\vtau \in \Theta_\vpi;\ \tau_{i,d^-} = \tau_{i,r} - \tau_{i,\ell} > 0, \quad i \in c \} & \text{if $c$ is right cycle.} \\
\end{array}
  \right. 
\]

From the definitions and the exchange in the role of lengths and suspension datas (see Equation \eqref{eq:rotations_exchange}), we also get the following result. 
\begin{lemma} \label{lem:R_bijection}
  Let $\vpi = (\pi_\ell,\pi_r)$ be a combinatorial datum of a quadrangulation $Q \in \QQQ_k$ and let  $ \vpi' = (\pi_\ell \, \pi_r \, \pi_\ell^{-1}, \pi_\ell^{-1} )$  be the combinatorial datum of $Q'=R Q$.
  Let $c$ be a cycle of $\vpi$ and let $c'$ be the corresponding cycle in $\vpi'$ given by~\eqref{def:cprime}. Then
\begin{itemize}
\item[(i)] $R$ maps $\{\vpi\} \times \Delta_\vpi \times \Theta_{\vpi,c}$ bijectively onto $\{\vpi'\} \times \Delta_{\vpi',c'} \times \Theta_{\vpi'}$,
\item[(ii)] $R$ maps $\{\vpi\} \times \Delta_{\vpi,c} \times \Theta_\vpi$ bijectively onto $\{\vpi'\} \times \Delta_{\vpi'} \times \Theta_{\vpi',c'}$.
\end{itemize}
\end{lemma}

\begin{theorem}[self-duality] \label{thm:Markov_self_dual}
Let $\pi$ be a permutation, let $c$ be a cycle of $\pi$. Then
\begin{equation}\label{eq:bijection}
\widehat{m}_{\vpi,c}: \{\vpi\} \times \Delta_{\vpi,c} \times \Theta_\vpi \rightarrow \{c \cdot \vpi\} \times \Delta_{c \cdot \vpi} \times \Theta_{c \cdot \vpi,c}
\end{equation}
is a bijection. Moreover, if $c$ is a cycle of $\pi_\ell$ the inverse is given by
\begin{equation} \label{eq:self_duality}
\widehat{m}_{\vpi,c}^{-1} = R^{-1} \circ \widehat{m}_{\vpi',c'} \circ R,
\end{equation}
where $\vpi' = R \cdot \vpi$ and $c'$ is given by~\eqref{def:cprime}.
\end{theorem}
The proof of the Theorem, which follows from the definitions and the Lemma, is given here below. 
Equation~\eqref{eq:self_duality} is a formulation of the \emph{self-duality property} of staircase moves
(we refer for example to Schwheiger~\cite{Schweiger} for the definition of duality).  
Geometrically it simply means that the inverse of a left (respectively right) staircase move is given by a right (respectively left) staircase move in the rotated staircase. 

Let us explain in which sense the bijection in~\eqref{eq:bijection} shows that there is a \emph{loss of memory} phenomenon (or Markov property). The space of quadrangulations $\QQQ_k$ projects on  the corresponding space of bipartite IETs, which is given by $\{(\vpi,\vlambda); \, \vpi \in \GGG,\ \vlambda \in \Delta_\vpi\}$.  Let $m_{\vpi,c}$ be the projection of $\widehat{m}_{\vpi,c}$ on the bipartite IETs space. In other words,  $m_{\vpi,c}$ is the map defined on $\{ \vpi \} \times \Delta_{\vpi, c}$ which sends $(\vpi, \vlambda)$ to $m_{\vpi,c} (\vpi, \vlambda) = (c \cdot \vpi, A_{\vpi,c}\, \vlambda)$.
 
\begin{corollary}[Markov property]
  The map $m_{\vpi,c}: \{\vpi\} \times \Delta_{\vpi,c} \to \{c \cdot \vpi'\} \times \Delta_{c \cdot \vpi'}$ is a bijection.
\end{corollary}
The corollary shows that given any (oriented) path in the graph $\GGG$, which corresponds to a sequence of staircase moves, there exists a quadrangulation $Q = (\vpi,\vwedges)$ from which we can apply this sequence of moves. In this sense, staircase moves have a Markov structure. For the greedy algorithm, one can use the sets $\Delta_{\vpi,c}$ to define  a natural Markov partition on $\QQQ_k$ that is a finite partition $\mathcal{P}$ of $\QQQ_k$ so that the image of each atom of $\mathcal{P}$ is  union of atoms.
As shown in \cite{Delecroix}, this is not the case for the left/right algorithm for which we should keep in memory one step of the history.

\begin{proof}[Proof of Theorem~\ref{thm:Markov_self_dual}]
By~\eqref{eq:rotations_exchange} and by definition of $\Delta_{\pi,c}$ and $\Theta_{c \cdot \pi,c}$ it is clear that the image of the map $\widehat{m}_{\vpi,c}$ is $\{c \cdot \vpi\} \times \Delta_{c \cdot \vpi} \times \Theta_{c \cdot \vpi,c}$. Using also Lemma~\ref{lem:R_bijection}, it follows that all compositions in the statement make sense.

Let $c$ be a cycle of $\pi_r$ and  $c'$ be the cycle associated to $c$ by~\eqref{def:cprime}. Let us denote by $ (\vpi', \vwedges')= R (\vpi,\vwedges)$,  $(\pi'',\vwedges'') = \widehat{m}_{ \vpi', c'} \, (\vpi',\vwedges')$ and $ (\vpi''', \vwedges''')= R^{-1} (\vpi'',\vwedges'')$.  We first compute $\pi', \pi'' $ and $\pi'''$ to get the action of the composition $R^{-1} \widehat{m}_{\pi',c'} R$ on combinatorial data.  By formulas \eqref{eq:def_Rinverse} for $R$, we have that 
\[
\pi'_\ell = \pi_\ell \pi_r \pi_\ell^{-1} \quad \text{and} \quad \pi'_r = \pi_\ell^{-1}.
\]  
Now recall that $c'$ is associated to a left slanted staircase in the rotated quadrangulation $Q'=R Q$, so $\widehat{m}_{ \vpi', c'}$ is a left staircase move. Thus, by definition of a left staircase move we get that $\pi'' = c' \cdot \pi'$ is given by
\[
\pi''_\ell = \pi'_\ell = \pi_\ell\, \pi_r\,  \pi_\ell^{-1}
\quad \text{and} \quad
\pi''_r(i) = \left\{ \begin{array}{ll}
  \pi'_r \, \pi'_\ell(i) & \text{if $i \in c'$,} \\
  \pi'_r(i) & \text{otherwise}
\end{array} \right.
= \left\{ \begin{array}{ll}
  \pi_r \, \pi_\ell^{-1} & \text{if $i \in c'$,} \\
  \pi_\ell^{-1} & \text{otherwise.}
\end{array} \right.
\]
Finally, by formulas~\eqref{eq:def_Rinverse} for $R^{-1}$, we have that 
\[
\pi'''_\ell = (\pi''_r)^{-1} = \left\{ \begin{array}{ll}
  \pi_\ell \pi_r^{-1}(i) & \text{if $\pi_\ell \, \pi_r^{-1}(i) \in c'$,} \\
  \pi_\ell(i) & \text{otherwise}
\end{array} \right.
\quad \text{and} \quad
\pi'''_r = \pi''_r \pi''_\ell (\pi''_r)^{-1}.
\]
By the definition of $c'$, te condition $\pi_\ell\, \pi_r^{-1}(i) \in c'$ is equivalent to $\pi_r^{-1}(i)  \in c$ and since $c $ is a right cycle, it is also equivalent to $i \in c$.
Now, to compute the expression of $\pi'''_r$, let us consider separately the cases  $i \in c$ and $i \notin c$. As shown above, if  $i \in c$ we also have  $\pi_\ell\, \pi_r^{-1}(i) \in c'$ and thus $(\pi''_r)^{-1}(i) = \pi_\ell \pi_r^{-1}(i)$. Hence 
$\pi''_\ell (\pi''_r)^{-1}(i) = \pi_\ell (i)$. Since, when $c$ is a right cycle, $c'= \pi_\ell \, c$ we then have that $\pi_\ell (i) \in c'$ and hence, by the above expression for $\pi_r''$ we get
\[
\pi''_r \, \pi''_\ell \, (\pi''_r)^{-1}(i) = \pi_r(i).
\]
Now consider the case $i \notin c$.  We get
$\pi''_\ell (\pi''_r)^{-1}(i) = \pi_\ell \pi_r(i)$.  Now $c$ and its complement are stable under $\pi_r$ and hence, $\pi_\ell \pi_r(i) \not\in c'$. Hence,  we  obtain
\[
\pi''_r \pi''_\ell (\pi''_r)^{-1}(i) = \pi_r(i).
\]
Thus, in both cases  $\pi'''_r = \pi_r$. One can verify from the formulas for the combinatorial datum of a right staricase move that  $c \cdot \pi''' = \pi$. This show that  $\pi'''$ is the combinatorial datum of the inverse staircase move in $S_c$. 

Let us now compute the wedges $\scw'$, $\scw''$ and $\scw'''$. 
From the formulas for $R$ and a left staircase move in $S_{c'}$ we get 
$$
\begin{array}{ll}
\scw'_{i,\ell} = \sqrt{-1}\, \scw_{i,r} &\scw'_{i,r} = -\sqrt{-1}\, \scw_{\pi_\ell^{-1}(i),\ell} \\
\scw''_{i,\ell} = \scw'_{i,\ell} = \sqrt{-1}\, \scw_{i,r} &\scw''_{i,r} = \left\{
\begin{array}{ll}
  \scw'_{i,\ell} + \scw'_{\pi_\ell'(i),r} 
  & \text{if $i \in c'$,} \\
  \scw'_{i,r} & \text{otherwise.}
\end{array} \right.
\end{array}
$$
Thus, since $\pi_\ell'= \pi_\ell \, \pi_r \, \pi_\ell^{-1}$, combining the above expressions we get that
$$
\scw''_{i,r} = \left\{ \begin{array}{ll} \sqrt{-1}\,   \scw_{i,r} - \sqrt{-1}\,  \scw_{\pi_r\, \pi_\ell^{-1}(i),\ell}  
  & \text{if $i \in c'$,} \\
   -\sqrt{-1}\, \scw_{\pi_\ell^{-1}(i),\ell}  & \text{otherwise.}
\end{array} \right.
$$
From the formula for $R^{-1}$ we  then get 
$$ \scw'''_{i,\ell} = \sqrt{-1} \scw''_{(\pi_r'')^{-1}(i),r} , \qquad \scw'''_{i,r} = -\sqrt{-1} \scw''_{i, \ell}  =  -\sqrt{-1} \,(\sqrt{-1}\, \scw_{i,r} ) = \scw_{i,r} . $$
To compute $\scw'''_{i,\ell}$, let us use the expression computed above for $(\pi_r'')^{-1}$ and consider separately two cases. If $i \in c$, then $(\pi_r'')^{-1}(i) = \pi_\ell \, \pi_r^{-1}(i)$ which belongs to $c'$ (since $c$ is invariant under $\pi_r$ and by definition of $c'$). Thus,  for $i \in c$ we get that 
\[
\scw'''_{i,\ell} =\sqrt{-1}\, \scw''_{ \pi_\ell \, \pi_r^{-1}(i),r} =  
 -   \scw_{ \pi_\ell \, \pi_r^{-1}(i),r} + \scw_{ \pi_r \, \pi_\ell^{-1}   \pi_\ell \, \pi_r^{-1} (i), \ell } = \scw_{i, \ell}  -   \scw_{ \pi_\ell \, \pi_r^{-1}(i),r}.
\]
On the other hand, if $i \notin c$,  $(\pi_r'')^{-1}(i)= \pi_\ell(i)$, which is not in $c'$, thus
\[
\scw'''_{i,\ell} = \sqrt{-1} \scw''_{\pi_\ell(i),r} =  \scw_{i,r} . 
\]
One can check that this is indeed the expression for the wedges of the inverse of the staircase move in $S_c$. 
The case when $c$ is a cycle of $\pi_\ell$ is analogous.

\end{proof}

\section{Existence of quadrangulations and staircase moves}\label{sec:induction}
In this section we prove the existence of quadrangulations for any surface that belongs to an hyperelliptic component of a stratum  (Theorem~\ref{thm:existence_of_quadrangulation}) and the existence of well slanted staircases for any of these quadrangulations (Theorem~\ref{thm:existence_of_staircase_move}). We first start with a precise definition of  hyperelliptic components of strata in terms of double cover of quadratic differentials.

\subsection{Quadrangulations in hyperelliptic components} \label{subsec:hyperelliptic}
We have already seen in~\S\ref{subsubsec:intro_quadrangulations} that translation surfaces can be constructed by gluing polygons or equivalently by assigning a non-zero Abelian differential on a Riemann surface. We first describe a more general construction which produces Riemann surfaces with quadratic differentials. We then define orientation covers of quadratic differentials, that are a particular case of translation surfaces. Then we define hyperelliptic components as the set of orientation covers of quadratic differentials that belong to some fixed stratum.

\subsubsection{Hyperelliptic components of strata of translation surfaces}\label{sec:hyperellipticdef}
While a translation surface is obtained by gluing polygons by translations, a quadratic differential can be obtained 
by gluing polygons by translations and rotation by $180$ degrees.
 
Let $P_i \subset \CC$ be a collection of polygons whose edges are identified into pairs such that:
\begin{enumerate}
  \item either the two edges in the pair are parallel with opposite normal vector (with respect to their polygons) and we identify the two edges by the unique translation that sends one to the other,
  \item or the two edges are parallel but have the same normal vector and we identify them under the unique rotation by 180 degrees (ie a map of the form $z \mapsto -z+c$) that maps one edge to the other.
\end{enumerate}
The quotient of $\cup P_i$ by the  identifications of the edges is a surface $X$ which carries the structure of a Riemman surface with a quadratic differential $q$ (which is induced from the form $dz^2$ on the polygons). If in this construction all pairs are of the first form then the construction reduces to the one described in~\S\ref{subsubsec:intro_quadrangulations} and $X$ is a translation surface, or, equivalently, a Riemann surface  $X$ which carries an Abelian differential $\omega$. Let $\Sigma \subset X$ denote as before the singularity set corresponding to the images of the vertices of the polygons.  A quadratic differentials has  conical singularities with angles of the form $k \pi$ with $k$ integer (instead of $2\pi k$ as in the case of Abelian differentials). Moreover, while an Abelian differential determines on $X\backslash \Sigma$ a  well defined notion of lines in direction $\theta \in S^1$, a quadratic differential only determines a notion of (non-oriented) lines in direction  $\theta \in \PP^1 \RR$. 

We define two quadratic differentials $(X,q)$ and $(X',q')$ to be isomorphic if there exists an homeomorphism $X \rightarrow X'$ such that $q = f^* q'$. We can also define this notion of isomorphism as cut and paste operations on polygons, similarly to the definition given in~\S\ref{subsubsec:intro_quadrangulations} for translation surfaces. 
We denote by $\QQQ(k_1-2,\ldots,k_n-2)$ the equivalence class of quadratic differentials  with conical singularities of angles $\pi k_1 , \ldots, \pi k_n$. The number $k_i-2$ correspond to the degree of the quadratic differential as $q$ can be written locally as $z^{k_i-2} dz^2$ around a singularity with conical angle $\pi k_i$. Note that we have the topological restriction that $\sum_{i=1}^n k_i  = 4g-4+2n$ where $g$ is the genus of the surface. If there are $m_i$ singularities with total angle $\pi k_i$ we use the notation $\QQQ((k_1-2)^{m_1},\ldots,(k_n-2)^{m_n} )$. 

Let $(X,q)$ be a quadratic differential. We associate to $q$ its canonical \emph{orientation cover}: it is the Abelian differential $(\tilde{X},\omega)$, unique up to isomorphism, such that there exists a degree $2$ map $\pi: \tilde{X} \rightarrow X$ and such that $\pi^* q = \omega^2$. The stratum in which $\tilde{X}$ belongs is easily computed as follows: each singularity of angle $\pi k_i$ with $k_i$ even is not ramified and gives two singularities on $\tilde{X}$ of angle $\pi k_i$; each singularity of angle $k_i$ with $k_i$ odd is ramified and gives a singularity on $\tilde{X}$ of angle $2 \pi k_i$. As an example, the orientation covers of surfaces in $\QQQ(2,3^2)$ belong to $\HHH(1^2,4)$.
Because a degree two cover is always normal, an orientation cover always comes with an involution whose quotient is the corresponding quadratic differential.

When a quadratic differential varies in its stratum, its orientation cover varies in a connected component of the corresponding stratum of Abelian differentials. When the  stratum of quadratic differentials is a sphere (i.e. a stratum of the form $\QQQ(k_1-2,\ldots,k_n-2)$ with $k_1 + \ldots + k_n = 2n-4$) such locus is called a \emph{hyperelliptic locus}. In this case the involution is an \emph{hyperelliptic involution}.
The points of an hyperelliptic surface which are fixed by the hyperelliptic involution are called \emph{Weierstrass points}. They might be conical singularities or regular points. In the latter case, they projects down to conical singularities of angle $\pi$ on the sphere that are called \emph{poles} (because they correspond to singularities of the form $z^{-1} dz^2$ for the quadratic differential). Because of Hurwitz formula, a hyperelliptic surface of genus $g$ has $2g+2$ Weirstrass points.

In most cases, \emph{hyperelliptic loci} have positive codimension in the corresponding stratum of Abelian differentials, but an infinite family of hyperelliptic loci have full dimension and form connected components.
\begin{theorem}[\cite{KontsevichZorich03}, section~2.1 p.5--7] \label{thm:KZhyperelliptic}
In each stratum $\HHH(2g-2)$ (respectively $\HHH(g-1,g-1)$) the hyperelliptic locus built as the orientation cover of quadratic differentials in $\QQQ(k-2,-1^{k+2})$ for $k=2g-1$ (resp. $k=2g$) forms a connected component. These are the only hypelliptic loci that form connected components of stratum.
\end{theorem}
Recall from the Introduction that we denote by $\CCC^{hyp}(k)$ the hyperelliptic component of $\HHH(k-1)$ if $k$ is odd or of $\HHH(k/2-1,k/2-1)$ if $k$ is even. Surfaces in  $\CCC^{hyp}(k)$ have  total conical angle $2k\pi$ and hence any quadrangulation on them is made by $k$ quadrilaterals.


\subsubsection{Two geometric results in hyperelliptic components}
Using the description of surfaces in an hyperelliptic component $\CCC^{hyp}(k)$ as double covers of quadratic differentials in the  stratum $\QQQ(k-2,-1^{k+2})$ of quadratic differentials, we prove two important results. The first one shows that a quadrangulation of a surface in a hyperelliptic component of a stratum is always preserved by the hyperelliptic involution. The second one is a cut and paste construction that will be used in some of the following proofs. 
\begin{lemma} \label{lem:hyperelliptic_staircase}
Let $Q$ be a quadrangulation of a surface $X$ in a hyperelliptic component $\CCC^{hyp}(k)$.
\begin{enumerate}
\item Each staircase for $Q$ is fixed (as a set) by the hyperelliptic involution of $X$.
\item If $q \in Q$ is a quadrilateral  then its image under the hyperelliptic involution is another quadrilateral that belongs to the same left and right staircases for $Q$ to which $q$ belongs.
  \end{enumerate}
\end{lemma}

\begin{proof}
Let us first show that a quadrangulation can be continuously deformed in such way that staircases become metric cylinders. We then prove the result when all staircases are metric cylinders. Finally, we show that the property for the latter is preserved under deformation in the component $\CCC^{hyp}(k)$ and hence holds for all surfaces in that component.

Let $X$ be a surface in $\CCC^{hyp}(k)$ and $Q = Q^{(0)}$ be a quadrangulation of $X$. Let us label its quadrilaterals and denote them by $q_1, \dots, q_k$.   
Let $\vpi = (\pi_\ell,\pi_r)$ be its combinatorial datum and $\vwedges = \vwedges^{(0)}$ its length datum, so that $Q = (\vpi, \vwedges)$.  
Let us introduce the length datum $\scw^{(1)}_{i,\ell} = -1 + \sqrt{-1}$ and $\scw^{(1)}_{i,r} =  1 + \sqrt{-1}$ for all $i = 1,\ldots,k$. Remark that the quadrangulation $Q^{(1)} = (\vpi, \vwedges^{(1)} )$ is a quadrangulation made by squares whose sides have length $\sqrt{2}$ and hence  staircases are metric cylinders. 
Consider the straight line in the parameter space of length data that goes from $\vwedges^{(0)}$ to $\vwedges^{(1)}$ given by $\vwedges^{(t)} = (1-t) \vwedges^{(0)} + t \vwedges^{(1)}$. 
Since both the train-track relations and the positivity conditions ($\lambda_{i,\ell} < 0 < \lambda_{i,r}$ and $\tau_{i,\ell}, \tau_{i,r} > 0$) are convex, $\vwedges^{(t)}$ is a valid length datum for $\vpi$ for all $0 \leq t \leq 1$. We hence get a path of quadrangulations $Q^{(t)} = (\vpi, \vwedges^{(t)})$ and a continuous path $X^{(t)}$ of translation surfaces.

We first claim that the hyperelliptic involution of $X^{(1)}$ maps each quadrilateral $q^{(1)}_i$ in $Q^{(1)}$ to another quadrilateral of $Q^{(1)}$ reversing the orientation. Indeed, since  $X^{(1)}$  is made by squares with side length $\sqrt{2}$,   the saddle connections of length $\sqrt{2}$ on $X^{(1)}$ are exactly the sides of $Q^{(1)}$. Since the hyperelliptic involution preserves the flat metric of $X^{(1)}$, it must preserve this set. Thus,  each quadrilateral $q^{(1)}_i$ of $Q^{(1)}$ is sent,  reversing the orientation,  to another quadrilateral $q^{(1)}_{\iota (i)}$ where $\iota$ is an involution of $\{1,\ldots,k\}$.
  
Now we claim that the map $\iota$ actually preserves staircases, that is $i$ and $\iota(i)$ belongs to the same cycles of both $\pi_\ell$ and $\pi_r$. 
Let consider a surface $X$ in $\QQQ(k-2,-1^{k+2})$ and a maximal cylinder $C$ in it. Because it is a sphere, each closed curve separates the surface into two connected components. Now the circumference of a cylinder is a closed curve so the zero of degree $k$ in $Y$ belongs to only one side of the cylinder. The other side contains only poles and hence it has to contain two poles. If we lift such cylinder to the corresponding hyperelliptic component it consists of one cylinder which contains two Weirstrass points in its middle. This proves that any cylinder in any surface that belongs to the hyperelliptic component is fixed (as a set) by the hyperelliptic involution. Hence the conclusion of the Lemma holds for the quadrangulation $Q^{(1)}$ of $X^{(1)}$.

Now it remains to deduce that the hyperelliptic involution on $X^{(t)}$ for $t \in [0,1]$ also sends the quadrilateral $i$ to the quadrilateral $q_{\iota(i)}$ reversing the orientation. Heuristically, this is because the quadrangulations $Q^{(t)}$ of $X^{(t)}$, $t \in [0,1]$, are obtained by a continuous deformation and  the hyperelliptic involution is continuous  on $\CCC^{hyp}(k)$. We warn the reader that it makes no sense to speak of continuity of the hyperelliptic involution on $\CCC^{hyp}(k)$. We need to consider the so called universal curve on $\CCC^{hyp}(k)$, that is the set of equivalence class $(X,x)$ where $X \in \CCC^{hyp}(k)$ and $x \in X$. This universal curve is also a connected component of a stratum (with a point with conical angle $2\pi$). The hyperelliptic involution acts on the universal curve by action on the second coordinate and is continuous on it. As it is an isometry, the hyperelliptic involution sends saddle connections to saddle connections. We would like to argue that the hyperelliptic involution is continuous on the set of saddle connection, but the problem is that the map $X \mapsto \xQuad(X)$ which to a surface associate its set of saddle connections (seen as a discrete subset of $(\RR \times \RR_+)^k$) is not continuous, as saddle connections may appear or disappear.
Nevertheless, if $X^{(t)}$ is a continuous path of surfaces and $\gamma^{(t)}: [0,1] \rightarrow X^{(t)}$ and $\eta^{(t)}: [0,1] \rightarrow X^{(t)}$ are such that
\begin{itemize}
\item the maps $(s,t) \mapsto \gamma^{(t)}(s)$ and $(s,t) \mapsto \eta^{(t)}(s)$ are continuous from $[0,1] \times [0,1]$ to $X$,
\item for each $t$, $\gamma^{(t)}$ and $\eta^{(t)}$ are saddle connections parametrized with constant speed,
\item at time $t=0$, the saddle connections coincide, i.e. we have $\gamma^{(0)}(s) = \iota \circ \eta^{(0)}(s)$,
\end{itemize}
then for all time $t \in [0,1]$, we have $\gamma^{(t)}(s) = \iota \circ \eta^{(t)}(s)$. This simply follows from a continuity argument.
We may apply this to our saddle connections that form the sides of our quadrangulations, namely $\gamma^{(t)}(s) = s \sc$ where $\sc$ is thought as an element of $\CC$.
\end{proof}

\begin{lemma} \label{lem:cut_hyp}
Let $X$ be a surface in a hyperelliptic component $\CCC^{hyp}(k)$ and let $s: X \rightarrow X$ be the hyperelliptic involution. 
Let $\gamma$ be a saddle connection in $X$ that is not fixed by the hyperelliptic involution.
Then $X \backslash (\Sigma \cup \gamma \cup s \gamma)$ has two connected components both of them having $\gamma$ and $s \gamma$ on their boundary.
Let $X_1$ and $X_2$ be obtained from these two connected components by identifying $\gamma$ and $s \gamma$ by translation.
Then $X_1$ and $X_2$ are (non empty) translation surfaces in hyperelliptic components. Furthermore, if $k_1\geq 1$ and $k_2\geq 1$ are such that  $X_1 \in \CCC^{hyp}(k_1)$  and  $X_2\in \CCC^{hyp}(k_2)$,  we have $k = k_1 + k_2$.
\end{lemma}

\begin{proof}
  Let $Y$ be the quotient of $X$ under the hyperelliptic involution. The image of $\gamma$ (which is also the image of $s \gamma$) in $Y$ is a segment that does not contain a pole in its interior (this is because a saddle connection in $X$ is preserved under the hyperelliptic involution if and only if it contains a Weirstrass point in its interior). We obtain a closed curve on the sphere which is a loop (both ends are the zero of the quadratic differential) and hence separates the sphere into two components whose boundaries each consists of a copy of the segment image of $\gamma$. Let us now add a pole in the middle point of each segment, hence defining a new quadratic differential on each surface. Taking the double covers of these new quadratic differentials we obtain two surfaces $X_1$ and $X_2$ as in the statement. The relation  $k = k_1 + k_2$ follows from computing total conical angles. 
\end{proof}

\subsubsection{Triangulations on the sphere and Ferenczi-Zamboni trees of relations} \label{subsubsec:triangulations}
From Lemma~\ref{lem:hyperelliptic_staircase}, we know that a quadrangulation of a surface that belongs to a hyperelliptic component of a stratum is necessarily fixed by the hyperelliptic involution of the surface. In particular, it makes sense to consider the quotient of the quadrangulation on the sphere. We see in this section that this quotient is naturally a triangulation that it is intimately related to the so called \emph{trees of relations} that appear in work by Ferenczi and Zamboni, see \cite{FerencziZamboni-struct}.

\smallskip
Let $q$ be a quadratic differential on the sphere $\CC \PP^1$ which belongs to $\QQQ(k-2,-1^{k+2})$ and let $z_0$ denotes the point of $\CC \PP^1$ at which $q$ has the zero of degree $k-2$.
We call a \emph{triangle} on $(\CC \PP^1, q)$ an open embedded triangle in $(\CC \PP^1, q)$ whose boundary consists of saddle connections between $z_0$ and itself that may pass through one pole. Notice that, since the conical angle at a pole is $\pi$, an edge which passes through a pole actually consists of two copies of the same segment. A \emph{triangulation} of $(\CC \PP^1, q)$ is a set of triangles on $q$ such that their interiors have empty intersection and their union is the whole $\CC \PP^1$. An example of a triangulation is shown in Figure~\ref{subfig:triangulation_Q_3p7}.

\begin{figure}[!ht]
\begin{center}
\subfigure[a quadrangulation in $\CCC^{hyp}(5)$\label{subfig:quadrangulation_H_4}]{\picinput{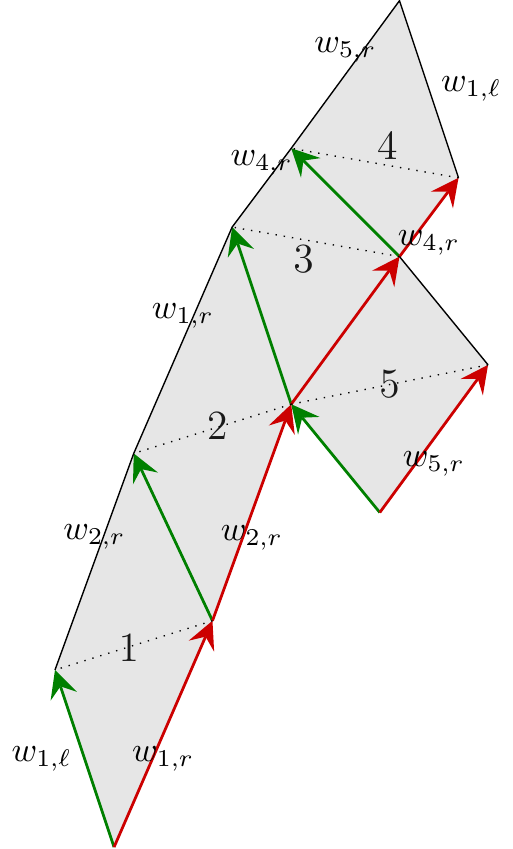}} \hspace{.5cm}
\subfigure[its quotient in $\QQQ(3,-1^7)$\label{subfig:triangulation_Q_3p7}]{\picinput{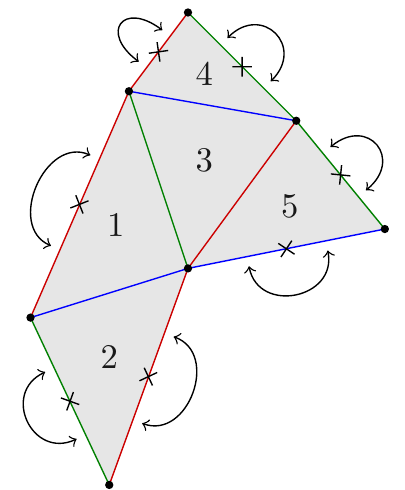}} \hspace{1cm}
\subfigure[its tree of relations\label{subfig:tree_of_relations}]{\picinput{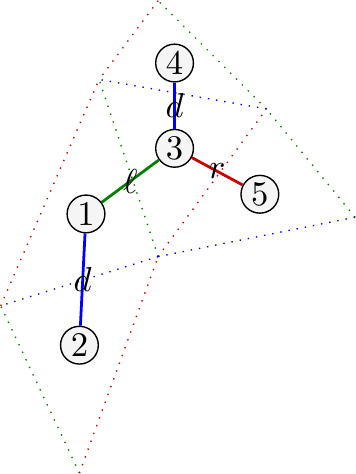}}
\end{center}
\caption{from a quadrangulation of a surface in $\CCC^{hyp}(5)$ to the tree of relations}
\label{fig:wedge-quotient}
\end{figure}

Given a triangulation $T$ on the sphere, we canonically associate its \emph{dual graph} $G_T$. The vertices $v_t$ are the triangles $t \in T$ and we join two vertices $v_t$ and $v_{t'}$ by an edge if the corresponding triangles $t$ and $t'$ share an edge which has no pole on it. An example of such graph is given in Figure~\ref{subfig:tree_of_relations}.

\begin{lemma} \label{lem:T_is_a_tree}
  Let $G_T$ be the dual graph associated to the triangulation $T$ of a quadratic differential $(\CC \PP^1, q)$ in a stratum $\QQQ(k-2,-1^{k+2})$. Then $G_T$ is a tree. 
 \end{lemma}

\begin{proof}
  The connectedness of $G_T$ comes from the connectedness of $\CC \PP^1$.
  Hence, to prove that it is a tree it is enough to show that the number of edges of $G_T$ is its number of vertices minus one.    By definition, the vertices are the triangles of $T$ so there are $k$ of them.
  Now, because it is a triangulation and there are $k+2$ poles the number of edges is $(3k - (k+2)) / 2 = k-1$.
\end{proof}

Recall that quadrangulations of surfaces in $\CCC^{hyp}(k)$ have by definition the additional property that the quadrilaterals are admissible (recall Definition~\ref{def:admissible}). In the quotient, we can see this property as a compatibility condition on the triangles.
More precisely, in any triangle there is exactly one vertex such that the vertical segment emanating from that vertex is contained in the triangle (as illustrated in Figure~\ref{subfig:label_of_triangles}, see also Figure~\ref{subfig:triangulation_Q_3p7} for an example). We can then assign labels to each side of a triangle as follows. Let us consider the unique vertical from a vertex of $t$ which is contained in $t$ and orient it so that it starts from the vertex. We label $d$ the side opposite to the vertex from which the vertical starts, which is the unique side crossed by the considered vertical. We then label $\ell$ and $r$ the other two sides of $t$ (which form a wedge which contains the considered vertical), so that rotating counterclockwise around the vertex one sees first the side labelled $r$, then the vertical, then the side labelled $\ell$, as shown in Figure~\ref{subfig:label_of_triangles}.

Let $Q$ be a quadrangulation of a surface in $\CCC^{hyp}(k)$ and $T$ the triangulation obtained taking its quotient by the hyperelliptic involution. Then the admissibility of the quadrilaterals in $Q$ implies that that pairs of sides of triangles of $T$ which are identified carry the same label, see Figure~\ref{subfig:coherent_triangle} and~\ref{subfig:noncoherent_triangle}. Thus, we can assign labels in $\{\ell,r,d\}$ to the edges of the dual tree $G_T$ associated to the triangulation $T$, by assigning to each edge of $G_T$ the common label of the dual pair of identified triangle edges (see the example in Figure~\ref{subfig:tree_of_relations}). This labelled tree is called the \emph{tree of relations} in~\cite{FerencziZamboni-struct} (we warn the reader that $\ell$, $r$ and $d$ are respectively replaced in~\cite{FerencziZamboni-struct} by $\hat{+}$, $\hat{-}$ and $\hat{=}$). As shown in~\cite{FerencziZamboni-struct}, the tree encodes indeed the train-track relations for the the length datum $\vwedges$ of $Q$ as follows. If the edge of the tree connecting the vertices $i$ to $j $ carries the label $r$ (resp.~$l$), the wedges of the quadrilaterals $q_i$ and $q_j$ obtained by double covers of the triangles dual to the vertices $i$ and $j$ are such that $\scw_{i,r} = \scw_{j,r}$ (resp.~$\scw_{i,\ell} = \scw_{j,\ell}$). If the edge connecting $i$ to $j$ carries the label $d$, the quadrilaterals $q_i$ and $q_j$ have parellel isometric diagonals, that is $\scw_{i,d} = \scw_{j,d}$. One can show that this set of equations is equivalent to the set of train-track relations $\scw_{i,\ell} + \scw_{\pi_\ell(i),r} = \scw_{i,r} + \scw_{\pi_r(i),\ell}$ for $1\leq i \leq k$ (a sketch is given in~\cite{FerencziZamboni-struct}).

\begin{figure}[!ht]
\begin{center}
  \subfigure[Labels on triangles\label{subfig:label_of_triangles}]{\picinput{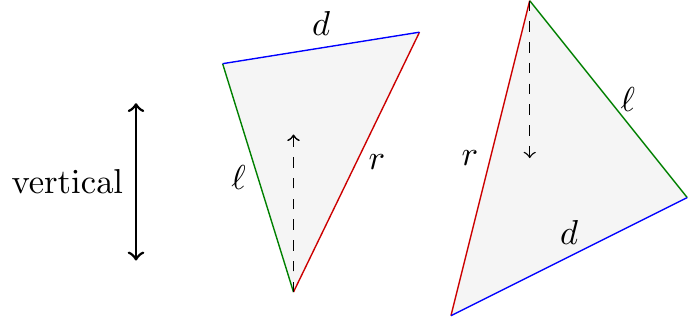}} \hspace{1cm}
  \subfigure[admissible gluing\label{subfig:coherent_triangle}]{\picinput{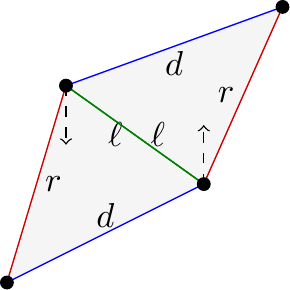}} \hspace{1cm}
  \subfigure[non admissible gluing\label{subfig:noncoherent_triangle}]{\picinput{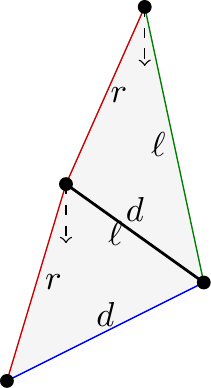}}
\end{center}
\caption{labels on triangles and admissibility of configurations}
\label{fig:triangles}
\end{figure}


\smallskip

If $Q$ is a labelled quadrangulation, also the triangles of the induced triangulation $T$ inherit labels $1\leq i \leq k$.
More precisely, each quadrilateral is cut in two triangles by its backward diagonal.
Bottom triangles and top triangles are exchanged by the hyperelliptic involution.
Let us consider a triangle $t$ on the sphere. Its preimage in $Q$ is a union of a top and a bottom triangle.
The bottom one belongs to some $q_i$ and we set $i$ as the label for $t$.

 We can then equivalently describe the tree of relations with three involutions $\sigma_\ell$, $\sigma_r$ and $\sigma_d$ of $\{1,\ldots,k\}$. Define $\sigma_\ell$  so that if the triangles $i$ and $j$ share an edge labelled $\ell$ then  $\sigma_\ell(i) = j$ and $\sigma_\ell(j) = i$. We define similarly $\sigma_r$ and $\sigma_d$. We use the notation $\vsigma$ for the triple $(\sigma_\ell,\sigma_r,\sigma_d)$ and call it the \emph{combinatorial datum} of the triangulation $T$.  
The following Lemma relates the combinatorial datum $\vpi$ of a quadrangulation $Q$ to the combinatorial datum $\vsigma$ of the quotient triangulation $T$. Equivalently, it links the tree of relations $G_T$ and the graph $G_Q$.
\begin{lemma}\label{lem:perm_triangulation}
If $T$ is a labelled triangulation with combinatorial datum $\vsigma = (\sigma_\ell,\sigma_r,\sigma_d)$ induced by a labelled quadrangulation $Q$ of a surface in $\CCC^{hyp}(k)$ with combinatorial datum $\vpi = (\pi_\ell,\pi_r)$ then $\sigma_d = \iota$ is the action of the hyperelliptic involution on the quadrilaterals of $Q$ and
\[
\pi_\ell = \sigma_r \circ \sigma_d \qquad \text{and} \qquad \pi_r = \sigma_\ell \circ \sigma_d.
\]
In particular $\pi_\ell^{-1} = \iota \pi_\ell \iota$ and $\pi_r^{-1} = \iota \pi_r \iota$.
\end{lemma}

\begin{proof}
By construction, the labels on the sphere are built in such way that $\sigma_d$ corresponds to the action of the hyperelliptic involution. 
Recall that quadrilaterals in $Q$ are cut in triangles by the backward diagonals.
Now $\pi_\ell$ can be seen on the bottom triangles as first crossing the diagonal (hence applying $\sigma_d$) and then crossing the top left side which is right slanted (hence applying $\sigma_r$). So $\pi_\ell = \sigma_r \sigma_d$. Reasoning in the same way for $\pi_r$ we get the other formula.
\end{proof}

Let us remark that one can show that $\sigma_\ell \sigma_r \sigma_d$ is a $k$-cycle, since it corresponds geometrically to turning around the singularity of angle $\pi k$ on the sphere. In~\cite{CassaigneFerencziZamboni}, it is shown that this $k$-cycle is a complete invariant that classifies pair of permutations in the same graph $\GGG$. Their main result can be rephrased as follows:

\begin{theorem}[\cite{CassaigneFerencziZamboni}] \label{thm:graph_of_graphs}
Let $Q$ be a quadrangulation of a surface in $\CCC^{hyp}(k)$ and $T_Q$ be the quotient triangulation.
Let $\vpi=(\pi_\ell,\pi_r)$ and $\vsigma=(\sigma_\ell,\sigma_r,\sigma_d)$ be respectively the combinatorial datum of $Q$ and $T_Q$.
Then, the permutation $\sigma_\ell \sigma_r \sigma_d = \pi_r \sigma_d \pi_\ell$ is a $k$-cycle which is invariant under the operation of staircase moves $\vpi \mapsto c \cdot \vpi$. Moreover, two combinatorial data $\vpi$ and $\vpi'$ that correspond to quadrangulations of surfaces in $\CCC^{hyp}(k)$ can be joined by a sequence of staircase moves and hence belong to the same graph $\GGG = \GGG(\vpi)$ if and only if $\pi_r \pi_\ell \sigma_d = \pi'_r \pi'_\ell \sigma_d'$.
\end{theorem}


The following corollary of this result is used to show that the rotation operator $R: \QQQ_k \to \QQQ_k$ defined in \S~\ref{sec:move_parameters} is well defined. 
\begin{corollary}\label{cor:same_graph}
Let $\vpi= (\pi_\ell, \pi_r)$ be a combinatorial datum of a labelled quadrangulation  $Q$ in $\QQQ_k$ and let $\vpi'= (\pi_\ell\, \pi_r\, \pi_\ell^{-1}, \pi_\ell^{-1})$. Then $\vpi'$ belongs to $\GGG (\pi)$. 
\end{corollary}
\begin{proof}
Consider the quadrangulation $Q'=R Q = (\vwedges', \vpi')$  (recall Definition~\ref{def:R}). It follows from the definition of $R$ that if $\iota$ denotes the action of the hyperelliptic involution on the labels of $Q$, the action $\iota'$ of   of the hyperelliptic involution on the labels of $Q'$ is given by $\iota' = \pi_\ell\, \iota\, \pi_\ell^{-1}$.
 Then, from the definition of $R$ and the equality  $\iota \pi_\ell^{-1} = \pi_\ell \iota$  which  follows from from Lemma~\ref{lem:perm_triangulation}, one has
\[
\pi'_\ell\, \pi'_r\, \iota' =
(\pi_\ell\, \pi_r\, \pi_\ell^{-1})\, \pi_\ell^{-1}\, (\pi_\ell \iota\, \pi_\ell^{-1}) =
\pi_\ell \pi_r \pi_\ell^{-1} ( \iota \pi_\ell^{-1}) = \pi_\ell \pi_r \pi_\ell^{-1} \pi_\ell \iota 
= \pi_\ell \pi_r \iota.
\]
This shows that $\vpi' \in \GGG (\vpi)$ by Theorem~\ref{thm:graph_of_graphs}, remarking that, with the notation in the Theorem, we have $\iota = \sigma_d$ and $\iota'= \sigma_d'$ by Lemma~\ref{lem:perm_triangulation}.  
\end{proof}

We remark finally that it is possible to define an operation on trees of relations (see~\cite{CassaigneFerencziZamboni} or~\cite{MarshSchroll}) that corresponds to a combinatorial staircase move, that is to the map which sends $\vpi \mapsto c \cdot \vpi$ (see the definitions in~\eqref{eq:def_c_pi1} and~\eqref{eq:def_c_pi2}). R.~Marsh and S.~Schroll in~\cite{MarshSchroll} generalize these operations to trees with $k$ labels on edges (here we have $k=3$ labels, namely $\ell$, $r$ and $d$) and show a link with cluster algebra combinatorics. They intepret moves on trees as changes of diagonals in $k$-angulations of polygons. For $k=3$, their triangulations are a combinatorial version of the metric triangulations of the sphere that we described above. 

\subsection{Existence of quadrangulations, proof of Theorem~\ref{thm:existence_of_quadrangulation}}
We now prove that for any surface $X \in \CCC^{hyp}(k)$  there exists quadrangulations   (Theorem~\ref{thm:existence_of_quadrangulation}). Before proceeding to the proof, we state and prove two lemmas that are valid for any translation surface, not necessarily in an hyperelliptic component. The first one is about existence of wedges and the second one about existence of admissible quadrilaterals.

\begin{lemma} \label{lem:existence_of_BA}
Let $X$ be a translation surface  which has no horizontal and no vertical saddle connections.
Then in any bundle of $X$ there are infinitely many left and right best approximations.

Moreover, for any bundle $\Gamma_i$ of $X$ we have
\[
\min\ \{\Im(\sc);\ \text{$\sc \in \Gamma_i$ is a best approximation and $|\Re(\sc)| < r$}\} < \frac{ \area(X)}{r}.
\]
\end{lemma}
Given a best approximation $\sc$, the  quantity $|\Re(\sc)| \Im(\sc)$, also called \emph{area of the best approximation} $\sc$, corresponds to the area of the immersed rectangle $R(\sc)$ given by Lemma \ref{lem:equivalentBA}. The above statement shows that this quantity is uniformely bounded from above. The optimal constant on a given surface is related to the Minkowski constant in the context of Cheung's Z-convergents, see~\cite{HubertSchmidt}. The lower bound of areas of best approximations is related to the Lagrange spectrum, see~\cite{HubertMarcheseUlcigrai} and~\S\ref{subsubsec:intro_Teich}. Finally, let us mention that there is a better bound for the systole (the length of the shortest saddle connection) following from J.~Smillie and B. Weiss' argument in~\cite{SmillieWeiss} (see the Appendix A in ~\cite{HubertMarcheseUlcigrai}), namely
\[
\sys(X) \leq 2\ \sqrt{\frac{\area(X)}{\pi (2g-2+n)}},
\]
where $g$ is the genus and $n$ the number of singularities of $X$.
We remark though that the proof of the above bound cannot be adapted to get bounds on the length of shortest saddle connection in a  given bundle.

The first part of the proof of Lemma~\ref{lem:existence_of_BA} is very similar to arguments used to prove minimality of the vertical flow 
under Keane's condition. 
The second statement in the Lemma is an adaptation  of the proof of an upper bound on the systole by Vorobets~\cite{Vorobets} to each bundle.
\begin{proof}[Proof of Lemma~\ref{lem:existence_of_BA}]
Let $X$ be a translation surface with no horizontal and no vertical saddle connections. 
Let $I$ be an horizontal segment in $X$ and assume that one of its endpoints, say $p$, is a singularity of $X$ and that $I$ does not contain any other singularity in its interior. 
We claim that there exists $t_1 > 0$ and $-\area(X) / |I| \leq t_2 < \area(X) / |I|$ such that for $t=t_1$ and $t=t_2$, $\phi_t(I)$ contains a point of $\Sigma$ in its interior (that is there exists $x$ in the interior of $I$ such that $\phi_t(x) \in \Sigma$).

Since the area of $X$ is finite, the set $\cup_{t\geq 0}\varphi_t(I)$ has to self-intersect. Let $s$ be the minimum first return time, that is the minimum $t>0$ such that there exists $x \in I$ for which $\phi_t(x) \in I$. Clearly $s \leq \area(X) / |I|$.  If there exist a singularity inside $\cup_{0< t < s} \varphi_t(I)$, that is there exists $0<t_0<s$ and $x_0$ such that $\varphi_{t_0}(x_0)\in \Sing$, we are done as we can take $t_1=t_2=t_0$. If there is none, it follows that $\phi_s$ is continuous on $I$. If $p \in \phi_s(I)$, we are done. We cannot have $\phi_s(I)= I$, otherwise there would be a vertical saddle connection. Thus, we can assume that the other endpoint of $I$, that we will denote by $y$, belongs to the interior of $\phi_s(I)$.
In this case, there is a point $z \in I$ such that $\phi_{-s}(z) = p$ and we can take $t_2 = -s$.
Let $x \in I$ be such that $\phi_s(x) =y$ (note that the distance between $p$ and $z$ is the same as the distance between $x$ and $y$).
Consider now the interval $I'\subset I$ which has $x$ and $y$ as endpoints. Reasoning as before, $\bigcup_{t > 0}\phi_t(I')$ has to self intersect. Let $s'>0$ be the minimum first return time of $I'$ in $I$.  If there exist a singularity inside $\cup_{0< t < s'} \varphi_t(I')$, then we are done. Otherwise, $\phi_{s'}(I')$ is an interval that intersects $I$ and which is disjoint from $\phi_s(I) \cap I$ by definition of first return time. Hence it has to contain $p$ in its interior and we can set $t_1 = s'$.

We now apply the claim to bundles of saddle connections.
Let us fix a positive real number $r > 0$ and  
let $\psi_t$ be the horizontal flow in $X$.  
For any given bundle $\Quad_i$ starting at a singularity $p \in \Sigma$, pick the vertical segment $I_r$ issued from $p$ that belongs to the bundle and whose length is ${\area(X)}/ r$.
Applying the claim to  $\psi_t$ from $I_r$ (remark the the property of having no horizontal and no vertical saddle connections is preserved by rotation of $\pi/2$),  we get the existence of a minimum $t_1 > 0$ such that $\psi_{t_1}(I_r)$ contains a singularity.
By construction, this gives a right geometric best approximation.
Similarly we obtain the existence of a minimum $t'_1 < 0$ such that $\psi_{t'_1}(I_r)$ contains a singularity. This gives us a left geometric best approximation.
We know from the claim that either $\min(|t_1|,|t'_1|) < r$.
We hence obtain a left and a right geometric best approximation whose imaginary part is less than $\area(X)/r$ and for one of them, the real part is less than $r$.
This proves the quantitative estimate of the statement.
By considering decreasing values of $r$, this construction provides saddle connections whose real part tends to $0$ (and imaginary part tends to $\infty$).
\end{proof}
We remark that the conclusion of  Lemma~\ref{lem:existence_of_BA} can still be proved under a weaker assumption, that is that the surface $X$ has no horizontal \emph{or} no vertical saddle connections. More precisely, if in a bundle $\Gamma_i$ there is a vertical saddle connection $\scw$ but no horizontal saddle connection then there is no best approximation $\scw'$ with $\Im(\scw') > \Im(\scw)$ but there are still infinitely many with arbitrarily small imaginary part.

\begin{lemma}[diagonal determine quadrilateral]\label{lem:d_to_q} \label{lem:existence_of_quadrilateral}
Let $X$ be a translation surface without vertical saddle connections and let $\sc$ be a saddle connection which is a geometric best approximation. Then there exists a unique admissible (in particular embedded) quadrilateral $q$ whose sides are all geometric best approximations and that has $\sc$ as foward diagonal.  

If moreover $\sc$ is left slanted (respectively right slanted), then there exists a unique right (resp.~left) slanted  admissible quadrilateral in $X$ whose sides are best approximations and so that $\sc$ is its bottom left side (resp.~right side).
\end{lemma}


We remark also that the second part of the Lemma does not give any information  on left slanted (resp. right slanted) admissible quadrilaterals in $X$ whose sides are geometric best approximations and have $\sc$ as bottom left (resp.~right) side. There might indeed be either none or several  such quadrilaterals, as it is clear from the last part of proof below.

In the proof of Lemma~\ref{lem:d_to_q}, we will use the following Lemma.
\begin{lemma}\label{lem:quadrilateral_embedded}
Let $X$ be a translation surface and let $P \subset X$ be an isometrically immersed convex polygon that contains no singularities in its interior or in the interior of its sides and whose vertices belong to $\Sigma$. Then the interior of $P$ is embedded in $X$.
\end{lemma}
\begin{proof}
Let $P_0 \subset \CC$ be convex polygon and let $f: P_0 \to X$ be an isometric immersion so that the image $P = f(P_0)$ is the given immersed polygon.
We assume that $P$ contains no singularities in its interior or in the interior of its sides and that its vertices belong to $\Sigma$.
We need to prove that $f$ is globally injective. 
Assume by contradiction that there exists two distinct points $p_1,p_2$ in the interior of $P_0$ such that $f(p_1)=f(p_2)$ and consider the segment $\gamma$ connecting $p_1$ to $p_2$. 
Then $f(\gamma)$ is an isometrically immersed closed curve on $X$ and hence a closed geodesic with respect to the flat metric. Thus, there exists a cylinder $C$ foliated by closed flat geodesics which contain $f(\gamma)$. Since $P = f(P_0)$ does not contain singularities, if $\gamma'$ is another segment inside $P$ which is obtained from $\gamma$ by parallel transport (that is $\gamma'=\gamma+c$ for some $c \in \CC$), $f(\gamma')$ is also obtained by parallel transport of $f(\gamma)$ inside $X$ and hence is still a closed flat geodesic.
Now consider the longest segments inside $P_0$ which are parallel to $\gamma$. Because of convexity, one of them necessarily starts at at a vertex of $P_0$. Now, we can find $c \in \CC$ such that $\gamma'=\gamma+c$ is contained in $P_0$ and starts from that singularity. By construction, the other endpoint of $\gamma'$ is either inside $P_0$ or in the interior of its sides. Because the two endpoints of $\gamma'$ are identified by $f$ this contradicts the fact that the interior of $P$ and the interior of its sides are free of singularities.
\end{proof}

\begin{proof}[Proof of Lemma~\ref{lem:d_to_q}]
By definition of best approximation, $\sc$ is the diagonal of an immersed rectangle $R(\sc)\subset X$. Let $v_\ell$ and $v_r$ respectively be the left and right vertical sides of $R$, see Figure~\ref{subfig:hit_singularities}.
Flow $v_\ell$ (respectively $v_r$) horizontally to the right (respectively to the left) until the first time it hits a singularity, that we call $p_\ell$ (respectively $p_r$) as shown in Figure~\ref{subfig:hit_singularities}. Both singularities hit are unique since otherwise $X$ would have a vertical saddle connection.  Consider the immersed convex quadrilateral  which has as vertices $v_r, v_\ell$ and the endpoints of $\sc$ (see  Figure~\ref{subfig:immersed_rectangle_around_quad}).  Since by construction  it does not contain conical singularities in its interior,   by Lemma~\ref{lem:quadrilateral_embedded} it is  embedded. Thus, we constructed  an admissible quadrilateral which has $\sc$ as forward diagonal. Furthermore, each of the sides of $q$ is a geometric best approximation since by construction each is the diagonal of an immersed rectangle without singularities in its interior (see Figure~\ref{subfig:immersed_rectangle_around_quad}).

\begin{figure}[!ht]
\begin{center}
  \subfigure[\label{subfig:hit_singularities}]{\picinput{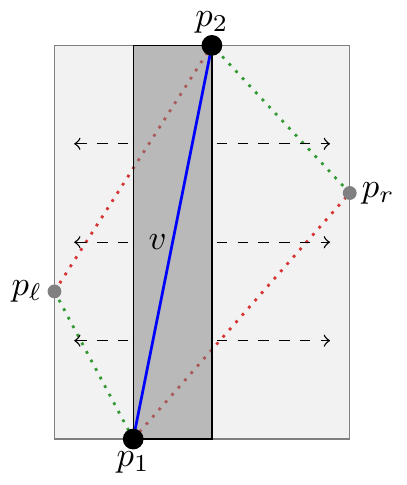}} \hspace{.6cm} 
  \subfigure[\label{subfig:immersed_rectangle_around_quad}]{\picinput{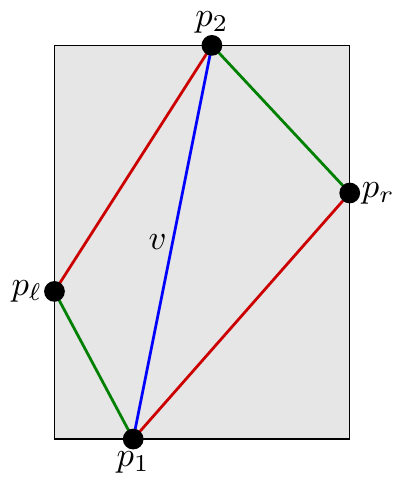}} \hspace{.6cm} 
  \subfigure[\label{subfig:q_from_bundle}]{\picinput{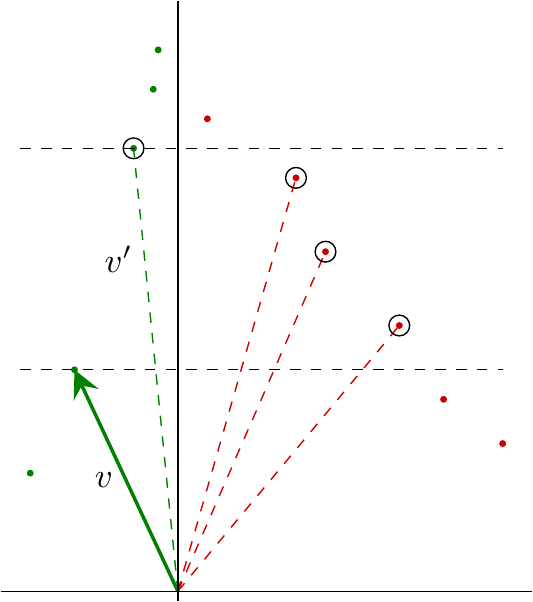}}
\end{center}
\caption{building a quadrilateral from a diagonal or a side (proof of Lemma~\ref{lem:d_to_q})}
\label{fig:lemma_d_to_q}
\end{figure}

The uniqueness comes from the construction: given an admissible quadrilateral $q$ whose sides are best approximations, its forward diagonal $\sc$ is a best approximation and we can build $q$ by flowing horizontally as above the vertical sides of  the immersed rectangle $R(\sc)$ associated to $\sc$. 


For the second part of the statement, we consider the same construction. Consider a fixed left slanted best approximation $\sc$ in  some bundle $\Gamma_i^\ell$. We want to determine which diagonals $\ssc$ may have produced $\sc$ by horizontal flowing the vertical left side of the associated rectangle $R(\ssc)$. Let $\sc'$ be the slanted saddle connection in $\Gamma_i^\ell$  which in next to $\sc$ in the natural order given by increasing imaginary part. One can see, looking at Figure~\ref{subfig:q_from_bundle}, that all the possible such diagonals $\ssc$ are exactly  the left slanted  saddle connection $\sc'$ and all the  right slanted saddle connections $\sc_r$ which satisfy  $\Im(\sc) \leq \Im(\sc_r) \leq \Im(\sc')$ (possibly none). In particular, only $\sc'$ is the diagonal of a right slanted quadrilateral as in the second part of the lemma. The right slanted saddle connections are all possible diagonals of the set (possibly empty) of left slanted admissible quadrilaterals with sides which are best approximations and $\sc$ as bottom left side.  
\end{proof}

We are now ready to prove Theorem~\ref{thm:existence_of_quadrangulation}. Let us first remark that the statement is trivial for the torus case, since for any given lattice with neither horizontal nor vertical vector there always exists a basis which form the wedge of an admissible quadrilateral. 
 
\begin{proof}[Proof of Theorem~\ref{thm:existence_of_quadrangulation}]
Let $q$ be  an admissible quadrilateral whose sides are all best approximations, whose existence is guaranteed by  Lemma~\ref{lem:existence_of_BA} and Lemma~\ref{lem:existence_of_quadrilateral}. We denote its bottom sides by $\sc_\ell$, $\sc_r$ and its top sides by $\sc'_r$ and $\sc'_\ell$.
  Now, consider its image $s(q)$ under the hyperelliptic involution $s$. 
  It is easy to see that, since all sides of $q$ are best approximations,  either $q = s(q)$ or $q$ and $s(q)$ have disjoint interiors. 
  In both cases, for each side $\sc$ of $q$ if $\sc = s(\sc)$ we do nothing, while if $\sc \not= s(\sc)$ we cut and paste as in Lemma~\ref{lem:cut_hyp}. After this operation,  we a obtain a  surface made by one or two quadrilateral (if respectively  or $q \not= s(q)$) and at most four surfaces $X_\ell$, $X_r$, $X'_r$ and $X'_\ell$  that contain respectively $\sc_\ell$, $\sc_r$, $\sc'_r$ and $\sc'_\ell$ (with the convention that we assume that $X_z$ is empty if $\sc_z = s(\sc_z)$). 
 Moreover, each of these surfaces belongs to a hyperelliptic component with strictly smaller total angle  by Lemma~\ref{lem:cut_hyp}. On  each non empty surface among $X_\ell$, $X_r$, $X'_r$ and $X'_\ell$ let us consider  a saddle connection given by Lemma~\ref{lem:existence_of_BA}, let us complete   it to an admissible quadrilateral by  Lemma~\ref{lem:existence_of_quadrilateral} and then iterate the above construction. In finitely many steps, the construction thus produces $k$ admissible quadrilateral which provide a quadrangulation of the original surface $X$.  
\end{proof}
Let us remark that  the proof does not extend to other components of strata.
We can still use a cut and paste construction but the resulting surfaces $X_\ell$, $X_r$, $X'_\ell$ and $X'_r$ might be connected to each other.
In particular, if two of them are connected we obtain a surface in which we want to complete a set of two saddle connections into a quadrangulation.

\subsection{Existence of staircase move, proof of Theorem~\ref{thm:existence_of_staircase_move}}
In this section, we give the proof of Theorem~\ref{thm:existence_of_staircase_move}.   

\begin{proof}[Proof of Theorem~\ref{thm:existence_of_staircase_move}]
The proof proceeds by induction on the number of quadrilaterals, or in an equivalent way on the  integer $k$ such that the surface belongs to $\CCC^{hyp}(k)$. The case of the torus  ($k=1$) is trivial, since a  staircase made of one quadrilateral is always well slanted.

Let $Q = (\vpi,\vwedges)$ be an admissible quadrangulation of a surface in $\CCC^{hyp}(k)$ and denote by $\iota$ the action of the hyperelliptic involution $s$  on the quadrilaterals (i.e. $\iota(i) = j$ if and only if $q_j = s(q_i)$). Let us prove by contradiction  that there exists at least one well slanted staircase in which it is possible to make a diagonal change. If  no staircase move for $Q$ is possible,  we claim that there exists a right  staircase $S$ which contains both left and right slanted quadrilaterals. 
Indeed,  no right staircase consists of only right slanted quadrilaterals, otherwise it would be well slanted and a right move would be possible. If all right staircases consist of only left slanted quadrilaterals, all left staircase moves are possible. Thus, there exists $S$ with both left and right slanted quadrilaterals. In particular, in $S$ there exist two consecutive quadrilaterals $q_i$ and $q_{\rperm(i)}$ which are respectively left slanted and right slanted. We remark that it follows that $\iota(i) \not= \rperm(i)$, since otherwise the diagonals of $q_i$ and $q_{\rperm(i)}$ would be parallel and hence $q_i$ and $q_{\rperm(i)}$ would have the same slantedness. In particular, the common edge $\scw_{\rperm(i),\ell}$ of $q_i$ and $q_{\rperm(i)}$ does not contain a Weierstrass point and hence $\scw_{\rperm(i),\ell} \neq s(\scw_{\rperm(i),\ell})$. 

Let us cut the quadrangulation $Q$ along the edge $\scw_{\rperm(i),\ell}$ (between $q_i$ and $q_{\rperm(i)}$) and along its image under hyperelliptic involution, which is the edge $\scw_{\iota(i),\ell}$ (between $s(q_{\rperm(i)}) = q_{\iota(\pi_r(i))}$ and $s(q_{i}) = q_{\iota(i)}$). From Lemma~\ref{lem:cut_hyp}, we know that after cutting along these edges we obtain two connected components and that, after identifying on each of them the corresponding copies of $\scw_{\rperm(i),\ell}$ and $\scw_{\iota(i),\ell}$ by parallel translations,  we obtain two quadrangulations of surfaces with strictly less quadrilaterals. We denote by $X'$ the surface containing $q_i$. By inductive assumption, there exists a staircase move in $X'$. Since $q_i$ is left-slanted, the saddle connection $\scw_{i,\ell}$ does not change during the move. Hence, the move lifts to $X$ and by glueing  back the two components we can globally define a staircase move on $X$.  
\end{proof}

From Lemma~\ref{thm:existence_of_staircase_move}, it is easy to see that the Keane's condition (no vertical saddle connections) is exactly the condition needed for any diagonal changes algorithm not to stop (for the analogous of this Lemma in the case of Rauzy-Veech induction see \cite{Yoccoz}).  
\begin{lemma} \label{lem:Keane}
Let $Q$ be a quadrangulation of a surface $X$ in $\CCC^{hyp}(k)$. There exists an infinite sequence of staircase moves starting from $Q$ such that the real part of each saddle connection in the wedges of $Q$ tends to zero if and only if $X$ has no vertical saddle connection.

Moreover, if $X$ has no vertical saddle connection  then for any infinite sequence of staircase moves starting from $Q$ there are infinitely many left and right diagonal changes and the width of each wedge goes to zero.
\end{lemma}
\begin{proof}
Let us first prove the second part of the Lemma. Assume that $X$ has no vertical saddle connection and let $Q^{(n)}$ be a sequence of quadrangulations obtained by staircase moves starting from $Q=Q^{(0)}$. 
Assume by contradiction that for some $1\leq i \leq k$ the quadrilateral $q_i^{(n)}$ undergos  only finitely many left changes, i.e. there exists $n_1$ so that for $n \geq n_1$ we have $w^{(n)}_{i,r} = w^{(n_1)}_{i,r}$. Then, because $\Re(w^{(n)}_{i,r}) \neq 0$ and the area of the surface is finite, the sequence $(\Im(w^{(n)}_{i,\ell}))_{n \in \NN}$ has to be bounded. Because of the discreteness of the set of saddle connections, this implies that there exists $n_2 \geq n_1$ such that for $n \geq n_2$, also $w^{(n)}_{i,\ell} = w^{(n_1)}_{i,\ell}$ and hence $q^{(n)}_i = q^{(n+1)}_i$ for any $n \geq n_2$. Since the top sides of $q^{(n)}_i$ are bottom sides for $q^{(n)}_{\pi_\ell(i)}$ and $q^{(n)}_{\pi_r(i)}$ respectively, 
 this implies also that for $n \geq n_2$ the quadrilateral $q_{\pi_\ell(i)}$ undergos only right diagonal changes and the quadrilateral $q_{\pi_\ell(i)}$ undergos only left diagonal changes. In particular, repeating the same argument $\scw^{(n)}_{\pi_\ell(i),l}$ and $\scw^{(n)}_{\pi_r(i),r}$ are ultimately constant. Because of the connectedness of the surface, or equivalently because the group generated by $\pi_\ell$ and $\pi_r$ acts transitively on $\{1,\ldots,k\}$ we can repeat the argument and show that the quadrangulations $Q^{(n)}$ are ultimately constant, contradicting the assumption that the sequence is obtained by staircase moves (which are by definition not identity).   
 
 \smallskip
 Let us now prove the first part. Let $X \in \CCC^{hyp}(k)$ and let  us first assume that there is an infinite sequence of staircase moves from the quadrangulation $Q=Q^{(0)}$ of $X$ such that the associated sequence of quadrangulations $Q^{(n)} = (\vpi^{(n)},\vwedges^{(n)})$ is such that both  $\Re(w^{(n)}_{i,\ell})$ and $\Re(w^{(n)}_{i,r})$ tend to zero for any $1\leq i \leq k$. 
  Then, necessarily, since the set of saddle connections is discrete,  $\Im(w^{(n)}_{i,\ell})$ and $\Im(w^{(n)}_{i,r})$ tend to infinity. Assume by contradiction that there is a vertical saddle connection $\sc$ on $X$ and let $\Quad_j$, $1\leq j \leq k$, be the bundle which contains it. Since by definition of wedges the sides of a wedge form a triangle embedded in the surface and the imaginary parts of the  wedge $w^{(n)}_{j}$ both go to infinity, for $n$ sufficiently large $\sc$ is contained in the triangle with sides $w_{i,\ell}$ and $w_{i,r}$. This contradicts the fact that the interior of the triangle is free of singularities. Thus $X$ has no vertical saddle connections.

  Conversely, let $X \in \CCC^{hyp}(k)$ be without vertical saddle connections and $Q^{(0)}$ be a quadrangulation of $X$. Because $X$ has no vertical saddle connection then no quadrangulation on $X$ is vertical. Applying inductively Theorem~\ref{thm:existence_of_staircase_move} from $Q^{(0)}$ we obtain an infinite sequence of quadrangulations. Using the first part of the proof, the width of each wedge necessarily tends to 0.
\end{proof}

\subsection{Non hyperelliptic components} \label{subsec:nonhyp_quadrangulation}
In this section we provide examples of translation surfaces which do not belong to a hyperelliptic component of a stratum and admit quadrangulations for which there are no possible staircase moves.
Our strategy consists in finding quadrangulation with $k$ quadrilaterals for which both $\pi_\ell$ and $\pi_r$ are $k$-cycles.
This construction is possible in many component of stratum but not in $\CCC^{hyp}(k)$ if $k \geq 3$. Then, once we found this combinatorial datum we find a length datum in order that there is at least one left-slanted and one right-slanted quadrilaterals.

We first consider the stratum $\HHH(0,0,0)$, which is the smallest stratum which does not contain a hyperelliptic component.
Let $\pi_\ell = (1,2,3) = \pi_r = (1,2,3)$ and consider the wedges
\[
w_{1,\ell} = (-1.3,2), \quad
w_{1,r} = (1,1), \quad
w_{2,\ell} = w_{3,\ell} = (-1.3,2) \quad \text{and} \quad
w_{2,r} = w_{3,r} = (1.7,1)
\]
One can check that these length data satisfy the train-track relations for $\vpi= (\pi_\ell,\pi_r)$ and hence correspond to a quadrangulation $Q$ (see Figure~\ref{fig:no_staircase_move_H000}). 
Moreover we have
\[
w_{1,d} = (-0.3,3), \quad w_{2,d} = (0.4,3) \quad \text{and} \quad w_{3,d} = (-0.3,3).
\]
Hence, there is no well slanted staircase in $Q$. 
\begin{figure}[!ht]
\begin{center} \picinput{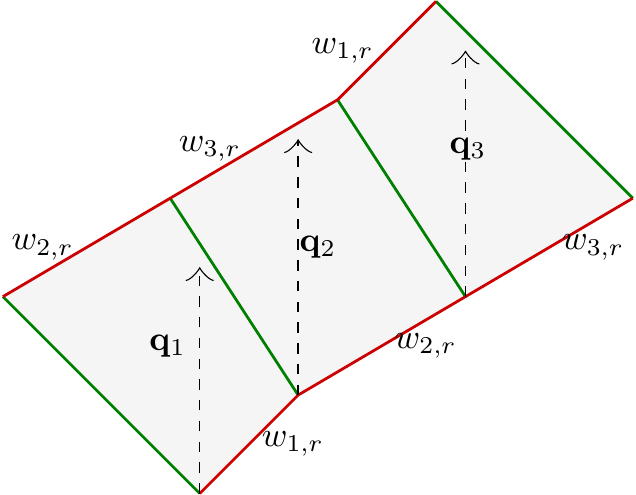} \end{center}
\caption{a quadrangulation of a surface in $\HHH(0,0,0)$ with no well slanted staircase}
\label{fig:no_staircase_move_H000}
\end{figure}
One can notice that the surface $X$ associated to $Q$ admits a hyperelliptic symmetry that exchanges $q_1$ and $q_3$ while fixes $q_2$. In other words, the quadrangulation is fixed by the hyperelliptic involution of $X$. The quotient of $X$ by the hyperelliptic involution belongs to $\QQQ(0,0,-1^4)$. One can also check that in that case, the graph associated to the triangulation described in~\S\ref{subsubsec:triangulations} is no more a tree.

\smallskip
Now we construct another example which belongs to $\HHH(4)$. This stratum is the smallest one which contains more than one component one of which is hyperelliptic. Let $\pi_r = (1,2,3,4,5)$ and $\pi_\ell = (2,1,3,5,4)$ and consider the wedges
\[
\begin{array}{lllll}
w_{1,r} = (2,1) & w_{2,r} = (1.5,1) & w_{3,r} = (2.5,1) & w_{4,r} = (3.5,1) & w_{5,r}=(1,1) \\
w_{1,\ell} = (-1.5,2) & w_{2,\ell} = (-2.5,2) & w_{3,\ell} = (-0.5,2) & w_{4,\ell} = (-1.5,2) & w_{5,\ell}=(-3,2).
\end{array}
\]
Then one can check that the train track relations are satisfied and hence $Q = (\vpi,\vwedges)$ is a quadrangulation (see Figure~\ref{fig:no_staircase_move_H4}). Moreover we have
\[
w_{1,d} = (-0.5,3),\quad
w_{2,d} = (1,3),\quad
w_{3,d} = (1,3),\quad
w_{4,d} = (0.5,3),\quad
w_{5,d} = (-0.5,3).
\]
This shows that there is no well slanted staircase in $Q$.
\begin{figure}[!ht]
\begin{center}\picinput{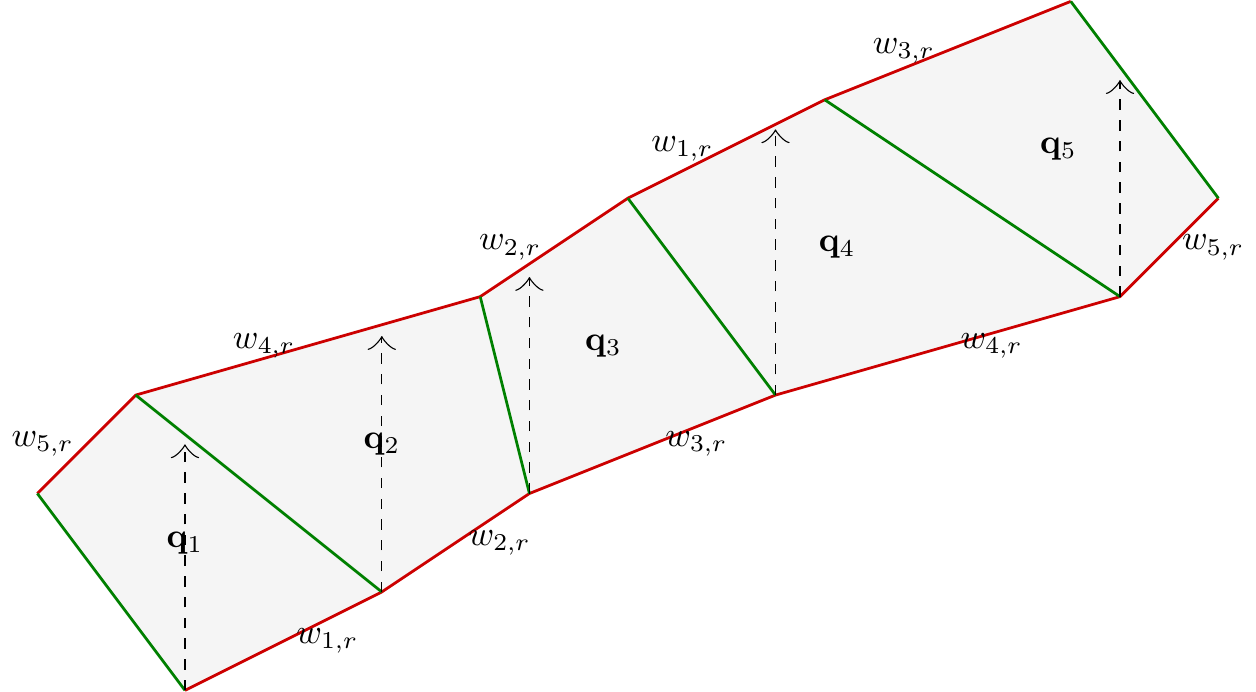}\end{center}
\caption{a quadrangulation of a surface in a non hyperelliptic component of $\HHH(4)$ with no well slanted staircase}
\label{fig:no_staircase_move_H4}
\end{figure}

\section{Best-approximations and bispecial words via staircase moves} \label{sec:BA_and_bisp}
In this section we prove Theorem~\ref{thm:wedges_are_best_approx} and show more generally that all best approximations in each bundle are produced by any slow diagonal changes algorithm (see Theorem~\ref{thm:wedges_are_best_approx2}). We then deduce several results. We first show,  by proving Theorem~\ref{thm:same_objects},  that the geometric objects, namely wedges and well slanted staircases, produced by any  sequence of staircase moves are the same. We then prove that the saddle connections which realize the  systoles along a Teichmueller geodesics are contained in the set of best approximations (Theorem~\ref{thm:flat_systoles}). Finally we prove that cutting sequences of bispecial words coincide with best approximations (Theorem~\ref{thm:best_approx_and_bispecials}) and explain how they can be generated recursively using diagonal changes (see Theorems~\ref{thm:seqsubstitutions}).


\subsection{Best approximations via staircase moves and applications}\label{sec:BAsystoles}
In this section we prove Theorem \ref{thm:wedges_are_best_approx}. Let us first prove the equivalent geometric characterization of best approximations as diagonals of immersed rectangles (Lemma~\ref{lem:equivalentBA}).


\begin{proof}[Proof of Lemma~\ref{lem:equivalentBA}]
In this proof, we will explicitly avoid the identification of saddle connections in a bundle $\Quad_i$ with their displacement vectors in $\CC$ and we will denote by $\hol(\gamma) \in \CC$ the displacement vector of a saddle connection $\gamma$ on $X$ and by $\hol(\Quad_i)$ the set of displacement vectors of saddle connections in $\Quad_i$. 
For each saddle connection $\gamma$ in $\rQuad[i]$ (respectively in $\lQuad[i]$)
let $\widetilde{R}(\gamma)$ be the rectangle given
by \begin{equation}\label{rectangle_def}
\widetilde{R}(\gamma) = [0, \Re \left(\hol(\gamma)\right) ] \times [0, \Im \left(\hol(\gamma)\right) ] \quad \mathrm{( resp.} \  \widetilde{R}(\gamma) = [ \Re \left(\hol(\gamma)\right), 0 ] \times [0, \Im \left(\hol(\gamma)\right) ] \mathrm{\, )}.
\end{equation} 
Using this notation, we first remark that Definition \ref{def:BA} can be rephrased as follows: 
a saddle connection $\gamma \in \rQuad[i]$ (respectiveley $\gamma \in \lQuad[i]$)
 is a \emph{(geometric) best approximation}  if and only if the rectangle $\widetilde{R}(\sc)$ does not contain any element of $\hol(\rQuad[i])$ (resp. $\hol(\lQuad[i])$) in its interior. 

Let $\sc$ be a saddle connection starting at a point $p_0 \in \Sigma$ and, assuming that $\sc$ is right slanted,  let $\rQuad[i]$ be the bundle to which $\sc$ belongs (the case of $\sc \in \lQuad[i]$ is analogous).
Suppose first that there exists an immersed rectangle $R(\sc) \subset X$  which has $\sc$ as a diagonal and does not contain singularities in its interior.  The image of  $R(\sc)$ by the developing map $\dev_{p_0}: R(\sc) \rightarrow \CC$ given by $p \mapsto \int_{p_0}^p \omega$ is exactly the rectangle $\widetilde{R}(\sc)$ in \eqref{rectangle_def}. 
If by contradiction $\sc$ is not a best approximation, by the remark at the beginning of the proof $\hol(\rQuad[i] )$ intersects the interior of $ \widetilde{R}(\sc)$. Thus  there is a saddle connection $\gamma \in \rQuad[i]$ whose holonomy $\hol(\gamma)$ belongs to the interior of $ \widetilde{R}(\sc)$. Since $\gamma$ belongs to the same bundle than $\sc$, this means that  $\gamma$ is contained in $R(\sc)$ and hence the endpoint of $\gamma$ is a singularity in the interior of $R(\sc)$, which contradicts   the initial assumption.  

Conversely,  assume that $\gamma \in\rQuad[i] $ is a best approximation (the proof for $\gamma \in  \lQuad[i]$ is analogous). Then we claim that we can immerse the rectangle $\widetilde{R}(\sc)$ given by (\ref{rectangle_def}) in $X$ so that its image $R(\sc)$ is an immersed rectangle which has $\gamma$ as diagonal and does not contain singularities in its interior. 
Define the immersion $\iota$ by sending $z = \rho e^{i\theta} \in \widetilde{R}(\sc)$ to the point $\iota(z) = \gamma_\rho^\theta(p_0)$ which has distance $\rho$ from $p_0$ and belongs to the unique linear  trajectory $(\gamma_t^\theta(p_0))_{t\geq 0}$ in direction $\theta$ which starts at $p_0$ and such that such that  $| \angle(\gamma_t^\theta(p_0), \gamma)|< \pi/2$. To see that $\iota$ is well defined, it is enough to check that these trajectories do not hit singularities.  This will show at the same time that the image $R(\gamma)$ of $\iota$ does not intersect $\Sigma$. If by contradiction  there is a singularity $p_1 \in \Sigma$ in the interior of $R(\gamma)$, since the saddle connection $\gamma'$ connecting $p_0$ to $p_1$  is inside $R(\gamma)$, it belongs to the same bundle than $\gamma$ and has holonomy in  $\widetilde{R}(\gamma)$, thus the interior  of $\widetilde{R}(\gamma)$ intersects $\hol(\rQuad[i])$, which contradicts the equivalent definition of best approximation given by the remark at the beginning of this proof. 
\end{proof}

\subsubsection{Staircase moves produce the same geometric objects}\label{sec:sameobjects}
Let us now prove the following theorem, which is a more precise formulation of Theorem~\ref{thm:wedges_are_best_approx} in the introduction.
\begin{theorem}\label{thm:wedges_are_best_approx2}
Let $X$ be a surface in $\CCC^{hyp}(k)$ with neither horizontal nor vertical saddle connections.
Let $(Q^{(n)})_{ n \in \mathbb{Z}}$ be any sequence of labeled quadrangulations $Q^{(n)}= (\vpi^{(n)}, \vwedges^{(n)})$ of $X$ where $Q^{(n+1)}$ is obtained from $Q^{(n)}$ by simultaneous staircase moves.
Then, for each $1\leq i \leq k$ the saddle connections in the sequence $(\scw_{i,\ell}^{(n)})_{ n \in \mathbb{Z}}$ (resp.  $( \scw_{i,r}^{(n)})_{ n \in \mathbb{Z}}$) are exactly all best approximations in $\lQuad[i]$ (resp. $\rQuad[i]$) ordered by increasing imaginary part. 
\end{theorem}

\begin{proof}
We first prove that any saddle connections belonging to the wedges of one of the quadrangulations in $(Q^{(n)})_{ n \in \mathbb{Z}}$  is a  best approximation. Let $Q =(\vpi, \vwedges) = Q^{(n)} $ be a quadrangulation in the sequence and let $w_\ell$ be a left slanted saddle connection belonging to the wedge of some $q \in Q$.  The proof for right slanted saddle connections in the wedges is analogous.
Let $S$ be the right staircase that contains $q$, so that $\scw_\ell$ belongs to the interior of the staircase $S$.  Let
  $R(\scw_\ell) \subset X$ be the image of the rectangle $ \widetilde{R}(\scw_\ell) $ in  $\CC$  which has $\scw_\ell$ as its diagonal,  shown in Figure~\ref{fig:immersed_rectangle_staircase}. Since each saddle connection belonging to the boundary of $S$ is left slanted, $\widetilde{R}(\scw_\ell)$ is contained in the universal cover $\widetilde{S}$ of $S$ shown in Figure \ref{fig:immersed_rectangle_univ_cover} and thus  $R(\scw_\ell)$ is contained inside $S$. Since the staircase $S$ does not contain any singularity in its interior, it follows that $R(\scw_\ell)$ is an immersed rectangle which contains no point of $\rQuad[i]$ in its interior. This shows that $\scw_\ell$ is a geometric best approximation by Lemma~\ref{lem:equivalentBA}.

\begin{figure}[!ht] \label{fig:immersed_rectangle}
\begin{center}
  \subfigure[in a staircase\label{fig:immersed_rectangle_staircase}]{\picinput{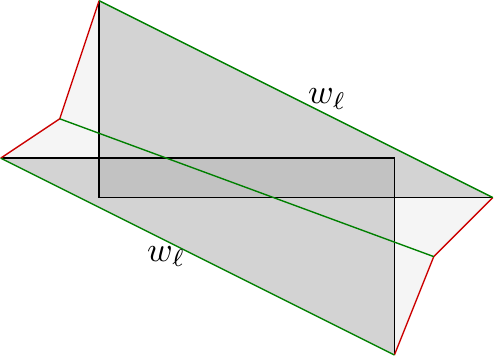}}  \hspace{1cm} 
  \subfigure[in the universal cover\label{fig:immersed_rectangle_univ_cover}]{\picinput{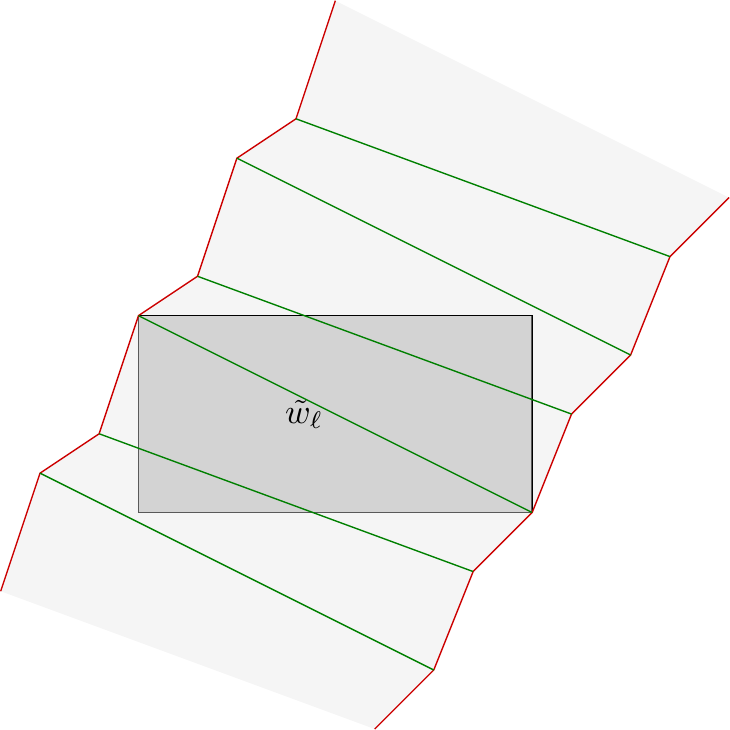}}
\end{center}
\caption{immersed rectangle around a side in a staircase which becomes embedded in its universal cover}
\end{figure}

\smallskip
Let us now prove that all geometric best approximations in $\rQuad[i]$ for any fixed $1\leq i \leq k$ appear in the sequence $(\scw_{i,r}^{(n)})_{n \in \ZZ}$ in their natural order. Since we just proved the saddle connections in $\rQuad[i]$  given by the sequence $(\scw_{i,r}^{(n)})_{n \in\mathbb{Z}}$ are geometric best approximations and by construction they are naturally ordered by increasing imaginary part, it is enough to show that if $\scw_{i,r}^{(n)}$ and $\scw_{i,r}^{(n+1)}$ are two successive saddle connections in $\rQuad[i]$ according to this natural order, there is no geometric best approximation with imaginary part strictly in between $\Im \scw_{i,r}^{(n)}$ and $\Im \scw_{i,r}^{(n+1)}$. For this, we claim that it is enough to show that if $R \subset \CC$ is the rectangle   $R=[0,\Re(\scw_{i,r}^{(n)})] \times [0,\Im(\scw_{i,r}^{(n+1)})] $ shown in  Figure~\ref{fig:BA_and_diagonal_changes}), then 
\[
\Quad_i^\ell \ \cap R = \left\{\scw_{i,r}^{(n)}, \scw_{i,r}^{(n+1)}\right\}.
\]
Indeed, this implies that there are no saddle connection $\sc \in\rQuad[i] $ with $\Im \scw_{i,r}^{(n)} < \Im \sc < \Im \scw_{i,r}^{(n+1)}$ and $0 < \Re \sc \leq \Re(\scw_{i,r}^{(n)})$. And if $\sc \in\rQuad[i]$ satisfies $\Im \scw_{i,r}^{(n)} < \Im \sc < \Im \scw_{i,r}^{(n+1)}$ and $\Re \sc > \Re(\scw_{i,r}^{(n)})$ then it is not a best approximation.

\begin{figure}[!ht] 
\begin{center}\picinput{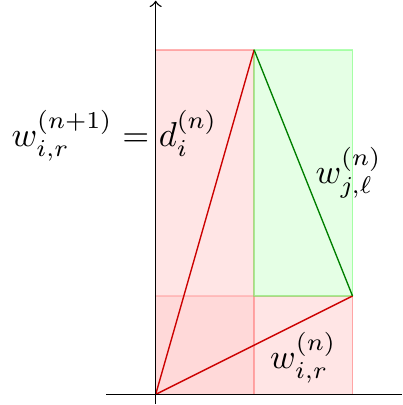} \end{center}
\caption{diagonal change seen on the displacement vectors} 
\label{fig:BA_and_diagonal_changes}
\end{figure}
By construction, since $\scw_{i,r}^{(n)}$ and $\scw_{i,r}^{(n+1)}$ are consecutive saddle connections, $\scw_{i,r}^{(n+1)}$ is the diagonal $d_i^{(n)}$ of the quadrilateral $q_i$ in $Q^{(n)}$. Thus the top right saddle connection of $q^{(n)}_i$, that we will denote by $\scw_{j,\ell}^{(n)}$, joins the endpoint of $\scw_{i,r}^{(n)}$ and $\scw_{i,d}^{(n)}$. 
Notice that $R$ is the union of the three smaller rectangles 
$R_1, R_2, R_3$ which have as diagonals respectively the  saddle connections $\scw_{i,r}^{(n)}$, $\scw_{i,d}^{(n)}$ and $\scw_{j,\ell}^{(n)}$ (see Figure~\ref{fig:BA_and_diagonal_changes}). By the previous part of the proof and by definition of geometric best approximation, each of the rectangles $R_1$ and $ R_2$ do not contain elements of $\rQuad[i]$ in their interior.
Thus, if by contradiction there exist an element $\sc \in \rQuad[i]$ in the interior of  $R$, there is an element of $\ssc \in \rQuad[i]$ inside $R_3$. 

Thus, there is also a saddle connection  $\ssc \in \lQuad[j]$ inside the image of $R_3$ on the surface  
contradicting that, by the first part of the theorem, also $\scw_{j,\ell}^{(n)}$ is a best approximation. This concludes the proof.
\end{proof}

 Theorem~\ref{thm:wedges_are_best_approx2} has the following corollary for sequences obtained by forward moves only:
\begin{corollary}\label{cor:wedges_are_best_approx2}
  Let $X$ be a surface in $\CCC^{hyp}(k)$ with no vertical saddle connections.
  Let $(Q^{(n)})_{ n \in \mathbb{N}}= \left((\vpi^{(n)}, \vwedges^{(n)})\right)_{ n \in \mathbb{N}}$ be any sequence of labeled quadrangulations of $X$ where $Q^{(n+1)}$ is obtained from $Q^{(n)}$ by simultaneous staircase moves. Then:
  
  \begin{itemize}
  \item[(i)] For each $1 \leq i \leq k$ the saddle connections in the sequence $( \scw_{i,r}^{(n)})_{ n \in \mathbb{N}}$ (resp.  $( \scw_{i,\ell}^{(n)})_{ n \in \mathbb{N}}$) are exactly all best approximations  $\sc$ in $\rQuad[i]$ (resp. in $\lQuad[i]$) which have $\Im \sc \geq \Im \scw_{i,r}$ (resp. $\Im \sc \geq \Im \scw_{i,\ell}$), or, equivalently, $|\Re \sc| \geq |\Re \scw_{i,r}|$ (resp. $|\Re \sc | \geq |\Re \scw_{i,\ell}|$).

  \item[(i)] For each $i$, $1 \leq i \leq k$, the set of diagonals $(\scw^{(n)}_{i,d})_n$ coincide with the set of best approximations $\sc$ in $\Gamma_i$ such that $\Im(\sc) > \max(\Im(w_{i,\ell}), \Im(w_{i,r}))$; or equivalently to the set of bottom sides of the quadrilaterals $(q_i^{(n)})_n$ different from the one of $q_i^{(0)}$.
  \end{itemize}
\end{corollary}

\begin{proof}
Part $(i)$ follows from Theorem~\ref{thm:wedges_are_best_approx2} since geometric best approximations are produced ordered by increasing imaginary part. Remark that if $\sc$ and $\ssc$ are left best approximations then $\Im \sc< \Im \ssc$ (resp.~$|\Re \sc|< |\Re \ssc|$ if and only if $|\Re \sc| < |\Re \ssc|$ (resp.  $|\Re \sc| < |\Re \ssc$).

Recall from Lemma~\ref{lem:Keane}, that if $X$ has no vertical saddle connection then each diagonal of $Q^{(n)}$ eventually becomes a side of a wedge. Hence Part \textit{(ii)} follows from Part \textit{(i)}.
\end{proof}
\smallskip
Combining Theorem~\ref{thm:wedges_are_best_approx}  with Lemma~\ref{lem:d_to_q} (diagonals uniquely determine their quadrilaterals) we can  now prove that any sequence of staircase moves produce not only the same sequence of saddle connections,  but also  the same sequence of  wedges and well slanted staircases:

\begin{theorem} \label{thm:same_objects}
Let $X$ be a surface in $\CCC^{hyp}(k)$ without vertical saddle connections and let $Q$ be a quadrangulation of $X$. Let $ ( Q_1^{(n)})_ { n \in \mathbb{N}}$, $ ( Q_2^{(n)})_{ n \in \mathbb{N}}$  be any two sequences of quadrangulations of the surface $X$ such that $Q^{(0)}_1= Q^{(0)}_2= Q$ and, for $i=1,2$, $Q_i^{(n+1)}$ is a new quadrangulation obtained from $Q_i^{(n)}$ by simultaneous staircase moves.
\begin{itemize}
  \item[(i)]  The collection of the wedges of the quadrangulations in the sequence $( Q_1^{(n)})_ { n \in \mathbb{N}}$ is the same as a set than the collection of the  wedges of the quadrangulations in the sequence $(Q_2^{(n)})_{ n \in \mathbb{N}}$.
  \item[(ii)] The set of well slanted staircases associated to the quadrangulations in $ ( Q_1^{(n)})_ { n \in \mathbb{N}}$ is the  the same than the set of well slanted staircases associated to the quadrangulations in $ ( Q_2^{(n)})_ { n \in \mathbb{N}}$.%
\end{itemize}
\end{theorem}



\begin{proof}
Let $ (Q_1^{(n)})_ { n \in \mathbb{N}}$, $ (Q_2^{(n)})_{ n \in \mathbb{N}}$  be as in the statement.
Because of Theorem~\ref{lem:Keane}, each diagonal in $Q_1^{(n)}$ will eventually become a side.
This is also true for $Q_2^{(n)}$.
By Corollary~\ref{cor:wedges_are_best_approx2}, the set of diagonals in $(Q_1^{(n)})_n$ and $(Q_2^{(n)})_n$ coincide.
Now, by Lemma~\ref{lem:d_to_q}, each diagonal uniquely determines its quadrilateral.
It follows that the set of quadrilaterals and the set of wedges in $(Q_1^{(n)})_n$ and $(Q_2^{(n)})_n$ are the same, thus concluding the proof of $(i)$.  

Let us now prove $(ii)$. Since we just showed that quadrilaterals for $(Q_1^{(n)})_n$ and $(Q_2^{(n)})_n$ are the same, it is enough to show that each such quadrilateral uniquely determines the well slanted staicase to which it belongs.
Let $q = q_i$ be a right slanted quadrilateral in $Q_1^{(n)}$ for some $n \in \NN$ (the case when $q$ is left slanted is similar) and let $\sc$ its right top side. We only need to prove that there is a unique quadrilateral $q'$ which is right slanted and has $\sc$ as it bottom left side, since such quadrilateral is necessarily a neighbour of $q$ in a well slanted right staircase.
From Theorem~\ref{thm:wedges_are_best_approx}, we know that $\sc$ is a geometric best approximation and from Lemma~\ref{lem:d_to_q} existence and uniqueness is guaranteed. Repeating this argument, we see that the right well slanted staircase which contains $q$ is uniquely determined.
\end{proof}

\subsubsection{Systoles and Lagrange values along Teichmueller geodesics}\label{sec:systolesproof}
Recall from~\S\ref{subsubsec:intro_Teich} that the systole on a translation surface is the length of the shortest saddle connection.
In this section we prove Theorem~\ref{thm:flat_systoles} on systoles along Teichmueller geodesics and then state and prove Theorem~\ref{thm:lagrange} which shows that diagonal changes can be used to compute the quantity $a(X)$ along closed Teichm\"uller geodesics.

The following general Lemma holds for any translation surface (not necessarily in a hyperelliptic component).
\begin{lemma}\label{lem:systoles_are_ba}
Let $X$ be a translation surface and let $\sc$ be a saddle connection on $X$ which realizes the systole for some time $t$ along the Teichmueller geodesics $(g_t X)_{t \in \RR}$. Then $\sc$ is a geometric best approximation.
\end{lemma}

\begin{proof} Let  $\sc \in \Quad_i$. 
Let us prove the first part. Assume that $\sc$ on $X$ realize the systole for some time $t>0$. Since the property of being a best approximations is  invariant under the geodesic flow $(g_t)_{t \in \mathbb{R}}$ (since immersed rectangles with horizontal and vertical sides are mapped to immersed  rectangles of the same form), we can replace $X$ by $g_{-t}X$ and assume that $t=0$. Thus, for any saddle connection $\ssc$ in $X$, the flat lenght $|\ssc|$ of $\ssc$ is greater or equal than $|\sc|$. In particular, the semicircle in $\CC$ centered in the origin and of radius $|\sc|$ does not contain any point of $\Quad_i$ in its interior. This implies in particular that the rectangle in $\CC$ which has $\sc$ as diagonal and vertical and horizontal sides does not contain any point of $\Quad_i$ in its interior and hence that $\sc$ is a geometric best approximation.
\end{proof}

\begin{proof}[Proof of Theorem \ref{thm:flat_systoles}]
The Theorem now follows immediately as a corollary of Theorem~\ref{thm:wedges_are_best_approx2} and Lemma~\ref{lem:systoles_are_ba} above: 
let $X$ and $Q$ be as in the assumptions. Assume that $\sc$ realizes the systole for some time $t_0$. Then by the first part of Lemma~\ref{lem:systoles_are_ba}, $\sc$ is a best approximation. Thus, by Theorem \ref{thm:wedges_are_best_approx} it appears as one of the wedges.
\end{proof}

If one is interested only in saddle connections which realize the systoles along a \emph{Teichmueller geodesic ray} $(g_t  X)_{t\geq 0}$ starting from $X$, one needs an extra assumption to avoid missing saddle connections which might realize minima for small values of $t$. The following result can be deduced from Theorem~\ref{thm:flat_systoles}.

\begin{corollary}\label{cor:flat_systoles}
Let $X$ be a surface in $\CCC^{hyp}(k)$ with neither horizontal nor vertical saddle connections. Let $Q = (\vpi,\vwedges)$ be a quadrangulation of $X$ for which each side $\sc$ satisfies $|\Re(\sc)| > \sys(X)$.
Let $\{Q^{(n)}, n \in \mathbb{N}\}$ be any sequence of quadrangulations obtained from $Q^{(0)}=Q$ by simultaneous staircase moves. Then the saddle connections on $X$ which realize the systole along the Teichmueller geodesic ray $(g_t  X)_{t\geq 0}$ are a subset of the sides of the quadrangulations in $\{ Q^{(n)}, n \in \mathbb{N}\}$.
\end{corollary}

\begin{proof} Since by assumption each side $\sc$ of $Q$ satisfies $|\Re(\sc)| > \sys(X)$, 
From Part $(i)$ of Corollary~\ref{cor:wedges_are_best_approx2} we know that the set of sides of $(Q^{(n)})_n$ contains all best approximations $\sc$ that satisfy $|\Re(\sc)| \leq \sys(X)$. Hence, because of Lemma~\ref{lem:systoles_are_ba}, it is enough to show that saddle connections $\sc$ that realize the systole at a positive time satisfies $|\Re(\sc)| \leq \sys(X)$. Now, by definition of the systole we have $\sys(g_t X) \leq e^t \sys(X)$. Thus, if $\sc$ is a saddle connection which realizes a systole at time $t > 0$ we have $|\Re(g_t \sc)| \leq |g_t \sc| = \sys(g_t X) \leq e^t \sys(X)$. As $|\Re(g_t \sc)| = e^t |\Re(\sc)|$, we obtain that $|\Re(\sc)| \leq \sys(X)$.
\end{proof}

\smallskip
We now deduce from Theorem~\ref{thm:wedges_are_best_approx2} that the values of the Lagrange spectrum $\LLL(\CCC^{hyp}(k))$ of a hyperelliptic component (defined in~\S\ref{subsubsec:intro_Teich} of the Introduction) can be computed using staircase moves. Let us recall that the definition of $a(X)$ is given in \eqref{def:aX}. 

\begin{theorem}\label{thm:lagrange}
Let $X$ be a surface in $\CCC^{hyp}(k)$ with neither horizontal nor vertical saddle connections and let $(Q^{(n)})_{ n \in \NN}$ be any sequence of labeled quadrangulations $Q^{(n)}= (\vpi^{(n)}, \vwedges^{(n)})$ of $X$ where $Q^{(n+1)}$ is obtained from $Q^{(n)}$ by simultaneous staircase moves. Then
\begin{equation}
a(X) = \liminf_{n\to +\infty} a (\vwedges^{(n)}), 
\quad \text{where} \quad a (\vwedges^{(n)}):= \min_{\sc\, \text{in}\, \vwedges^{(n)}}  |\Re v ||\Im v|,
\end{equation}
where the minimum in the definition of $a(\vwedges^{(n)})$ is taken over all areas of saddle connections belonging to the wedges in $\scw^{(n)}_1, \dots, \scw^{(n)}_k$.
\end{theorem}

\begin{proof} Let us assume for simplicity that $\area(X)=1$. 
  Let us recall   that it is shown in \cite{HubertMarcheseUlcigrai} that the quantity $a(X)$ (which is defined in~\eqref{def:aX} in the introduction) is also equal to $s^2(X)/2$ where $s(X) = \liminf_{t \to \infty} \sys(g_t X) $. Set $Q = Q^{(0)}$ and let $\lambda_{min}(Q) = \min_{\sc\, \text{in}\, w^{(0)}} |\Re(\sc)|$ where as in the definition of $a(\vwedges^{(n)})$ the minimum is taken over all saddle connections that belongs to the wedges of $Q$. Consider a time $t_0 > 0$ such that $\sys(g_{t_0} X) < e^{t_0} \lambda_{min}(Q) = \lambda_{min}(g_{t_0} Q)$ (such time exists since the systole function is bounded from above on each stratum of translation surfaces of unit area). Because of Corollary~\ref{cor:flat_systoles}, the saddle connections in  the wedges $\{\vwedges^{(n)}; n \in \NN\}$ contain all saddle connections that realize the systoles at time larger than $t_0$.  
  
  Let $(t_k)_{k \in \NN}$ be the sequence of times when the systole function has a local minimum for $t\geq t_0$ and let $\sc_k$ be a saddle connection in the wedges $\vwedges^{(n_k)}$ that realizes the systole, that is such that $|g_{t_k} \sc_k| = \sys (g_{t_k} X)$. Since $t_k$ is a local minimum of $t \mapsto |g_t \sc_k|$, it follows that $g_{t_k} \sc_k$ is the diagonal of a square, so $\sys (g_{t_k} X) = \sqrt{2} \Im \sc_k =\sqrt{2} |\Re \sc_k| $ and  $a(\vwedges^{(n_k)}) =\Im \sc_k |\Re \sc_k| = (\sys (g_{t_k} X)/\sqrt{2})^2$. Thus, 
since the liminf of a sequence is invariant under reordering (more precisely if $\pi: \NN \rightarrow \NN$ is a bijection and $(u_n)_{n \in \NN}$ is a sequence of real numbers then $\liminf u_n = \liminf u_{\pi(n)}$), 
$$a(X) = \frac{ (\liminf_{t \to \infty} \sys(g_t X))^2}{2} = \frac{ (\liminf_{k \to \infty} \sys(g_{t_k} X))^2}{2} =  \liminf_{k \to \infty } a(\vwedges^{(n_k)}) \geq \liminf_{n \to \infty } a(\vwedges^{(n)}).  $$
The opposite inequality, that is $a(X)  \leq \liminf_{n \to \infty } a(\vwedges^{(n)})$,  is obvious from the definition~\eqref{def:aX} of $a(X)$ and the invariance of liminf under reordering, since saddle connections belonging ot the wedges $\vwedges^{(n)}$ are a subset of all saddle connections with positive imaginary parts.
\end{proof}

\subsection{Description of the language via staircase moves}\label{sec:language}
In this section we prove that diagonal changes allow to effectively construct the list of bispecial words in the language of cutting sequences.
We first show in~\S\ref{sec:bispecial_cuttseq} that there is a correspondence between bispecial words and geometric best approximations (see Lemma~\ref{lem:bispecial}). Theorem~\ref{thm:best_approx_and_bispecials} about bispecial words then follows from Theorem~\ref{thm:same_objects} of the preceding section. In~\S\ref{subsubsec:substitutions} we show that cutting sequences of best approximations can be constructed by recursive formulas determined by a sequence of staircase moves (see Theorem~\ref{thm:seqsubstitutions} for the precise statement).

\subsubsection{Bispecial words as cutting sequences of best approximations}\label{sec:bispecial_cuttseq}
Except in the proof of Theorem~\ref{thm:best_approx_and_bispecials}, we consider in this section general translation surfaces, i.e. we do not assume that they belong to an hyperelliptic component $\CCC^{hyp}(k)$.

Given a labeled quadrangulation $Q = (\vpi, \vwedges)$ of a translation surface $X$, recall that $\LLL_Q$ denotes the language of cutting sequences of trajectories of the vertical flow on $X$ (see~\S\ref{subsubsec:intro_language}). The alphabet of $\LLL_Q$ is $\AAA = \{1,\ldots,k\} \times \{\ell,r\}$ where $(i,\ell)$ and $(i,r)$ are respectively the labels of the saddle connections $\scw_{i,\ell}$ and $\scw_{i,r}$ of the wedge $\scw_i$ in $Q$.

\begin{lemma}\label{lem:bispecial}
Let $X$ be a translation surface with total angle $k$ and without vertical nor horizontal saddle connections. Let $Q=(\vpi, \vwedges)$ be a labeled quadrangulation of $X$. 
A word $\word= \letter_1 \dots \letter_n $ in $\LLL_Q$ is bispecial if and only if it is the cutting sequence of a geometric best approximation $\sc$ in a bundle $\Quad_{i}$ with $\Im \sc \geq \Im \scw_{i,d}$. Furthermore, if $\word$ is a not empty word in $\LLL_q$:
\begin{itemize}
\item[(i)] If $\word$ is a left special word, then its left extensions are $(i,\ell)$ and $(i,r)$ for some $i \in \{1,\ldots,k\}$. 
\item[(ii)] If $\word$ is right special, then its right extensions are $(\pi_r(j), \ell)$ and $(\pi_\ell(j), r)$ for some $j \in \{1,\ldots,k\}$.
\item[(iii)] If $\word$ is bispecial and its left and right extensions are respectively $(i,\ell)$, $(i,r)$ and $(\pi_r(j),\ell)$, $(\pi_\ell(j),r)$ then the words $(i,\ell)\, \word\, (\pi_\ell(j),r)$ and $(i,r)\, \word\, (\pi_r(j),\ell)$ are in $\LLL_Q$ and exactly one of $(i,\ell)\, \word\, (\pi_r(j),\ell)$ or $(i,r)\, \word\, (\pi_\ell(j),r)$ is in $\LLL_Q$.
\end{itemize}
\end{lemma}
We remark that properties \textit{(i)}, \textit{(ii)}and \textit{(iii)} of the above lemma constitutes the characterization of the language that comes from interval exchange transformations (see~\cite{BelovChernyatev} and~\cite{FerencziZamboni-language}).

In the proof of Theorem~\ref{thm:best_approx_and_bispecials}, given word in the language we want to associate to it a set of orbits of the vertical flow that have that word as a cutting sequence. The following definition is convenient to pass from combinatorics to geometry. 
\begin{definition}\label{def:beam}
  Let $Q$ be a quadrangulation of a translation surface $X$ with no vertical saddle connections. 
  Let $w \in \LLL_Q$ be a non-empty word. We define the \emph{beam} or \emph{cylinder} ${[}\word{]}$ associated to $\word$ as the set of finite orbits of the vertical flow whose coding is exactly $\word$ and are maximal with respect to that property.
\end{definition}
The following Lemma describes the geometric shape of a beam. Examples of beams are shown in Figure~\ref{fig:beams}.
\begin{figure}[!ht]
\begin{center}
  \subfigure[beam of a letter \label{subfig:beam_letter}]{\picinput{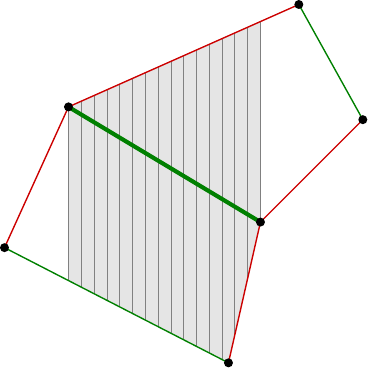}}
\hspace{2cm} 
\subfigure[beam of a word \label{subfig:beam_word}]{\picinput{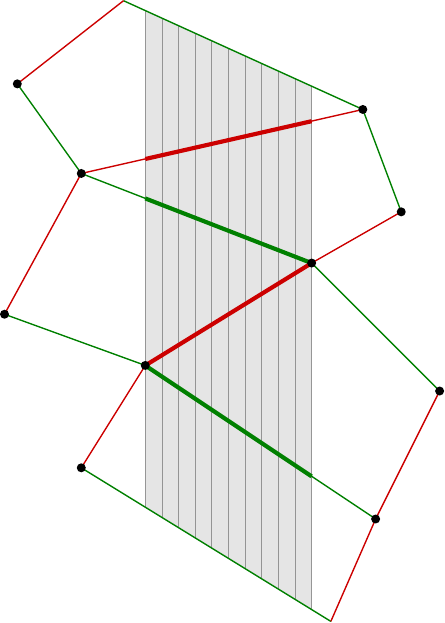}}
\caption{examples of beams of trajectories illustrating Definition~\ref{def:beam}}
\label{fig:beams}
\end{center}
\end{figure}
 
\begin{lemma} \label{lem:beam}
Let $Q$ be a quadrangulation of a translation surface $X$ with no vertical saddle connections. Let $\word$ be a non-empty word in $\LLL_Q$. Then the beam ${[}\word{]}$ is an immersed polygon delimited on the left and the right by vertical separatrices.
The bottom side is delimited either by one side of $Q$ or by a pair of sides $w_{i,\ell}$ and $w_{i,r}$ for some $1 \leq i \leq k$.
The top side is delimited by one side of $Q$ or by a pair of sides $w_{\pi_\ell(j),r}$ and $w_{\pi_r(j),\ell}$ for some $1 \leq j \leq k$.
\end{lemma}

\begin{proof}[Proof of Lemma~\ref{lem:beam}]
We will denote by $\scw({\letter_k})$ the saddle connection in a wedge corresponding to the label $\letter_k \in \AAA$, that is $\scw(\letter_k)=\scw_{i, \ell}$ if $\letter_k = (i, \ell)$ or $\scw(\letter_k)=\scw_{i, r}$ if $\letter_k = (i, r)$.
Let $\word = \letter_1 \ldots \letter_n$ be a non empty word in $\LLL_Q$ and let $i$ and $j$ be respectively such that $\scw(\letter_1)$ is a top side of $q_i$ and $\scw(\letter_n)$ is a bottom side of $q_j$. Let us first remark that, by definition of a quadrilateral, if $x$ is a point on $\scw(\letter_n)$ then the first saddle connection crossed by the forward orbit $(\phi_t(x))_{t > 0}$ is either $\scw_{\pi_\ell(j),r}$ or $\scw_{\pi_r(j),\ell}$. Similarly, for $x$ on $\scw(\letter_1)$ the first saddle connection crossed by the backward orbit $(\phi_t(x))_{t < 0}$ is either $\scw_{i,\ell}$ or $\scw_{i,r}$. Furthermore, the sets of points on $\scw(\letter_n)$ that first hit backward or forward a given side is a connected subsegment of $\scw(\letter_n)$. 

We now proceed by induction on the length $n$ of the word $\word$. 
For a word $\word = \letter$ of length $1$, one can see from the previous remark that the beam is a polygon such that two of its vertices are the endpoints of the associated saddle connection, as shown in Figure~\ref{subfig:beam_letter}.
\begin{figure}[!ht]
\begin{center}
\subfigure[\label{subfig:beam2}]{\picinput{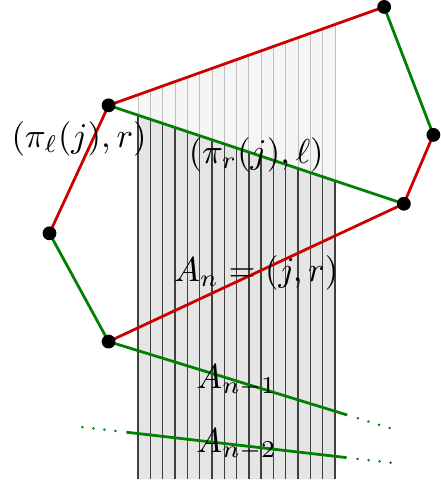}}
\hspace{.5cm}
\subfigure[\label{subfig:beam3}]{\picinput{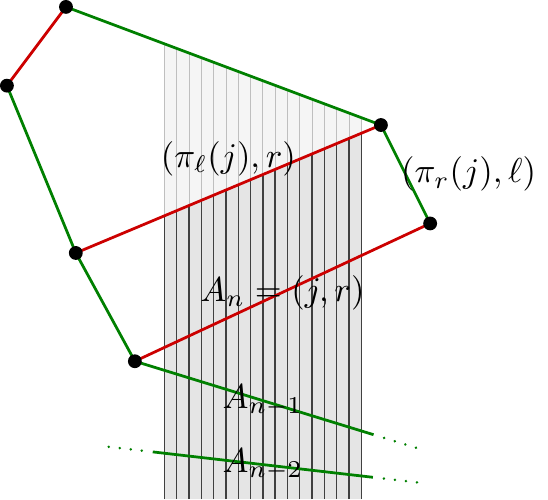}}
\hspace{.5cm} 
\subfigure[\label{subfig:beam4}]{\picinput{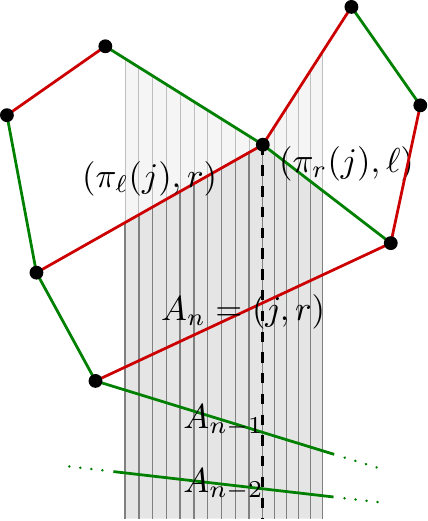}}
\caption{possible splitting of the beam of trajectories ${[}w{]}$ for $w = A_1 \dots A_n \in \LLL_Q$}
\label{fig:beam_splittings}
\end{center}
\end{figure} 

For the inductive step, refer to Figure~\ref{fig:beam_splittings}.
Assume that the result holds for all words of length $n$ and consider a word $A_1 \ldots A_{n+1}$ of length $n+1$.
As before, let $j$ be such that $\letter_n$ is a bottom side of $q_j$ and $j'$ be such that $\scw(\letter_{n+1})$ is a bottom side of $q_{j'}$.
For each orbit in $[\word]$ consider the intersection with the wedge $\scw(\letter_n)$ that corresponds to the $n$-th crossing of the sides of $Q$.
By induction hypothesis, this set of points is a connected segment $J$ in $\scw(\letter_n)$.  
By the initial remark, 
we know that the vertical trajectories emanating from $J$ either
\begin{enumerate}
  \item all cross the wedge $\scw(\letter_{n+1})$ for $\letter_{n+1}=(\pi_r(j), \ell)$ as in Figure~\ref{subfig:beam2},
  \item or all cross $\letter_{n+1}=(\pi_{\ell}(j), r)$) as in Figure~\ref{subfig:beam3},
  \item or one of them hit the conical singuarity which is the top vertex of the quadrilateral $q_j$ as in Figure~\ref{subfig:beam4}.
\end{enumerate}
In the two first cases, the beam ${[}\letter_1 \cdots \letter_{n}\letter_{n+1}{]}$ is obtained simply prolonging the trajectories of the beam ${[}\letter_1 \cdots \letter_{n}{]}$ until, after crossing  $\scw(\letter_{n+1})$, they hit the top side of $q_{j'}$. In the third case, the trajectories are split into two connected subsets of trajectories, accordingly to whether after $\scw(\letter_n)$ trajectories cross $\scw_{\pi_r(j),\ell}$ or $\scw_{\pi_\ell(j),r}$. In all cases, it is clear from the construction that the beam ${[}\letter_1 \cdots \letter_{n+1}{]}$ is again an immersed polygon with the same properties.
\end{proof}

\begin{proof}[Proof of Lemma~\ref{lem:bispecial}]
From Lemma~\ref{lem:beam} the possible right extensions of a non-empty word $\word = A_1 \ldots A_n$ in $\LLL_Q$ are of the form $(\pi_\ell(j),r)$ and $(\pi_r(j),\ell)$ for the integer $j$ such that $A_n \in \{j\} \times \{\ell,r\}$ (see Figure~\ref{fig:beam_splittings}). Similarly, its left extensions are of the form $(i,\ell)$ and $(i,r)$ for $i$ such that $A_1 \in \{i\} \times \{\ell,r\}$. This proves items \textit{(i)} and \textit{(ii)}.

\smallskip

\begin{figure}[!ht]
\begin{center}
\subfigure[\label{subfig:BA_to_bisp1}]{\picinput{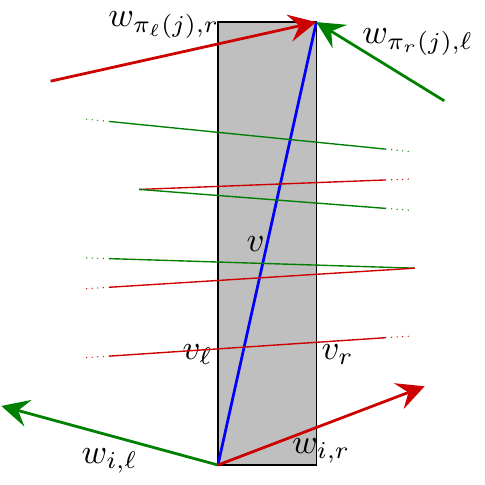}} \hspace{1cm}
\subfigure[\label{subfig:BA_to_bisp2}]{\picinput{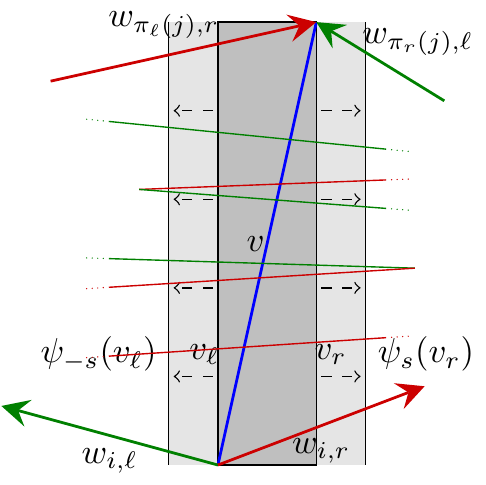}}
\caption{cutting sequences of geometric best approximations are bispecial}
\end{center}
\end{figure} 

We now prove that cutting sequences of best approximations with imaginary parts as in the statement of the Lemma are exactly all bispecial words. 
Let $\sc$ be a geometric best approximation in $\Quad_i$ with $\Im \sc \geq \Im \scw_{i,d}$.
If $\Im \sc = \Im \scw_{i,d}$ then $\sc = \scw_{i,d}$ and the cutting sequence of $\sc$ is the empty word, which is bispecial.
Let us hence assume that $\Im \sc > \Im \scw_{i,d}$ and let $\word = A_1 \ldots A_n$ be the cutting sequence of $\sc$ where $n \geq 1$.
Let $i$ and $j$ be so that $A_1 \in \{i\} \times \{\ell,r\}$ and $A_n \in \{j\} \times \{\ell,r\}$.
Since $\sc$ is a best approximation, by Lemma~\ref{lem:equivalentBA} there exists an immersed rectangle $R(\sc) \subset X$ with no singularity in its interior and no singularity on its sides other than the endpoints of $\sc$. Without loss of generality, we may assume that $\sc$ is right slanted.  Let $v_\ell$ and $v_r$ be respectively the left and right vertical side of $R(\sc)$, as shown in Figure~\ref{subfig:BA_to_bisp1}. 
Since $\Im \sc > \Im \scw_{i,d}> \Im \scw_{i,r}$, the beginning of $\sc$ belongs to the sector determined by the wedge $(\scw_{i,\ell}, \scw_{i,r})$ and $ \scw_{i,r}$ crosses the vertical side $v_r$, as shown in Figure~\ref{subfig:BA_to_bisp1}. We now claim that  $\scw_{\pi_r(j),\ell}$ crosses the other vertical side, that is $v_{\ell}$. Indeed, since $\scw_{\pi_r(j),\ell}$ is left slanted and 
$R(\sc)$ cannot its starting point in its interior, either $\scw_{\pi_r(j),\ell}$ crosses $v_{\ell}$ or it crosses the bottom side of $R(\sc)$. The latter  possibility cannot happen since otherwise $\scw_{\pi_r(j),\ell}$ would have to intersect $\scw_{i,r}$, which is excluded from the definition of quadrangulations. 

Let us call vertical trajectory in $R(\sc)$ any finite trajectory which is obtained intersecting a vertical trajectory with $R(\sc)$.  
It follows from what we proved  that the first side of $Q$ hit by any vertical trajectory in $R(\sc)$  is $\wedge_{i,r}$, while $\scw_{\pi_r(j),\ell}$ is the last side of $Q$ hit, see Figure~\ref{subfig:BA_to_bisp2}. We claim that in between these two hitting times the  cutting sequence of the vertical trajectory in $R(\sc)$ is the same than the cutting sequence $W$ of $\sc$. Indeed, since $R(\sc)$ does not contain singularities and sides of $Q$ cannot cross  neither $\wedge_{i,r}$ nor  $\scw_{\pi_r(j),\ell}$, they have to cross both $v_{\ell}$ and $v_r$.  
Thus, any vertical segment in $R(\sc)$ has cutting sequence $(i,r)\, \word\, (\pi_\ell(j),r)$. Now flow horizontally $\sc_\ell$ to the left and to $\sc_r$ to the right. If $\psi_t$ denotes the horizontal flow, for  any $t>0$ such $\psi_s(v_r)$ does not contain any singularity for $0< s \leq t$, the vertical trajectory $\psi_t(v_r)$ has coding $(i,r)\, \word\, (\pi_r(j),\ell)$ (see Figure~\ref{subfig:BA_to_bisp2}). Similarly,  for  any $t<0$ such $\psi_{s}(v_\ell)$ does not contain any singularity for $-t\leq s < 0$, the vertical trajectory $\psi_t(v_\ell)$ has coding $(i,\ell)\, \word\, (\pi_\ell(j),r)$. This shows that $\word$ is bispecial.

\begin{figure}[!ht]
\begin{center}
\subfigure[\label{subfig:bisp_to_BA1}]{\picinput{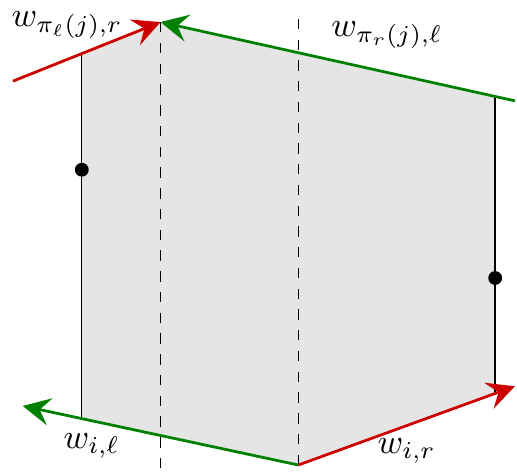}}
\hspace{1cm}
\subfigure[\label{subfig:bisp_to_BA2}]{\picinput{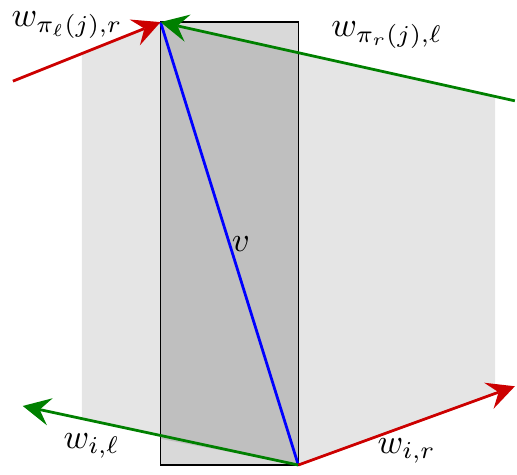}}
\caption{bispecial words are cutting sequences of geometric best approximations}
\label{fig:bispecial_beams}
\end{center}
\end{figure} 

\smallskip
Let us now assume that $\word = \letter_1 \ldots \letter_n$ is a non-empty bispecial word.
We know from Lemma~\ref{lem:beam} the that letters that may be append to its left are $(i,\ell)$ and $(i,r)$ where $i$ is such that $A_1 \in \{i\} \times \{\ell,r\}$. Similarly the letters that may be append to its right are $(\pi_\ell(j),r)$ and $(\pi_r(j),\ell)$ where $A_n \in \{j\} \times \{\ell,r\}$. From the Lemma~\ref{lem:beam} the beam $[\word]$ is an immersed polygon whose sides are either vertical or part of the sides of $Q$. Because $\word$ is bispecial, both the top and bottom of $Q$ consists of two sides and in particular they contain the top singularity of $q_j$ and the bottom singularity of $q_i$ respectively. Consider the saddle connection $\sc$ which connects these two singularities and let us assume without loss of generality that it is left slanted (as in Figure~\ref{subfig:bisp_to_BA1}).  
Let us show that it is a best approximation by constructing an immersed rectangle that has $\sc$ as its diagonal. Consider the vertical trajectory $v_\ell$ in the beam that hits the top singularity of $q_j$ and the vertical trajectory $v_r$ in the beam emanating from the bottom singularity of $q_i$. 
Let us consider the quadrilateral $P$ built from the beam $[\word]$ by cutting its left and right parts up $v_\ell$ and $v_r$,  see the dark quadrilateral in Figure~\ref{subfig:bisp_to_BA1}). Then flow 
 vertically forward each point on the top sides and backward each point on the bottom sides until they first hit an horizontal trajectory. Extending $P$ by these trajectories, we obtain a rectangle $R$ which contains $P$, as shown in Figure~\ref{subfig:bisp_to_BA2}. By construction $R$ is a rectangle which has $\sc$ as a diagonal and it does not contain singularities in its interior (since $P$ is contained in the interior of the beam and when extending top and bottom sides one hits a horizontal before hitting a singularity by definition of quadrangulation).  This shows that $\sc$ is a best approximation and, arguing as in the previous part of the proof, it also follows that $\sc$ has  cutting sequence $\word$.

\end{proof}

Exploting  Lemma~\ref{lem:bispecial}, we can now deduce Theorem~\ref{thm:best_approx_and_bispecials} from Theorem~\ref{thm:wedges_are_best_approx}. 
\begin{proof}[Proof of Theorem~\ref{thm:best_approx_and_bispecials}]
Let $X$ be a surface in $\CCC^{hyp}(k)$ with no vertical saddle connections. Let $Q$ be a quadrangulation of $X$ and let $\LLL_Q$ be the associated language. Let $( Q^{(n)})_{ n \in \mathbb{N}}$ be a sequence of quadrangulations $Q^{(n)}= (\vpi^{(n)}, \vwedges^{(n)})$ obtained starting from $Q$ by simultaneous staircase moves.
Then, by Lemma~\ref{lem:bispecial}, the set of bispecial words coincide with the set of geometric best approximations $\sc$ in some $\Quad_i$ such that $\Im \sc \geq \scw_{i,d}$. By Corollary \ref{cor:wedges_are_best_approx2}, these are exactly the diagonals in $(Q^{(n)})_n$.
\end{proof}

\subsubsection{Cutting sequences  by staircase moves}\label{subsubsec:substitutions}
In this section we show how to produce all cutting sequences of best approximations from the sequence of staircase moves, see Theorem~\ref{thm:seqsubstitutions}. The key step is Lemma~\ref{lem:substitution} which describe the combinatorial operation that allows to deduce the cutting sequence of a diagonal of an admissible quadrilateral obtained by staircase moves from the cutting sequences of its sides.

\begin{theorem}\label{thm:seqsubstitutions}
Let $X \in \CCC^{hyp}(k)$ be a translation surface with no vertical saddle connections. Let $Q$ be a quadrangulation of $X$ and let $\LLL_Q$ be the associated language of cutting sequences.  
Let $\{Q^{(n)}\}_{n \in \mathbb{N}}$ be {any} sequence of labeled quadrangulations $Q^{(n)}=Q(\vpi^{(n)}, \vwedges^{(n)})$ starting from $Q^{(0)}=Q$ and such that $Q^{(n+1)}$ is obtained from $Q^{(n+1)}$  by performing a staircase move in the staircase $S_{c_n}$ for $Q^{(n)}$ given by a cycle $c_n$ of $\vpi^{(n)}$.  Set 
\begin{equation}\label{base}
D_i^{(0)} = \emptyset, \qquad {L}_i^{(0)} =(\pi_\ell^{-1}(i),\ell), \qquad {R}_i^{(0)} = (\pi_r^{-1}(i),r), \qquad \textrm{for}\ 1\leq i \leq k .
\end{equation}
Let ${L}_i^{(n)}, {R}_i^{(n)}$ and $ D_i^{(n)}$ for $n \geq 1$ be given by the following recursive formulas:
\begin{eqnarray}
{L}^{(n+1)}_i &= &
\left\{ \begin{array}{ll}
{L}^{(n)}_i {R}^{(n)}_{\pi_\ell^{(n)}(i)} & \text{if $i \in c_n$ and $c_n$ is a cycle of  $\pi^{(n)}_r$}, \\
{L}^{(n)}_i & \text{otherwise,}
\end{array} \right.\label{eq:recursiveL}
\\
{R}^{(n+1)}_i &= &
\left\{ \begin{array}{ll}
{R}^{(n)}_i {L}^{(n)}_{\pi_r^{(n)}(i)} & \text{if $i \in c_n$ and $c_n$ is a cycle of  $\pi^{(n)}_\ell$}, \\
{R}^{(n)}_i & \text{otherwise,}
\end{array} \right. \label{eq:recursiveR}
\\
{D}^{(n+1)}_i &= &
\left\{ \begin{array}{ll}
{D}^{(n)}_i {R}^{(n)}_{\pi^{(n)}_l\pi_r^{(n)}(i)} & \text{if $i \in c_n$ and $c_n$ is a cycle of  $\pi_{r}$}, \\
{D}^{(n)}_i {L}^{(n)}_{\pi^{(n)}_r\pi_\ell^{(n)}(i)} & \text{if $i \in c_n$ and $c_n$ is a cycle of  $\pi_\ell$}, \\
{D}^{(n)}_i  & \text{if $i \notin c_n$.} 
\end{array} \right. \label{eq:recursiveD}
\end{eqnarray}
 Then the bispecial words of $\LLL_Q$ are exactly all words which appear in the sequences $(D_i^{(n)})_{n\in\NN}$ for $1\leq i \leq k$. 
\end{theorem}
We will prove Theorem~\ref{thm:seqsubstitutions} from Theorem~\ref{thm:best_approx_and_bispecials} by showing that for any $n \in \NN$ the word $D_i^{(n)}$ given by the recursive formulas in the statement is the cutting sequence of the diagonal $\scw_{i,d}^{(n)}$ for any $1\leq i\leq k$. We remark that an analogous Theorem in the setup of interval exchange transformations is proved by Ferenczi and Zamboni in~\cite{FerencziZamboni-struct}. In their context, the analogous of the words $L_i^{(n)}$ and $R_i^{(n)}$ that are needed to build the bispecial words $D_i^{(n)}$ can be interpreted as cutting sequences of Rohlin towers for the bipartite IETs $(\vpi^{(n)}, \vlambda^{(n)})$ (see~\S\ref{sec:correspondence_Q_IETs}).

Let us first prove a preliminary Lemma that shows how the cutting sequence of a diagonal of quadrilateral in a quadrangulation can be deduced from the cutting sequences of the sides and the combinatorial datum (see also Figure~\ref{fig:cases_lemma}).
\begin{lemma}\label{lem:substitution}
Let $Q=(\vpi, \vwedges)$ be obtained from $Q' = (\vpi',\vwedges')$ by a sequence of staircase moves. 
Let $\word_{i,\ell}, \word_{i,r}$ be respectively the cutting sequences of the saddle connections $\scw_{i,\ell}$ and $\scw_{i,r}$ with respect to the labelling of $Q'$.
Then the cutting sequence $D_i$ of a diagonal $w_{i,d}$ in $Q$ is given by
\begin{equation} \label{eq:D_from_W_right}
D_i =  \left\{
\begin{array}{lll}
  \word_{i,r}\, (j,r)\, (\pi_r(i),\ell)\, W_{\pi_r(i),\ell} &
  \text{if $\scw_{i,r} \not= \scw'_{i,r}$ and $\scw_{\pi_r(i),\ell} \not= \scw'_{\pi_r(i),\ell}$,} &
  \subref{subfig:case_nn} \\
  (\pi_r(i),\ell)\, W_{\pi_r(i),\ell} &
  \text{if $\scw_{i,r} = \scw'_{i,r}$ and $\scw_{\pi_r(i),\ell} \not= \scw'_{\pi_r(i),\ell}$,} &
  \subref{subfig:case_en} \\
  \word_{i,r}\, (j,r) &
  \text{if $\scw_{i,r} \not= \scw'_{i,r}$ and $\scw_{\pi_r(i),\ell} = \scw'_{\pi_r(i),\ell}$,} &
  \subref{subfig:case_ne} \\
  \emptyset &
  \text{if $\scw_{i,r} = \scw'_{i,r}$ and $\scw_{\pi_r(i),\ell} = \scw'_{\pi_r(i),\ell}$,} &
  \subref{subfig:case_ee}  
\end{array}\right.
\end{equation}
where $j = (\pi_r')^{-1} \pi_r (i)$.
Similarly, we have
\begin{equation} \label{eq:D_from_W_left}
D_i = \left\{
\begin{array}{ll}
  \emptyset &
  \text{if $\scw_{i,\ell} = \scw'_{i,\ell}$ and $\scw_{\pi_\ell(i),r} = \scw'_{\pi_\ell(i),r}$,} \\
  (\pi_\ell(i),r)\, W_{\pi_\ell(i),r} &
  \text{if $\scw_{i,\ell} = \scw'_{i,\ell}$ and $\scw_{\pi_\ell(i),r} \not= \scw'_{\pi_\ell(i),r}$,} \\
  \word_{i,\ell}\, (j,\ell) &
  \text{if $\scw_{i,\ell} \not= \scw'_{i,\ell}$ and $\scw_{\pi_\ell(i),r} = \scw'_{\pi_\ell(i),r}$,} \\
  \word_{i,\ell}\, (j,r)\, (\pi_\ell(i),r)\, W_{\pi_\ell(i),r} &
  \text{if $\scw_{i,\ell} \not= \scw'_{i,\ell}$ and $\scw_{\pi_\ell(i),r} \not= \scw'_{\pi_r(i),\ell}$,}
\end{array}\right.
\end{equation}
where $j = (\pi'_\ell)^{-1}\, \pi_\ell(i)$.
\end{lemma}

\begin{figure}[!ht]
\begin{center}
\subfigure[\label{subfig:case_nn}]{\picinput{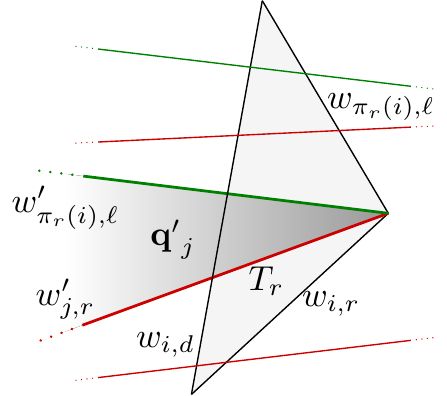}} \hspace{1.9cm}
\subfigure[\label{subfig:case_en}]{\picinput{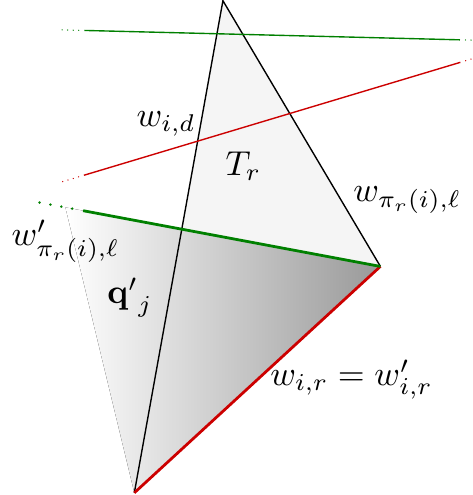}}\\ \vspace{1cm}
\subfigure[\label{subfig:case_ne}]{\picinput{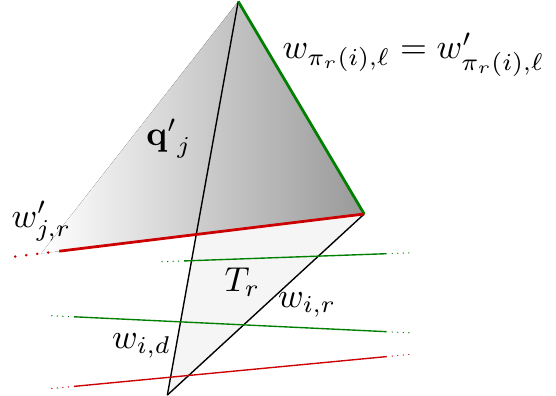}} \hspace{1cm}
\subfigure[\label{subfig:case_ee}]{\picinput{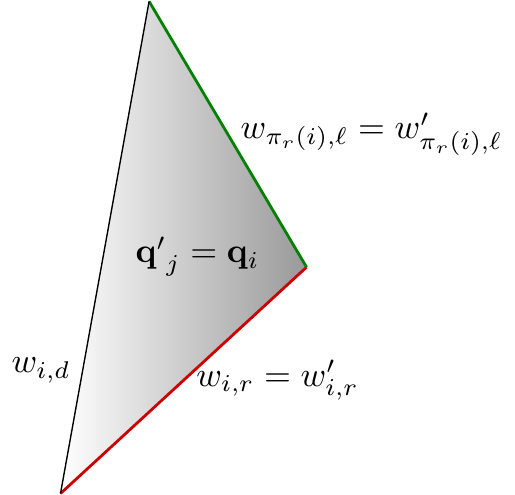}}
\end{center}
\caption{the four cases in the proof of Lemma~\ref{lem:substitution}}
\label{fig:cases_lemma}
\end{figure}

\begin{proof}
We prove only~\eqref{eq:D_from_W_right} as the case of~\eqref{eq:D_from_W_left} is the same after vertical reflection.

Let $p_i$ for $1\leq i \leq k$ be the vertex of the wedge $\scw_i$ of $Q$. Consider the quadrilateral $q_i \in Q$. The diagonal $\scw_{i,d}$ divides $q_i$ in two triangles.
Let us consider the right triangle $T_r$  which has sides $\scw_{i,r}$, $\scw_{i,d}$ and $\scw_{\pi_r(i),\ell}$. Remark that right most vertex of $T$, that is the endpoint of $\scw_{i,r}$, is $p_{\pi_r(i)}$.

Since $q_i$ and hence also $T_r$ does not contain any singularity in its interior, any saddle connection of $Q'$ which crosses the diagonal $\scw_{i,d}$ has to cross either the union of the interior of the two other sides $\scw_{i,r}$ and $\scw_{\pi_{r}(i),\ell}$ of the triangle, or has $p_{\pi_r(i)}$ as an endpoint.
Remark that there at most two saddle connections of $Q'$ which intersect $\scw_{i,d}$ and ends in $p_{\pi_r(i)}$ before leaving $T_r$, namely $\scw'_{\pi_r(i),\ell} = \scw'_{\pi'_r(j),\ell}$ and $\scw'_{j,r}$ where $j=(\pi'_r)^{-1} \pi_r(i)$. The saddle connection $\scw'_{\pi_r(i),\ell}$ crosses $\scw_{i,d}$ if and only if $\scw_{\pi_r(i),\ell} \neq \scw'_{\pi_r(i),\ell}$ (case (b) and (d) in~\eqref{eq:D_from_W_right} and Figure~\ref{fig:cases_lemma}). On the other hand, the saddle connection $\scw'_{j,r}$ crosses $\scw_{i,d}$ if and only if $\scw_{i,r} \neq \scw'_{i,r}$ (case (c) and (d)).

In the case $\scw_{\pi_r(i),\ell} \neq \scw'_{\pi_r(i),\ell}$ and $\scw'_{i,r} \not= \scw_{i,r}$ (see Figure~\ref{subfig:case_nn}) the cutting sequence of $w_{i,d}$ is obtained by concatenation of the one of $w_{i,r}$, the two letters $(j,r)$ and $(\pi_r(j),\ell)$ and then the cutting sequence of $w_{\pi_r(i),\ell}$. The other three cases, which are somewhat degenerate, are obtained similarly, referring to Figures~\ref{subfig:case_en},~\ref{subfig:case_ne} and~\ref{subfig:case_ee}.
\end{proof}

Recall that diagonal change consists in replacing one side of a wedge by the diagonal. Hence Lemma~\ref{lem:substitution} already provide a way to obtain recursively the cutting sequences of sides and diagonals. In order to simplify notations and gather all four cases Theorem~\ref{thm:seqsubstitutions}, we defined words $L_i$ and $R_i$. These words are \emph{extended} cutting sequences of sides, that is cutting sequences preceded by the labels of some of the incoming edges in the starting vertex. Keeping the same notation as in the Lemma, let us define $L_i$ and $R_i$ from the cutting sequence of the sides by
\begin{eqnarray}
\label{eq:L_def}
L_i &= & \left\{\begin{array}{ll}
((\pi'_r)^{-1}(i), r) & \text{if $\scw_{i,\ell} = \scw'_{i,\ell}$}, \\
((\pi'_r)^{-1}(i) , r) \, ( i , \ell)\, \word_{i ,\ell} & \text{if $\scw_{i,\ell} \neq \scw'_{i,\ell}$}. \\
\end{array} \right.\\ 
\label{eq:R_def}
R_i &=& \left\{ \begin{array}{ll}
((\pi'_\ell)^{-1}(i), \ell) &  \text{if $\scw_{i,r} = \scw'_{i,r}$}, \\
((\pi'_\ell)^{-1}(i), \ell)\, (i, r)\, \word_{i,r} &   \text{if $\scw_{i,r} \neq \scw'_{i,r}$}, \\
\end{array} \right.
\end{eqnarray}

\begin{proof}[Proof of Theorem \ref{thm:seqsubstitutions}]
Let us first show by induction on $n$ that the words $L_i^{(n)}$, $R_i^{(n)}$ given by the recursive formulas \eqref{eq:recursiveL} and \eqref{eq:recursiveR} in the statement are respectively the words defined from cutting sequence of sides by~\eqref{eq:L_def} and~\eqref{eq:R_def}.

For $n=0$ the definitions in~\eqref{base} also coincide with the definitions given by~\eqref{eq:L_def} and \eqref{eq:R_def}.
Let us fix $n \in \NN$ and assume that for all $1 \leq i \leq k$ the words $L_i^{(n)},R_i^{(n)}$ given by~\eqref{eq:L_def} and \eqref{eq:R_def} satisfy the recursive formulas in the statement and let us show that then the same is also true for $n+1$.
Let us assume that $Q^{(n+1)}$ is obtained from $Q^{(n)}$ by a left staircase move in $S_{c_n}$ (i.e. $c_n$ is a cycle of $\pi^{(n)}_\ell$).

By definition of a left move, $\scw_{i,\ell}^{(n+1)}= \scw_{i,\ell}^{(n)}$ (and hence $\word_{i,r}^{(n+1)}= \word_{i,r}^{(n)}$) for every $1 \leq i \leq k$ and $\scw_{i,r}^{(n+1)}= \scw_{i,r}^{(n)}$ (and hence $\word_{i,r}^{(n+1)}= \word_{i,r}^{(n)}$) unless $i \in c_n$.
Thus, from \eqref{eq:L_def} and \eqref{eq:R_def} we obtain that $L_i^{(n+1)} = L_i^{(n)}$ for all $1\leq i \leq k$ and $R_i^{(n+1)}= R_i^{(n)}$ for all $i \notin c_n$.  

Consider now $i \in c_n$. In that case $\scw_{i,r}^{(n+1)}= \scw_{i,d}^{(n)}$.
We will consider four possible cases that correspond to the four cases in Lemma~\ref{lem:substitution} and Figure~\ref{fig:cases_lemma}. Case (a) is the only case that happens for any $n$ sufficiently large. Cases (b), (c) and (d) only happen for initial steps of the induction and should be treated separately.
In Lemma~\ref{lem:substitution} we set $Q'=Q^{(0)}$ and $Q=Q^{(n)}$ and with this notation in mind one can refer to Lemma~\ref{lem:substitution} and Figure~\ref{fig:cases_lemma}. Using the same notation introduced in  Lemma~\ref{lem:substitution}, we denote $j : = (\pi_r^{(0)})^{-1} \pi_\ell(i)$.

{\bf Case (a):} $\scw^{(n)}_{i,r} \neq \scw^{(0)}_{i,r}$ and $\scw^{(n)}_{\pi_r^{(n)}(i),\ell} \not= \scw^{(0)}_{\pi_r^{(n)}(i),\ell}$.\\
We first apply Lemma~\ref{lem:substitution} to $\scw_{i,r}^{(n+1)}= \scw_{i,d}^{(n)}$ and get 
\[
\word^{(n+1)}_{i,r} =  D^{(n)}_i  =  \word_{i,r}^{(n)} (j,r) (\pi_r^{(0)}(i),\ell) W_{\pi_r(i)^{(n)},\ell} = \word_{i,r}^{(n)} L^{(n)}_{\pi^{(n)}_r(i) }
\]
Now, using~\eqref{eq:R_def} for $R_i^{(n)}$ and $R_i^{(n+1)}$ we get
$$
R^{(n+1)}_{i}
= ((\pi_r^{(0)})^{-1} (i) , \ell) \, (i , r)\, \word^{(n+1)}_{i,r}
= ((\pi_r^{(0)})^{-1} i , \ell)\, ( i , r)\, \word_{i,r}^{(n)}\, L_{\pi^{(n)}_r(i) }
= R^{(n)}_{i}\, L^{(n)}_{\pi_r^{(n)}(i) }.
$$

\smallskip
{\bf Case (b):} $\scw_{i,r}^{(n)} = \scw^{(0)}_{i,r}$ and $\scw_{\pi_r^{(n)}(i),\ell}^{(n)} \neq \scw^{(0)}_{\pi_r^{(n)}(i),\ell}$. \\
From Lemma~\ref{lem:substitution} we get that $\word^{(n+1)}_{i,r} = D^{(n)}_i = (\pi_r^{(n)}(i),r) \word_{\pi_r^{(n)}(i), \ell}^{(n)}$ and from~\eqref{eq:R_def} we obtain
\[
R^{(n+1)}_{i}
= ((\pi^{(0)}_r)^{-1}(i),\ell)\, (i,r)\, \word^{(n+1)}_{i,r}
= ((\pi^{(0)}_r)^{-1}(i),\ell)\, (i,r)\, (\pi_r^{(n)}(i),r) \word_{\pi_r^{(n)}(i),\ell}^{(n)}
= R^{(n)}_{i}\, L^{(n)}_{\pi^{(n)}_r(i)}.
\]

\smallskip
{\bf Case (c):} $\scw_{i,r}^{(n)} \neq \scw^{(0)}_{i,r}$ and $\scw_{\pi_r^{(n)}(i),\ell}^{(n)} = \scw^{(0)}_{\pi_r^{(n)}(i),\ell}$. \\
From Lemma~\ref{lem:substitution} we get that $\word^{(n+1)}_{i,r} = D^{(n)}_i = W_{i,r}^{(n)} (j,r)$ and from~\eqref{eq:R_def} we obtain
\[
R^{(n+1)}_{i}
= ((\pi^{(0)}_r)^{-1}(i),\ell)\, (i,r)\, \word^{(n+1)}_{i,r}
= ((\pi^{(0)}_r)^{-1}(i),\ell)\, (i,r)\, \word_{i,r}^{(n)} (j,r)
= R^{(n)}_i L^{(n)}_{\pi_r^{(n)}(i)}.
\]

\smallskip 
{\bf Case (d):} $\scw_{i,r}^{(n)} = \scw^{(0)}_{i,r}$ and $\scw_{\pi_r^{(n)}(i),\ell}^{(n)} = \scw_{\pi_r^{(n)}(i),\ell}^{(0)}$ \\ 
In that case, $q_i$ is a quadrilateral in both $Q^{(0)}$ and $Q^{(n)}$. Hence $L_{\pi^{(n)}_r(i)}^{(n)} = (i,r)$, $R_i^{(n)} = ((\pi_\ell^{(0)})^{-1}(i),\ell)$ and $\word^{(n+1)}_{i,r} = D^{(n)}_i = \emptyset$. We get
\[
R_i^{(n+1)}
= ((\pi^{(0)}_r)^{-1}(i),\ell)\, (i,r)\, W^{(n+1)}_{i,r}
= ((\pi^{(0)}_r)^{-1}(i),\ell)\, (i,r)
= R_i^{(n)}\, L^{(n)}_{\pi_r^{(n)}(i)}.
\]

\smallskip
Hence, in all cases the recursive relation~\eqref{eq:recursiveR} holds for $n+1$.
The case of right staircase move is symmetric in which only the $L_i$ change and proves that~\eqref{eq:recursiveL} holds in that case.
This conclude the proof that $L_i^{(n)}$ and $R_i^{(n)}$ given recursively in~\eqref{eq:recursiveL} and~\eqref{eq:recursiveR} coincide with the definition~\eqref{eq:L_def} and~\eqref{eq:R_def}.

\bigskip

Let us now verify the relations~\eqref{eq:recursiveD} about cutting sequence of diagonals. For $n=0$, The cutting sequences $D^{(0)}_i$ of the diagonals $\scw_{i,d}^{(0)}$ are clearly the empty word for all $1 \leq i \leq k$. Now assume that the relation holds for $n$. We consider as before the case where $i$ has a left diagonal change at step $n$. We apply Lemma~\ref{lem:substitution} to the quadrilaterals $q_i'$ obtained after the move.
Since its diagonal is $\scw_{i,d}^{(n+1)} \neq \scw_{i,d}^{(n)}$, we get
\[ 
D^{(n+1)}_i  =  \word_{i,r}^{(n+1)} R^{(n+1)}_{\pi_r^{(n+1)}(i)}.
\] 
Since by definition of a move $\word_{i,r}^{(n+1)} = D^{(n)}_i$ and $\pi_r^{(n+1)} = \pi_r^{(n)} \pi_\ell^{(n)}$, this shows that the inductive assumption~\eqref{eq:recursiveD} also holds for $n+1$ and conclude the proof.
\end{proof}


\end{document}